\documentclass[11pt]{book}
\usepackage[english]{babel}

\usepackage{epsfig}
\usepackage{amsfonts}
\usepackage{amssymb}
\usepackage{amsmath}
\usepackage{amsthm}
\usepackage{latexsym}
\usepackage{mathrsfs}
\usepackage{enumerate}
\usepackage{array}
\usepackage{xypic}
\usepackage[vcentermath,enableskew]{youngtab}
\usepackage[nottoc]{tocbibind}
\usepackage[margin=0.7in]{geometry}

\newtheorem{thm}{Theorem}[chapter]
\newtheorem{cor}[thm]{Corollary}
\newtheorem{lem}[thm]{Lemma}

\newtheorem{prop}[thm]{Proposition}
\theoremstyle{definition}

\theoremstyle{remark}
\newtheorem{rmk}[thm]{Remark}
\DeclareMathOperator{\trace}{tr}
\DeclareMathOperator{\height}{ht}
\DeclareMathOperator{\aut}{Aut}
\DeclareMathOperator{\res}{Res}
\DeclareMathOperator{\ind}{Ind}
\DeclareMathOperator{\ch}{ch}
\DeclareMathOperator{\sgn}{sgn}
\DeclareMathOperator{\imagp}{Im}
\DeclareMathOperator{\realp}{Re}
\DeclareMathOperator{\var}{Var}
\newcommand{\round}[1]{\left[#1\right]}
\newcommand{\fractional}[1]{\left\{#1\right\}}

\newcommand{\C}{\mathbb{C}}
\newcommand{\R}{\mathbb{R}}
\newcommand{\Z}{\mathbb{Z}}

\newcommand{\Q}{\mathbb{Q}}
\newcommand{\Young}{\mathcal{Y}}
\newcommand{\w}{\mathsf{w}}
\newcommand{\p}{\mathbf{p}}
\newcommand{\s}{\mathbf{s}}
\newcommand{\g}{\mathbf{g}}
\newcommand{\f}{\mathbf{f}}
\newcommand{\pbar}{\overline{\mathbf{p}}}

\newcommand{\sphere}{\mathbb{CP}^1}
\newcommand{\moduliquadratic}{\mathcal{Q}_g}

\newcommand{\pcop}{\mathfrak{W}}

\title{Near-Involutions, the Pillowcase Distribution, and Quadratic Differentials}
\author{Rodolfo Antonio R\'ios Zertuche R\'ios Zertuche}

\date{September 2012}

\begin{document}
\maketitle
\frontmatter

\chapter{Abstract}
In the context of A. Eskin and A. Okounkov's approach to the calculation of the volumes of the different strata of the moduli space of quadratic differentials, two objects have a prominent role. Namely, the characters of near-involutions and the pillowcase weights.  For the former we give a fairly explicit formula. On the other hand, the pillowcase weights induce a distribution on the space of Young diagrams. We analyze this distribution and prove several facts, including that its limit shape corresponds to the one induced by the uniform distribution, that the probability concentrates on the set of partitions with very similar 2-quotients, and that there is no hope for a full Central Limit Theorem.

This is a reformatted version of the author's Ph.D. thesis, advised by Professor Andrei Okounkov. The results will be published in a forthcoming paper.

\chapter{Acknowledgements}

I am deeply indebted to my parents, for their unconditional support throughout every challenge that I have faced in life, and also to Professor Xavier G\'omez-Mont, that rare yet crucial mentor who will go all the way for a student he believes in.

I am very grateful to my adviser, Professor Andrei Okounkov, for his generosity at accepting me as his student, for his valuable guidance, and for sharing his good wisdom with me.

I am grateful to Professor Grigori Olshanski for several enlightening conversations on the topic of my thesis and for reading this work, and to Professor John N. Mather for numerous enlightening conversations on other topics.

I am also grateful to Professor Joseph Kohn, who was my first-year adviser and gave me an important push toward confidence.

I cannot cease to thank Jill LeClair for significantly lightening my life and the life of all the graduate students in the department.

The completion of this project would have been impossible without the continued support of my brothers Diego and Luis, and of many friends, among whom I want to name especially Ren\'e Flores, Anand Murugan, Jes\'us Puente, Jorge Saavedra, and P\'eter Varj\'u.
Of great importance too were the people who prepared me for this endeavor, among whom I want to name especially my teachers Professor Gonzalo Contreras and Professor Renato Iturriaga, and my friends Andr\'es Mart\'inez Arizpe and Jos\'e Luis Mart\'inez Meyer.



I also want to acknowledge the generous support of a \textsc{conacyt}-Mexico fellowship, a scholarship from the Mexican Secretariat of Public Education (\textsc{sep}), and a fellowship from Princeton University.

\vspace {.7cm}
\hspace{7cm}\textsc{r.a.r.z.r.z.}

\hspace{7cm}Princeton, \textsc{n.j.}, May, 2012

\tableofcontents
\mainmatter
\chapter{Introduction}
\textit{This is a reformatted version of the author's dissertation, advised by Professor 
Andrei Okounkov, submitted to the Faculty of Princeton University in candidacy for the 
degree of Doctor of Philosophy and defended on August 29, 2012. The results will be published in a forthcoming paper.}
\section{Overview of results}

This thesis presents progress in the computation of the volumes of the moduli spaces of quadratic differentials on Riemann surfaces. As it will be explained in Section \ref{innature}, quadratic differentials arise in the study of several dynamical systems, like billiards and interval-exchange transformations. The problem of finding the volumes of the different strata of their moduli spaces was solved by A. Eskin and A. Okounkov \cite{pillow}. This thesis is motivated by the pursuit of a closed formula for these volumes. Such a formula is not directly available from the methods described in \cite{pillow}.

Our work builds on the program proposed by A. Eskin and A. Okounkov \cite{pillow}, which we explain in Section \ref{motivationsection}. We proceed now to give a quick overview of our results. Most of the objects defined here are motivated by \cite{pillow}. The reader is referred to Appendix \ref{prerequisitesappendix} as a reference for the less standard background material we will use.

A \emph{partition} of a positive integer $n$ is a decomposition of $n$ as a sum of positive integers. To each partition $\lambda$ corresponds an irreducible linear representation of the symmetric group $S(n)$, with character $\chi^\lambda:S(n)\to\R$ and dimension $\dim \lambda$.

Denote by $\mathbf f_\eta$ the expression
\[\mathbf f_\eta(\lambda)=\left|C_\eta\right|\frac{\chi^\lambda(\eta)}{\dim\lambda},\]
where $C_\eta$ denotes the conjugacy class of the elements of $S(n)$ with cycle type $\eta$. (The function $\mathbf f_\eta$ arises naturally in the way explained in Section \ref{generatingfunctionmanipulations}.)

The \emph{pillowcase weights} $\w$ were first defined by A. Eskin and A. Okounkov \cite{pillow}, as follows:
\begin{equation}\label{pillowcaseweight}
\w(\lambda)=\left(\frac{\dim\lambda}{|\lambda|!}\right)^2\mathbf f_{(2,2,\dots)}(\lambda)^4.
\end{equation}
Here, $|\lambda|$ denotes the sum of the parts $\sum_i \lambda_i$ of the partition $\lambda=(\lambda_1,\lambda_2,\dots,\lambda_k)$.
For a parameter $q$, $|q|<1$, the weights $q^{|\lambda|}\w(\lambda)$ determine a probability distribution, the \emph{pillowcase distribution}, on the space of all partitions $\lambda$ of all the positive integers, after dividing by the normalization constant
 \[Z=\sum_\lambda q^{|\lambda|} \w(\lambda),\]
 which they also computed (see Section \ref{sec:quasimodularitytwoquotients} or \cite[Section 3.2.4]{pillow}). The volumes of the strata of the moduli space of (disconnected) quadratic differentials turn out to be given by the dominant term in the $q\to1$ asymptotics of the expectations $\langle\g_\nu\rangle_{\w,q}$ of the functions
\[\g_\nu(\lambda)=\frac{\mathbf f_{(\nu,2,2,\dots,2)}(\lambda)}{\mathbf
f_{(2,2,\dots,2)}(\lambda)}\]
with respect to this distribution
 (see Section \ref{howtocomputethevolume} for details). Here, the partition $\nu$ encodes the multiplicities of zeros and poles of the quadratic differentials in the corresponding stratum.

 This thesis investigates the numbers $\langle\g_\nu\rangle_{\w,q}$. Our first contribution is a fairly explicit formula for $\g_\nu$, which we present in Section \ref{sec:nearinvolutions}.

 We also prove that the expectations of the shifted Schur functions involved in that formula are quasimodular forms. This is important since it implies that the precise computation of these expectations is within our reach. This is done in Section \ref{sec:quasimodularitytwoquotients}.

 We then focus in the analysis of the pillowcase distribution. We are able to prove that it induces a limit shape, which coincides with the one for the uniform distribution; we explain this in Section \ref{sec:limitshape}. Thus the expectation is multiplicative in its highest degree. 
 However, we give examples that provide a negative answer to the question of existence of a full Central Limit Theorem; see Remark \ref{rmk:noCLT}.

A partition $\lambda=(\lambda_1,\lambda_2,\dots,\lambda_k)$ is said to be \emph{balanced} if its Young diagram can be constructed by concatenating two-cell dominoes $\yng(2)$.
Balanced partitions are in one-to-one correspondence with pairs of partitions $(\alpha,\beta)$ with $|\lambda|/2=|\alpha|+|\beta|$ through the construction of the so-called 2-quotients (see Section \ref{twoquotientssection}).

The weights $\w(\lambda)$ vanish on partitions that are not balanced. In Section \ref{sec:wformulahooks}, we prove the following appealing formula for the pillowcase weights on balanced partitions $\lambda$:
 \[\w(\lambda)=\left(\frac{\prod\textrm{odd hook lengths of $\lambda$}}{\prod\textrm{even hook lengths of $\lambda$}}\right)^2\]
(for the definition of the \emph{hooks lengths}, please see Section \ref{sec:hookformula}). We use this formula in Section \ref{sec:variational} to show that the probability is concentrated on the set of balanced partitions with 2-quotients composed of very similar partitions $\alpha\approx\beta$. This has the consequence, explained in Section \ref{sec:meinardus}, that the first term of the asymptotics of expectations of the form
\[\langle f(\alpha)g(\beta)\rangle_{\w,q}\]
coincide with the first term of the asymptotics of
\[\langle fg\rangle_{\w,q}.\]

From degree considerations (see Remark \ref{rmk:degreeconsiderations}), it is expected that the first term in the asymptotics of our formula for $\g_\nu$ will vanish. In fact, we are able to show this directly in Section \ref{sec:expectationasympt} using the result of Section \ref{sec:meinardus}. We end with a brief discussion of what the next term looks like in Section \ref{sec:nextterm}.

\section{Plan of the thesis}
The rest of this chapter is devoted to explaining the motivations for this thesis.

Chapter \ref{mychapter} is devoted to the study of the functions $\g_\nu$, and Chapter \ref{chap:pillowcase} is mainly about the weights $\w(\lambda)$ and the distribution they induce on the space of Young diagrams.

Appendix \ref{prerequisitesappendix} describes some mathematical prerequisites. 

\section{Motivation}\label{motivationsection}
In their paper \cite{pillow}, A. Eskin and A. Okounkov propose a method to compute the volume of the moduli space of quadratic differentials. This has been the main motivation for our work, so we review their approach here.

While Sections \ref{quadraticdifferentials} and \ref{moduli} mainly collect definitions and facts, the method is explained in Section \ref{howtocomputethevolume}, and additional results that make it work are collected in Sections \ref{pillowcasesection} through \ref{generatingfunctionmanipulations}.
Section \ref{innature} gives a couple of examples that aim to frame our study in a broader context. 
\subsection{Quadratic differentials}\label{quadraticdifferentials}
\paragraph{Abelian differentials.}

Let $S$ be a compact Riemann surface with the topology locally inherited from $\R^2$. Let $\widetilde {\Omega}^1$ denote the pre-sheaf of holomorphic 1-form germs, which to each open set $U$ assigns the complex vector space $\widetilde{\Omega}^1(U)$ of one-forms $\omega$ holomorphic on $U$. Sections of the corresponding sheaf $\Omega^1$ are known as \emph{Abelian differentials}.

\paragraph{Quadratic differentials.}

Define a new pre-sheaf $\widetilde{\mathfrak Q}$ on $S$ by letting $\widetilde {\mathfrak Q}(U)$ be the complex vector space of squares $\omega^2$ of 1-form germs $\omega\in \widetilde{\Omega}^1(U)$. Let $\mathfrak Q$ denote the corresponding sheaf. A \emph{quadratic differential} on $S$ is a meromorphic section of $\mathfrak Q$ with poles of degree at most one. In other words, on an atlas $\{U_i,z_i\}$ of open sets $U_i$ covering $S$ with local coordinates $z_i$, a quadratic differential consists of expressions $f_i(z_i)(dz_i)^2$ on each open set $U_i$, where the functions $f_i$ and $f_j$ are connected on $U_i\cap U_j$ by
\[\frac{f_j(z_j)}{f_i(z_i)}=\left(\frac{dz_i}{dz_j}\right)^2.\]
The ordinary product $(dz)^2$ should not be confused with the antisymmetric wedge product $dz\wedge dz=0$.

\paragraph{Zeroes and poles.}

Together, the zeroes and poles of a quadratic differential are known as its \emph{singular} or \emph{conic points}. Unless it vanishes identically, a quadratic differential has only finitely many singular points.

From the definition above, it follows that the quotient of two quadratic differentials is a meromorphic function on $S$ with as many zeroes as poles (taking their multiplicity into account). It follows that all quadratic differentials on $S$ have the same number of zeroes and poles. This should be twice the number of zeroes and poles on a holomorphic Abelian differential, because a quadratic differential is locally the square of an Abelian differential. By Riemann-Roch, we know that an Abelian differential has $2g-2$ zeroes on $S$, so the quadratic differentials must have a total of
\[\mathrm{zeroes}-\mathrm{poles}=4g-4,\]
counted with multiplicity.

\paragraph{Local square roots.}

If $Q$ does not have a zero or a pole at $p\in S$ then, on a small-enough neighborhood $U$ of $p$, $Q$ has a square root. It suffices to take $U$ to be a small open ball where $Q$ does not vanish and remains holomorphic. Let $w$ be a coordinate on $U$, and $f$ be such that $Q(w)=f(w)\,(dw)^2$. Then the square root of $Q$ on $U$ looks like
\[\sqrt{f(w)}\,dw.\]
One has a choice of two different complex square roots of $f(w)$. This choice can be done consistently throughout a small-enough neighborhood of $p$. Call the resulting form $\omega$. Define $z$ to be
\begin{equation}\label{coordinatechange}
z(w)=\int_p^w\omega,\qquad w\in U
\end{equation}
where the integral is taken over any path joining $p$ and $w$, and contained in $U$; by Cauchy's theorem, $z(w)$ is independent of the chosen path. Since $f$ does not vanish on $U$, $z$ is locally injective. Using $z$ as a new coordinate (possibly within a smaller open set $V\subset U$ that still contains $p$), we see that $dz=\omega$, so that $(dz)^2\equiv Q$.

\paragraph{Flat structure and global square root.}

The form $dz$ induces the structure of a \emph{flat surface} on $U$. A flat surface is a two-dimensional smooth manifold whose transition maps are given by translations $(x,y)\mapsto (x+a,y+b)$ in $\R^2$. We can find such coordinates $x$ and $y$ on $U$ by integrating the real and imaginary parts of $dz=dx+i\,dy$.

The flat structure cannot in general be extended throughout the set
\[S^\circ=S-\{\textrm{singular points of $Q$}\};\]
it may be necessary to have not only translations, but also their composition with reflections $z\mapsto -z$, among the transition maps. An example of this is the pillowcase orbifold; see Section \ref{pillowcasesection}. On the other hand, $S^\circ$ always has a double cover $\widetilde S^\circ$ where the definition of $dz$ can be extended globally, and filling in the holes in the obvious way, $\widetilde S^\circ$ can be completed into a ramified degree-two cover $\widetilde S$ of $S$.

Let us show how to construct $\widetilde S^\circ$. Take an atlas $\{U_i,w_i\}$ on $S^\circ$, such that each open set $U_i$ is simply connected, and $w_i$ is a coordinate on it. We want to produce a new space gluing copies of the sets $U_i$. On $U_i$, write $Q=f_i(w_i)\,(dw_i)^2$, and let $U_i^+$ and $U_i^-$ be two identical copies of $U_i$ corresponding to the two branches of the square root $\sqrt{f_i(w_i)}$. The choice of signs $+$ and $-$ on $U^\pm_i$ is arbitrary, of course, but the resulting construction will not depend on it. Now form the manifold corresponding to the full atlas $\{U_i^\pm,w_i\}$, where the sets $U_i^{s_i}$, $s_i\in\{+,-\}$, are glued along the intersections $U_i^{s_i}\cap U_j^{s_j}$ if the corresponding branches of the square root coincide on $U_i\cap U_j$. Finally, define $dz$ locally on $\widetilde S$ by letting $dz=\sqrt {f_i(dw_i)}dw_i$ or, equivalently, by defining $z$ locally as in equation \eqref{coordinatechange} and taking its differential.

\paragraph{Flat Riemannian metric.}

The form $dz$ induces on each open simply-connected set $U$, along with the flat structure, a Riemannian metric, locally given by $(dx)^2+(dy)^2$. This metric has null curvature throughout $U$. The same can be said about its extension to all of  $\widetilde S^\circ$. Moreover, since the metric is invariant under reflections, it descends to $S^\circ$. Intuitively speaking, we have thus pushed all the curvature of $S$ to a few conic points, where all the curvature is concentrated.

\paragraph{Horizontal and vertical foliations.}

The forms $dx$ and $dy$ induce two foliations of $\widetilde S^\circ$. The leafs are at each point $p$ tangent to the kernel of the respective forms in the tangent space $T_p\widetilde S^\circ$. These foliations are also invariant by local reflections $(x,y)\mapsto (-x,-y)$, so they descend to $S^\circ$. They are known as the \emph{vertical} and \emph{horizontal foliations}, respectively.

If $p$ is a point of $S$, we can consider the leaves of the horizontal and vertical foliations that  have $p$ in their topological closure. If the point $p$ is not a conic point, the picture is the same as for the origin in $\mathbb R^2$: there is one horizontal leaf entering $p$ from left and right, and there is a vertical leaf, coming from above and below. As we go around $p$, they alternate, and we cross two vertical pieces, and two horizontal pieces.

On a conic point, the situation is slightly different. If we move around $p$ on a non-self-intersecting closed loop, the leaves corresponding to the horizontal and vertical foliations again alternate because they are orthogonal to each other. What changes is the number of them that we cross. We could cross $n$ pieces of each type, horizontal or vertical, where $n$ can be any positive integer $\neq 2$. In other words, the angle around each conic point is no longer $2\pi$, but can be $\pi$, $3\pi$, $4\pi$, $5\pi$, $6\pi$,\dots, or $n\pi$ in general.

This is related to the degree of the zero or pole of $Q$ at $p$. If again $Q=f(w)(dw)^2$ locally, and if $f$ can be expanded as a series
\[f(w)\sim c_d(w-p)^d+c_{d+1}(w-p)^{d+1}+\textrm{higher order terms},\quad c_d\neq0,\]
near $p$, that is, $f$ has a zero of degree $d$ at $p$, then the number $n$ of leaves of the horizontal and vertical foliations with $p$ in their closure is $d+2$.
Heuristically, this is because, $dz$ (as above) is locally very much like $\sqrt{(w-p)^d}dw$, so $z$ is $\sim(w-p)^{\frac d2+1}$ (we take $z=0$ at $p$). Now, $z=\pm1$ on its `intersections with the horizontal axis', and $z=\pm i$ on its `intersections with the vertical axis,' and in this case the leaves of the horizontal and vertical foliations play the role of these axes. The two equations $(w-p)^{\frac d2+1}=\pm1$ and $(w-p)^{\frac d2+1}=\pm i$ have $d+2$ solutions each. The solutions to the first equation correspond to the leaves of the horizontal foliation emanating from $p$; the solutions to the second one correspond to the vertical foliation.

In polar coordinates centered at the singular point, the metric will look like
\[dr^2+\left(n \, r\,d\theta\right)^2.\]

\subsection{Examples of quadratic differentials in nature}\label{innature}
In this section we mention some research areas in which quadratic differentials have found applications.

A good overview is given by A. Zorich in \cite{zorich}. He includes applications we will not mention here, like Novikov's problem on the dynamics of Fermi surfaces in the study of electronic configurations in metals.

\paragraph{Billiards.}
Let $P$ be a polygon in the plane $\C$. Inside it, we define the billiard flow by having a point-like particle move in straight lines until it reaches the boundary, at which point it bounces following the usual rules of optical reflection. This dynamical system is known as a \emph{polygonal billiard}. A good introduction to these objects can be found in \cite{tabachnikovmasur,tabachnikov}.

Let $\Gamma\subset O(2)$ be the group generated by the reflections on the sides of $P$. If the size $|\Gamma|$ of this group is finite, $P$ is said to be a \emph{rational billiard}. This is equivalent to all the angles of the polygon $P$ being rational multiples of the number $\pi$.

To associate a surface and a quadratic differential with $P$, we take $|\Gamma|$ copies $gP$ of $P$, $g\in\Gamma$, and we think of each of them as being the image of $P$ under the action of a different element $g$ of $\Gamma$. For each copy $gP$ of $P$ and each reflection $r\in \Gamma$, we glue each edge $E$ of $gP$ to the edge $rE$ of $rgP$. When $\Gamma$ is finite, the result is a compact Riemann surface $S$.

The quadratic differential $(dz)^2$ in each copy $gP$ induces a quadratic differential $Q$ on $S$. It is easy to verify that the singularities of $Q$ are located at the vertices of the different copies of $P$. Thus, rational billiards can be seen as a (proper) subfamily of the moduli space of quadratic differentials.

The application of these ideas is well illustrated by the work of S. Kerckhoff, H. Masur, and J. Smillie \cite{kerckhoffmasursmillie}, in which they prove the ergodicity of rational billiard flows through the examination of the situation in the associated surface.

\paragraph{Interval exchange transformations.}

Let $X$ be an interval in the real line, and decompose it as a disjoint union of subintervals $X=X_1\sqcup X_2\sqcup\cdots\sqcup X_n$, where the indices indicate their actual order within $X$. Let $s\in S(n)$ be a permutation. The map $T:X\to X$ that scrambles the above decomposition according to $s$, $X=X_{s(1)}\sqcup X_{s(2)} \sqcup \cdots\sqcup X_{s(n)}$, is an \emph{interval exchange transformation}.

It turns out that all such objects in fact arise from picking a surface $S$, a quadratic differential $Q$ on $S$, a straight segment $X$ inside $S$ that does not intersect the singularities of $Q$, and studying the geodesic flow in the direction perpendicular to $X$; the first-return map $T:X\to X$ is an interval exchange transformation. The extremes of the intervals in this case are determined by the orbits that meet a singularity of $Q$ and hence do not return to $X$.

This connection with the theory of quadratic differentials is fruitful. A consequence of the work of S. Kerckhoff, H. Masur, and J. Smillie \cite {kerckhoffmasursmillie}, is that the interval exchange transformations are ergodic \cite{zorich}. 
\subsection{Moduli spaces}\label{moduli}
In this section we define a space that parameterizes all quadratic differentials on a surface, and we explain some of its properties.

\paragraph{Equivalence relation.}

We say that two quadratic differentials $Q_1$ and $Q_2$ defined on Riemann surfaces $S_1$ and $S_2$, respectively, are \emph{equivalent} if there is a homomorphism $\phi:S_1\to S_2$ sending the singular points of $Q_1$ to the singular points of $Q_2$ of the same order, and having the same transition functions in neighborhoods of all other points. That is, if around the point $p\in S_1$ we have a local coordinate $w$ and $Q_1=f_1(w)(dw)^2$, and if around the point $\phi(p)\in S_2$ we have a local coordinate $z$ and $Q_2=f_2(z)(dz)^2$, then $f_1(w(q))=f_2(z(\phi(q)))$ for every $q\in S_2$ near the point $\phi(p)$. This implies that $\phi$ is a diffeomorphism on the complement of the singular points of $Q_1$, and also that $S_1$ and $S_2$ are surfaces of the same genus.

\paragraph{Definition of the moduli space and its strata.} The \emph{moduli space of quadratic differentials} $\moduliquadratic$ is the set of all equivalence classes of pairs $(S,Q)$ consisting of a Riemann surface $S$ of genus $g$ and a quadratic differential $Q$ defined on it.

 The space $\moduliquadratic$ is naturally stratified by the degree of the quadratic differentials $Q$ at their of singular points. Let $k=(k_1, k_2, \dots, k_\ell)$ be the orders of the zeroes ($k_i>0$) and poles ($k_i=-1$) of the quadratic differential. We allow also $k_i=0$, which stand for marked points.
 We will denote by $\mathcal Q(k_1,k_2,\dots,k_\ell)$ the stratum of $\moduliquadratic$ corresponding to quadratic differentials $Q$ with (ordered) singular points $p_1,p_2,\dots,p_\ell$ of degrees $k_1,\dots, k_\ell$.

We say that a stratum is \emph{self-resolvent} if for every $(S,Q)$ in the stratum the degree-two cover $\widetilde S$ of $S$ constructed in Section \ref{quadraticdifferentials} is simply a disjoint union of two copies of $S$. Equivalently, this means that the square-root of the quadratic differential is defined globally on $S$. If this happens for one pair $(S,Q)$, then it happens for all points in the same stratum. For example, taking square-roots, we see that there is a one-to-one correspondence between the quadratic differentials represented in the strata $\mathcal Q(2k_1,2k_2,\dots,2k_\ell)$, with $\sum_i k_i=2g-2$ and $k_i\geq0$ for $i=1,\dots,\ell$, and the Abelian differentials with singularities of type $k_1,k_2,\dots,k_\ell$, so this stratum is self-resolvent.

\paragraph{Dimension and local coordinates.}

 H. Masur \cite{masur} and W. A. Veech \cite{veech}
 established the fact that the spaces $\mathcal Q(k_1, \dots , k_\ell)$ are complex orbifolds of complex dimension $2g-1+\ell$ if they are self-resolvent, and $2g-2+\ell$ otherwise.

Local coordinates on $\mathcal Q(k_1,\dots,k_\ell)$ are given as follows \cite{kontsevichlyapunovexponents,lanneauhyperelliptic}. Let  $(S,Q)\in\mathcal Q(k_1,\dots,k_\ell)$, and let $\widetilde S$ be the degree-two covering that resolves the square root of $Q$ as explained in Section \ref{quadraticdifferentials}. Denote by $\pi$ the covering map $\pi:\widetilde S\to S$ and by $\omega$ the square root of the pullback $\pi^*Q$ of the quadratic form $Q$. Let $P$ be the inverse image under $\pi$ of the set of singular points of $Q$. There is an involution $\sigma:\widetilde S\to\widetilde S$ that interchanges the fibers of $\pi$. Then $\sigma^*\omega=-\omega$, and the fixed points of $\sigma$ are a subset of $P$.
Since $\omega$ can be written as a closed form $dz$ on $\widetilde S-P$ (see Section \ref{quadraticdifferentials}), it defines an element of the relative cohomology group $H^1(S;P)$. The involution $\sigma$ induces an involution
\[\sigma^*:H^1(\widetilde S;P)\to H^1(\widetilde S;P),\]
which splits the vector space $H^1(S;P)$ into a direct sum of two eigenspaces:
\[H^1(\widetilde S;P)=V_1\oplus V_{-1}.\]
The class $[\omega]$ of the form $\omega$ belongs to the space of anti-invariant forms $V_{-1}$. A small neighborhood of $[\omega]$ inside $V_{-1}$ gives a local coordinate chart around $(S,Q)$.

Alternatively, one can define the \emph{period map} $\Phi$ that also gives a local coordinate \cite{masurzorichmultiplesaddle}. We will denote by $p_i$ the elements of the singular set $P=\{p_1,p_2\dots,p_\ell\}$ of $Q$. Let $\gamma_1,\dots,\gamma_{2g}$ be a standard symplectic basis of $H_1(\widetilde S)$, and complete it to a basis of $H_1(\widetilde S;P)$ by adding paths $\gamma_{2g+1},\dots,\gamma_{2g+\ell-1}$ whose boundary is in $P$ and more specifically:
\begin{equation*}
\partial \gamma_{2g+i}=p_{i+1}-p_1,\quad i=1,2,\dots,\ell-1.
\end{equation*}
We record the \emph{full period map} $\widetilde \Phi$ that will be used later, and is defined by
\begin{equation}\label{fullperiodmap}
\widetilde \Phi(S,Q)= \left(\int_{\gamma_{1}}\omega,\int_{\gamma_{2}}\omega,\dots, \int_{\gamma_{2g+\ell-1}} \omega\right).
\end{equation}
The map $\widetilde \Phi$ is locally injective but is not a local surjection in general; it maps to a space of very high dimension. To correct this, let $U_{-1}$ be the subspace of $H_1(\widetilde S;P)$ on which $\sigma$ acts as multiplication by $-1$. Choose cycles $c_1,c_2,\dots$ that form a basis of $U_1$. Then we let
\begin{equation}\label{periodmap}
\Phi(S,Q)=\left(\int_{c_1}\omega,\int_{c_2}\omega,\dots, \int_{c_{\dim \mathcal Q(k)}} \omega\right).
\end{equation}
This is the period map.

\paragraph{The area-2 slice.}

By the Riemann bilinear relations, in the symplectic basis $\gamma_1$, \dots, $\gamma_{2g}$ for the homology group $H_1(\widetilde S)$, the area of $S$ induced by $Q$, which is half the area of $\widetilde S$ induced by $\omega$,  is given by
\begin{equation}\label{areaofasurface}
\mathrm{area}(S)=\frac12\int_{\widetilde S} |\omega|^2\,dx\,dy
=\frac i4\int_{\widetilde S} \omega\wedge\bar\omega
=\frac i4 \sum_j \left(\int_{\gamma_j}\omega\int_{\gamma_{j+g}}\bar\omega -\int_{\gamma_j}\bar\omega\int_{\gamma_{j+g}}\omega\right)\!.
\end{equation}
Let $\mathcal Q_1(k_1,\dots,k_\ell)$ be the slice of the stratum corresponding to pairs $(S,Q)$ for which the area induced on $S$ by $Q$ equals 2. Note that the last expression in equation \eqref{areaofasurface} implies that in the coordinates given by the period map, $\mathcal Q_1(k_1,\dots,k_\ell)$ is a hyperboloid.

\paragraph{The measure.}

We want to define a measure $d\nu$ on $\mathcal Q_1(k_1,\dots,k_\ell)$. For a set $X\subset \mathbb R^n$, let
\[CX=\{tx:0\leq t\leq 1, x\in X\}\]
be the cone over $X$. Let $E\subset\mathcal Q_1(k_1,\dots,k_\ell)$ be a closed or open set entirely contained in the domain of a coordinate chart $\Phi$. Define
\[d\nu(E)=\mathrm{vol}(C\Phi(E)). \]
This defines the measure $d\nu$. It was proven by H. Masur \cite{masur} and W. Veech \cite{veechgauss} that the resulting volume $d\nu\left(\mathcal Q_1\left(k_1,\dots,k_\ell\right)\right)$ is finite.

\paragraph{Classification of the connected components.}

The classification of the connected components of the strata $\mathcal Q(k_1, \dots , k_\ell)$ was given by E. Lanneau \cite{lanneau}, building on the classification of the connected components of the strata of the moduli spaces of Abelian differentials, given by Kontsevich and Zorich \cite{kontsevichzorich}. We summarize the results. Let $g$ be the genus of the surfaces in the stratum $\mathcal Q(k_1, \dots , k_\ell)$ in question, and let $k=(k_1,\dots,k_\ell)$ be the list of integers $-1\leq k_i \neq 0$  that gives their singularity data, satisfying $\sum_i k_i=4g-4$.
The \emph{hyperelliptic} connected components are the ones that correspond to an entire stratum (with different singularity data $k$); we will list the correspondence below.
We have:
 \begin{itemize}
 \item If the quadratic differentials in the  stratum are self-resolvent \cite{kontsevichzorich}:
 \begin{itemize}
 \item If $g=1$, there is only one connected component.
 \item If $g=2$, the stratum has only one connected component, which is hyperelliptic.
 \item If $g=3$:
        \begin{itemize}
        \item If $k=(4,4)$ or $k=(8)$, the stratum has two connected components.
        \item All other strata are connected.
        \end{itemize}
  \item If $g\geq 4$:
     \begin{itemize}
     \item If $k=(4g-4)$ or $k=(4l,4l)$ for $l\geq2$, the stratum has three connected components: a hyperelliptic one, and two others corresponding to odd and even spin structures.
    \item If $k=(4l-2,4l-2)$ for $l\geq 2$, the stratum has two connected components: a hyperelliptic one, and a non-hyperelliptic one.
    \item All other strata are connected.
    \end{itemize}
    \end{itemize}
 \item If the stratum is not self-resolvent \cite{lanneau}, all strata that have two or more connected components have exactly one hyperelliptic component. The picture is the following:
   \begin{itemize}
   \item If $g=0$, all strata are connected.
   \item If $g=1$, all strata are connected, but the strata corresponding to $k=\emptyset$ and $k=(1,-1)$ are empty.
   \item If $g=2$, and if $k=(-1,-1,6)$ or $k=(-1,-1,3,3)$, then the stratum has two connected components. All other strata with $g=2$ are connected.
   \item If $g=3$, and if $k=(-1,9)$, $k=(-1,3,6)$, or $k=(-1,3,3,3)$, the stratum has two connected components. All other strata with $g=3$ are connected.
   \item If $g=4$, and if $k=(12)$, the stratum has exactly 2 connected components. All other strata with $g=4$ are connected.
    \item If $g\geq 5$:
    \begin{itemize}
    \item In  the following cases there are exactly two connected components:
        \begin{itemize}
        \item $k=(4(g-l)-6,4l+2)$, $0\leq l\leq g-2$,
        \item $k=(4(g-l)-6,2l+1,2l+1)$, $0\leq l \leq g-1$, and
        \item $k=(2(g-l)-3,2(g-l)-3,2l+1,2l+1)$, $0\leq l \leq g-2$.
        \end{itemize}
    \item All other strata are connected.
        \end{itemize}
    \end{itemize}

 \end{itemize}

The hyperelliptic components are (by definition \cite{lanneauhyperelliptic}) precisely the ones corresponding to the \emph{images} of the following injective maps:
\begin{enumerate}
\item $\mathcal Q(2(i-j)-3,2j+1,-1^{2i+2})\to\mathcal Q(2(i-j)-3,2(i-j)-3,2i+1,2i+1)$, where $j\geq-1$, $i\geq 1$, $i-j\geq 2$,
\item $\mathcal Q(2(i-j)-3,2j,-1^{2i+1})\to\mathcal Q(2(i-j)-3,2(i-j)-3,4j+2)$, where $j\geq0$, $i\geq 1$, $i-j\geq1$,
\item $\mathcal Q(2(i-j)-4,2j,-1^{2i})\to\mathcal Q(4(i-j)-6,4j+2)$, where $j\geq0$, $i\geq2$, and $i-j\geq 2$.
\end{enumerate}

\paragraph{Action of $SL_2(\R)$.}

As remarked by  H. Masur \cite{masur} and W. A. Veech \cite{veech}, there is an action of $SL(2,\mathbb R)$ on any given quadratic differential $Q$ on a surface $S$, given by
\[A\cdot Q_p(v)=Q_p(Av),\qquad A\in SL(2,\R), v\in T_pS,p\in S.\]
This preserves the singular points of $Q$ and induces an action of $SL(2,\mathbb R)$ on the slice $\mathcal Q_1(k_1,\dots,k_\ell)$, given by
\[A\cdot(S,Q)=(S,A\cdot Q).\]
The action of the one-parameter subgroup of $SL(2,\mathbb R)$ whose elements are of the form
\[\begin{pmatrix}e^t & 0 \\ 0 & e^{-t} \end{pmatrix}\]
turns out to be \emph{ergodic} in each connected component of $\mathcal Q_1(k_1,\dots,k_\ell)$. This means that it preserves the measure defined above, and that invariant sets that are Lebesgue measurable are of either null or full measure.

\subsection{Counting quadratic differentials} \label{howtocomputethevolume}
We now describe the approach of A. Eskin and A. Okounkov \cite{pillow} to finding the volume $d\nu(\mathcal Q_1(k_1,\dots,k_\ell))$.

\paragraph{Description of the method.}

In Section \ref{specialfamily}, we will define a regular lattice $\mathcal F_{\nu,\mu,p}$ inside $\mathcal Q(k_1,\dots,k_\ell)$. Since the stratum is closed under multiplication by complex scalars, the situation then becomes very similar to the following: Imagine having an open cone $K\subset \R^n$. If we were able to count the number of points $K\cap\Z^n\cap \{|x|<D\}$ for each $D>0$,  then the volume of $K\cap\{|x|<1\}$ would simply equal
\[\lim_{D\to+\infty} D^{-n}\left|K\cap\Z^n\cap\left\{|x|<D\right\}\right|.\]

The objects forming the lattice $\mathcal F_{\nu,\mu,p}$ will turn out to be certain ramified covers of the sphere $\sphere$. Let us the denote by $\mathcal F_d\subset \mathcal F_{\nu,\mu,p}$ the covers of degree $d$.
By Lemma \ref{latticelem}, the area induced by the quadratic form corresponding to each element of $\mathcal F_d$ is equal to $2d$. The area induced by the quadratic differential whose coordinates are $\sqrt D \Phi(S,Q)$ is $D$ times the area of $(S,Q)$. Thus, when identified as a lattice inside $\mathcal Q(k_1,\dots,k_\ell)$, we have
\[
\bigcup_{d\leq D} \mathcal F_d = \mathcal F_{\nu,\mu,p}\cap C_{\sqrt D}\mathcal Q_1(k_1,\dots,k_\ell),
\]
where $C_uX=\{(S,tQ):0\leq t\leq u,(S,Q)\in X\}$, $u>0$.
Their count, as we will see in Section \ref{generatingfunctionmanipulations}, is encoded in the following generating function:
\begin{align}
Z(\mu,\nu;q)&=\sum_{d=0}^\infty q^d\left|\mathcal F_d\right|\notag
\\
&=\sum_\lambda q^{|\lambda|/2} \left(\frac{\dim\lambda}{|\lambda|!}\right)^2
\mathbf f_{(\nu,2,2,\dots)}(\lambda)\mathbf f_{(2,2,\dots)}(\lambda)^3
\label{generatingfunction}
\end{align}
The sum is over all partitions $\lambda$.
In Section \ref{asymptotics}, we will explain that this generating function is, in fact, a polynomial of quasimodular forms.
Whence the limit
\begin{equation}\label{volumelimit}
d\nu(\mathcal Q_1(k_1,\dots,k_\ell))=\lim_{d\to\infty} d^{-\dim_\C\mathcal Q(k_1,\dots,k_\ell)}|\mathcal F_d|
\end{equation}
is amenable to computation. One does this by looking at the $q\to 1$ (from below) asymptotics of the generating function \eqref{generatingfunction}; this amounts to letting $d\to\infty$. 
In Section \ref{asymptotics}, we will explain the approach proposed in \cite{pillow} for the understanding of these asymptotics.

\paragraph{Technical remarks.}

Some additional comments are necessary in order to make the above argument rigorous. First, the boundary of $\mathcal Q_1(k_1,\dots,k_\ell)$ is rectifiable; this is easy to show. Second, in order to prove equation \eqref{volumelimit}, it is better to work with the image of $\Phi$. One should observe that it follows from the proof of the finiteness of the volume in \cite{masur,veechgauss} that for every $\varepsilon>0$ there is some compact subset $K_\varepsilon\subseteq \mathcal Q_1(k_1,\dots,k_\ell)$ such that $\nu(K_\varepsilon)\geq \nu\left(\mathcal Q_1(k_1,\dots,k_\ell)\right)-\varepsilon$. It is for cones over these sets $K_\varepsilon$ that the above should be done, if one is to do it carefully, and this is not difficult. This concludes the justification of the equality \eqref{volumelimit}.

Finally, we remark that this method does not distinguish between the different connected components of the strata, and applications in dynamics do require this distinction. However, a quick examination of the classification given in Section \ref{moduli} reveals that it is enough to know the volumes of the hyperelliptic components, except for finitely many sporadic cases, whose volumes can be computed using, for example, the method devised by M. Kontsevich and A. Zorich \cite{kontsevichzorich}, or the one devised for the case of Abelian differentials by A. Eskin and A. Okounkov \cite{branchedcoverings}, which is analogous to what was described in this section.

\subsection{The pillowcase orbifold} \label{pillowcasesection}
\paragraph{Definition and description.}

Let $L$ be a lattice in the complex plane $\C$, $L\cong\Z^2$, and let $\mathbb T^2=\C/L$ be the associated complex torus. The \emph{pillowcase orbifold} $\mathfrak P$ is the space obtained by taking the quotient space of $\mathbb T^2=\C/L$ by the action of the automorphism $z\mapsto -z$.

Although the underlying topological space does not depend on $L$, the complex structure of the pillowcase orbifold does. For our purposes, however, any lattice that induces area 2 on the orbifold will do. For simplicity we will assume from here on that $L=2\Z+2i\Z$.

The intuitive picture is exactly the one that its name suggests: the pillowcase orbifold can be obtained by superimposing two identical squares of side 1 and gluing their sides. More precisely, one can take the fundamental domain for the action of $z\mapsto -z$ given by the rectangle between the points $0$, $1$, $2i$, and $2i+1$. Here is a picture.

\begin{center}
\includegraphics{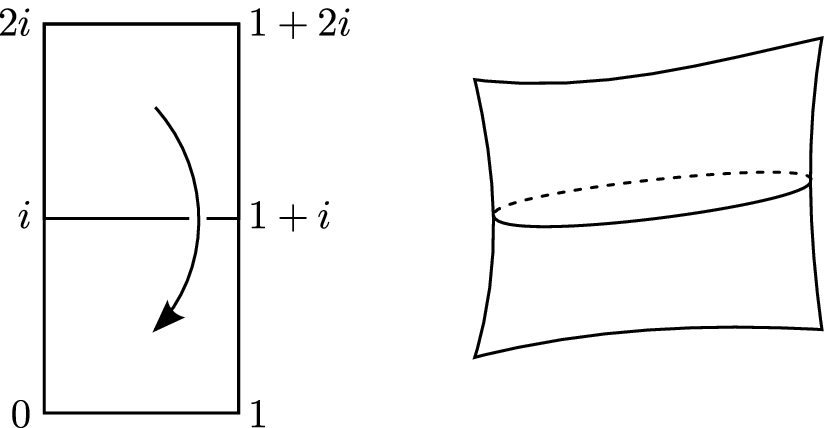}
\end{center}

The action of $z\mapsto -z$ will identify the boundary segments as follows:
\[[0,1]\leftrightarrow[2i,1+2i], \qquad [0,i]\leftrightarrow[2i,i],\qquad [1,1+i]\leftrightarrow[1+2i,1+i].\]
In this notation, $[a,b]$ denotes the oriented line segment joining $a$ with $b$. This fundamental domain is formed by two squares of side one, divided by the segment $[i,i+1]$, which we can picture as the locus where we should fold to obtain $\mathfrak P$.
There are four conic points on the resulting orbifold, one corresponding to each corner $0$, $1$, $i$, and $i+1$. The angle around each conic point is $\pi$.

\paragraph{Covers and global square-roots.}

Now consider the covers $\mathcal C$ of $\mathfrak P$ ramified only at the conic points of $\mathfrak P$. One can picture these covers as surfaces built by gluing squares of side one.

The quadratic differential $(dz)^2$ that the torus $\mathbb T^2$ inherits from $\C$ descends to $\mathfrak P$, except at the conic points. It can also be pulled back to the cover $\mathcal C$ by means of the cover map $\pi:\mathcal C\to\mathfrak P$. On $\mathfrak P$, $(dz)^2$ has four simple poles located at the conic points.

Locally on $\mathcal C$, $\pi^*(dz)^2$ has two square roots given by the Abelian differentials $\pi^*dz$ and $-\pi^*dz$, the pullbacks of the locally-defined square roots of $(dz)^2$ on $\mathfrak P$.

In $\mathfrak P$ it is impossible to find a global (i.e., defined at all points except the conic ones) square root of $(dz)^2$. To see why, note that taking the fundamental domain described above, it is clear that $dz$ is a well-defined square root of $(dz)^2$ within the interior of the two squares of side 1 that constitute the pillowcase. However, as soon as one tries to go from the square $(0,1)\times(0,i)$ to the square $(0,1)\times(i,2i)$ by crossing the side segment $[0,i]$, it becomes clear that the correct continuation for $dz$ would be $-dz$ on the second square. However, this contradicts the fact that the right continuation for $dz$ when crossing the segment $[0,1]$ is $dz$ itself. Using the procedure described in Section \ref{moduli}, we find again the degree-two cover $\pi_{\mathbb T^2}:\mathbb T^2\to \mathfrak P$ where $dz$ is well-defined globally.

\subsection{A special family of covers of the pillowcase} \label{specialfamily}
The pillowcase orbifold $\mathfrak P$, with resolvent cover $\mathbb T^2$, is homeomorphic to the sphere $\sphere$. Viewed as such, the map $\pi_{\mathbb T^2}:\mathbb T^2\to \mathfrak P$ has ramification profile $(2)$ at each of the four conic points. (Ramification profiles are explained in Section \ref{ramifiedcovers}.)

In a similar vein, we will be interested in certain covers $\pi:S\to\mathfrak P$. We fix a partition $\nu$ of an \emph{even} number into \emph{odd} parts, a partition $\mu$, and $\ell(\mu)$ points $p_1,\dots ,p_{\ell(\mu)}$ in $\mathfrak P$. We require the cover to have profile $(\nu, 2^{d-|\nu|/2})$ over $0\in\mathfrak P$, and profile $(2^{d})$ over the other three conic points of $\mathfrak P$. Additionally, we require each point $p_i$ to have profile $(\mu_i,1^{d-\mu_i})$. Denote by $\mathcal F_{\nu,\mu,p}$ the (discrete) family of covers $\pi:S\to \mathfrak P$ with these properties.

The quadratic differential $\pi^*(dz)^2$ on $S$ has zeros of order $\nu_i-2$ on each one of the preimages of $0$, and $2\mu_i-2$ on one of the preimages of each $p_i$. As in the case of $\mathbb T^2$, a ramification profile of $2$ over a conic point of $\mathfrak P$ produces no zeros on the quadratic differential defined on the cover. In other words, $\mathcal F_{\nu,\mu,p}\subseteq \mathcal Q(k)$ with $k=(\nu_i-2,2\mu_i-2)$, so the choices of $\nu$ and $\mu$ are obvious, given $k$: $\nu$ corresponds to the odd entries of $k$, while $\mu$ corresponds to the even ones. Simple poles correspond to $\nu_i=1$.

We remark that specializing the Riemann-Hurwitz formula \eqref{riemannhurwitz} to the case of $\pi:S\to\mathfrak P$ in  $\mathcal F_{\nu,\mu,p}$, since $\chi(\mathfrak P)=0$, we get
\[\chi(S)= \ell(\mu)+\ell(\nu)-|\mu|-|\nu|/2.\]
Note that this does not depend on $d$, and completely determines the genus $g$ of $S$ through the relation $\chi(S)=2-2g$.

As explained in Section \ref{howtocomputethevolume}, the following lemma is crucial.

\begin{lem}\label{latticeprop}
The family $\mathcal F_{\nu,\mu,p}$ forms a regular lattice within the stratum $\mathcal Q(k)$ where $k=\{\nu_i-2,2\mu_i-2\}$.
\end{lem}

To prove the lemma we will first need another lemma. Given a point $(S,Q)$ in the stratum $\mathcal Q(k)$, denote by $\widetilde S$  the unique degree-two cover $\pi:\widetilde S\to S$ of $S$ on which there is a form $\omega$ that is a globally defined square root of the pullback of $Q$ to $\widetilde S$ (see Section \ref{quadraticdifferentials}), and by $\mathbb T^2=\R^2/(2\Z)^2$ the standard torus of area 4.
Recall the period map $\Phi$  and the full period map $\widetilde \Phi$ were defined in equations \eqref{periodmap} and \eqref{fullperiodmap}, respectively.

\begin{lem}[Analogous to \protect{\cite[Lemma 3.1]{branchedcoverings}}]
\label{latticelem}
Let $S$ be a surface of genus $g$.
The first $2g$ coordinates $\left(\widetilde \Phi (S,Q)\right)_i$, $i=1,2,\dots,2g$, of the image of a point $(S,Q)\in \mathcal Q(k)$ under the full period map $\widetilde \Phi$ are in the lattice $2\Z+2i\Z$ of complex points with even real and imaginary parts if, and only if, the following holds:
\begin{enumerate}[$(a)$]
  \item There exists a holomorphic map $\tau:\widetilde S\to \mathbb T^2$ that makes $\widetilde S$ into a ramified cover of $\mathbb T^2$.
  \item $\omega=\tau^*dz$, for $\omega$ as above.
  \item The ramification points $q_i\in\widetilde S$ of $\tau$ are a subset of $\pi^{-1}(\textrm{conic points of $Q$})$, and they coincide with the set of zeros of $\omega$.
  \item The ramification of $\tau$ around $q_i$ is locally of the form $z\mapsto z^{k_i+1}$.
  \item $\tau(q_{i+1})-\tau(q_1)=\left(\widetilde\Phi(S,Q)\right)_{2g+i}\mod (2\Z)^2$.
  \item The degree of $\tau$ is equal to a quarter of the area of $S$, as defined by equation \eqref{areaofasurface}.
\end{enumerate}
\end{lem}
\begin{proof}[Proof of Lemma \ref{latticelem}]
Sufficiency is clear; to prove necessity, define $\tau$ by
\[\tau(z)=\int_p^z \omega \mod 2\Z+2i\Z,\]
where $p\in \widetilde S$ is an arbitrary point. The integral is taken over any path joining $p$ and $z\in \widetilde S$. By Cauchy's theorem, the result depends only on the homotopy class of  of the chosen path. But since $\left(\widetilde \Phi(S,Q)\right)_j\in 2\Z+2i\Z$, $j=1,2,\dots,2g$, the integral of $\omega$ along any closed path in $S$ results in a complex number whose coordinates are even integers, so $\tau$ is independent of the chosen path.
\end{proof}

\begin{proof}[Proof of Lemma \ref{latticeprop}]
 We want to relate the images of $\Phi$ and $\widetilde \Phi$. Let $U$ be an open set entirely contained in a chart of $\mathcal Q(k)$, that is, such that $\Phi$ is injective in $U$, and assume $U$ is closed under the action of $\C$. It is clear that $\widetilde \Phi$ is injective on $U$ as well.

For each cover $S\to\mathfrak P$ in $\mathcal F_{\nu,\mu,p}$, the double-covering $\widetilde S$ satisfies the hypotheses of Lemma \ref{latticelem}, so within $\widetilde \Phi(U)$, we have
\begin{equation}\label{familyinsidelattice}
\widetilde \Phi(U\cap \mathcal F_{\nu,\mu,p})\subseteq\widetilde \Phi(U)\cap \left( (2\Z+2i\Z)^{2g}\times \C^{\ell-1}\right).
\end{equation}
Moreover,
if we let $P_{2g}$ be the projection of $\C^{2g+\ell-1}$ onto its first $2g$ coordinates, then from the definition of $\mathcal F_{\nu,\mu,p}$ it is clear that \[P_{2g}\Phi\left(U\cap\mathcal F_{\nu,\mu,p}\right)=P_{2g}\Phi(U)\cap(2\Z+2i\Z)^{2g}.\]
Also, the requirement that the ramifications be above the prescribed points of $\mathbb T^2$ ensures that $\widetilde\Phi\left(U\cap \mathcal F_{\nu,\mu,p}\right)$ is also structured as a lattice in the remaining $\ell-1$ complex dimensions.

Going from the images $\widetilde \Phi(S,Q)$ to the corresponding points $\Phi(S,Q)$ simply entails finding a matrix $A=(a_{ij})$ with integral entries such that
\[c_i=\sum_j a_{ij}\gamma_j,\]
where $c_i$ and $\gamma_j$ are as in equations \eqref{periodmap} and \eqref{fullperiodmap}, respectively. Such a matrix can be found in which the $(a_{ij})$ are constant throughout $\mathcal Q(k)$ (simply fix in the underlying topological space of the surface $S$ which is the same throughout all points $(S,Q)\in\mathcal Q(k)$, the generators of the homology groups involved in the definition of $\Phi$ and $\widetilde\Phi$).
Hence there is a linear mapping taking $\widetilde \Phi(U)$ to $\Phi(U)$ bijectively. Of course, $\Phi(U\cap \mathcal F_{\nu,\mu,p})$ is just the image of the lattice $\widetilde\Phi\left(U\cap\mathcal F_{\nu,\mu,p}\right)$ under the linear transformation $A$.
\end{proof}

\subsection{Counting ramified covers of the sphere} \label{generatingfunctionmanipulations}

Consider a cover of the sphere $\sphere$ of degree $d$, ramified at points $p_1,p_2,\dots,p_k\in \sphere$. For each $i$, let $\Gamma_i$ be a small non-self-intersecting loop encircling $p_i$ only, and intersecting no other loop $\Gamma_j$, $i\neq j$. Let $s_i$ be the monodromy permutation associated to the loop $\Gamma_i$ (see Section \ref{ramifiedcovers}). Since the sphere is simply connected, the sum of the cycles $\Gamma_i$ is homologically equivalent to the null cycle in $H_1(\sphere-\{p_1,\dots,p_k\})$. One deduces that the composition of monodromy permutations $s_1s_2\cdots s_k$ is in fact equal to the identity. This is why the following lemma is true.

\begin{lem}\label{coversofspherelemma}
The ramified covers of the sphere $\pi:C\to\sphere$ of degree $d$ are completely determined, up to isomorphism, by the following data:
\begin{itemize}
\item A finite set of points $p_1,p_2\dots,p_k\in\sphere$.
\item Permutations $s_1,s_2,\dots,s_k\in S(d)$ such that $s_1s_2\cdots s_k=1$.
\end{itemize}
\end{lem}

\begin{rmk}\label{coversofsphereremark}
The correspondence is not one-to-one. For instance, by conjugating the permutations $s_i$ by an arbitrary permutation $\sigma$ we get another instance of data determining the same cover. If the cover does not have nontrivial automorphisms, there are $|S(d)|=d!$ ways to alter the data without changing the cover determined by it. On the other hand, if the cover does have nontrivial automorphisms, there are only $|S(d)|/|\aut(C)|$ ways to alter the data.
\end{rmk}

Let $\mathsf H_d(\eta^1,\dots,\eta^k)$ denote the number of degree $d$ covers $\pi:C\to \sphere$ of the
sphere, ramified at $k$ points $p_1,p_2,\dots,p_k$, with monodromies $s_1,s_2,\dots,s_k\in S(d)$ of cycle type $\eta^1,\eta^2,\dots,\eta^k$, respectively, counted with weight $1/|\aut(C)|$. $\mathsf H_d(\eta^1,\dots,\eta^k)$ is known as the \emph{Hurwitz number} for covers of this type.

As $d$ goes to infinity, the number of covers with non-trivial automorphisms grows much slower than the total number of covers, so the weights $1/|\aut(C)|$ end up being irrelevant; see \cite[Section 3.1]{branchedcoverings}.

\begin{prop}
\[\mathsf H_d(\eta^1,\dots,\eta^k)=\sum_{|\lambda|=d}\left(\frac{\dim \lambda}{d!}\right)^2\prod_{i=1}^k\mathbf f_{\eta^i}(\lambda). \]
\end{prop}
\begin{rmk}
Specializing this fomula to the case of $\mathcal F_{\nu,\mu,p}$ with $p_i=0$ for all $i$, and assimilating the parts of $\mu$ into $\nu$, we get exactly formula \eqref{generatingfunction}: the ramification over 0 is of type $(\nu,2,2,\dots,2)$, and we have ramifications of type $(2,2,\dots,2)$ over the other three conic points of the pillowcase orbifold $\mathfrak P$.
\end{rmk}
\begin{proof}
By Lemma \ref{coversofspherelemma} and Remark \ref{coversofsphereremark}, $\mathsf H_d(\eta^1,\dots,\eta^k)$ equals the cardinality $|A_\eta|$ of the set
\[
A_\eta=\{(s_1,s_2,\dots,s_k)\in S(d)^k:\textrm{$s_i$ has cycle type $\eta^i$, $s_1s_2\cdots s_k=1$}\},
\]
divided by $|S(d)|=d!$.

Let $\Q S(d)$ be the group algebra of the symmetric group $S(d)$. Its center $\mathcal Z(d)$ is known as the \emph{class algebra} and is generated by the vectors
\[z_\lambda=\sum_{\sigma\in C_\lambda} \sigma,\quad \textrm{$\lambda$ a partition of $d$}.\]
Here, $C_\lambda$ denotes the conjugacy class of $S(n)$ consisting of all elements of cycle type $\lambda$. If we take the product $z_{\eta^1}z_{\eta^2}\cdots z_{\eta^k}$, we get precisely the sum of all the different products $s_1s_2\cdots s_k$ of permutations $s_i$ with respective cycle types $\eta^i$.

To determine the cardinality of $A_\eta$, we want to count only the summands that equal the identity. The key observation is that in the adjoint representation $\Q S(d)$, the matrices associated to the action of all other summands have only zeros on the diagonal, since they are permutation matrices. Hence, the number of summands that equal the identity is the trace of the operator $z_{\eta^1}z_{\eta^2}\cdots z_{\eta^k}$ divided by the dimension $d!$ of the space $\Q S(d)$. In other words,
\begin{equation*}
|A_\eta|=\frac1{d!}\trace z_{\eta^1}z_{\eta^2}\cdots z_{\eta^k}.
\end{equation*}

Recall that $\Q S(d)$ reduces to a direct sum of all the irreducible representations $V_\lambda$ of $S(d)$ indexed by partitions $\lambda$ with $|\lambda|=d$, where each of these representations appears with the same multiplicity as its dimension (see Section \ref{representationssection}):
\[\Q S(d)=\bigoplus_{|\lambda|=d} (V_\lambda)^{\oplus\dim \lambda}.\]

Since the operators $z_{\eta^i}$ commute with the whole algebra, by Schur's lemma (see Section \ref{representationssection})
we know that they act as scalars in each $V_\lambda$. The actual scalar they represent is computed as follows: The trace of a scalar matrix $\alpha I$, $\alpha \in \C$, equals $\alpha$ times the dimension of the space, so we will take the trace and divide by the dimension $\dim \lambda$. Since $z_{\eta^i}$ is the sum of all the elements of the conjugacy class $C_{\eta^i}$, its trace equals the trace of each one of them ---namely, the character $\chi^\lambda (\eta^i)$--- multiplied by the size $|C_{\eta^i}|$ of the conjugacy class. We conclude that $z_{\eta^i}$ acts as multiplication by
\[\mathbf f_{\eta^i} (\lambda)=|C_{\eta^i}|\frac{\chi^\lambda(\eta^i)}{\dim \lambda}\]
within $V_\lambda$.

Since the operators $z_{\eta^i}$ are scalars in $V_\lambda$, we can express the trace as a product, as follows:
\begin{align*}
\trace_{\Q S(d)} z_{\eta^1}z_{\eta^2}\cdots z_{\eta^k}&=
\sum_{|\lambda|=d} \dim\lambda\cdot \trace_{V_\lambda} z_{\eta^1}z_{\eta^2}\cdots z_{\eta^k}\\
&=\sum_{|\lambda|=d}\dim\lambda\prod_i \trace_{V_\lambda} z_{\eta^i} \\
&=\sum_{|\lambda|=d}(\dim\lambda)^2\prod_i \mathbf f_{\eta^i}(\lambda).
\end{align*}
Since, as explained above,
\begin{equation}\label{keytofootnote}
\mathsf H_d(\eta^1,\dots,\eta^k)=\frac{1}{d!}|A_\eta|=\frac1{(d!)^2}\trace_{\Q S(d)} z_{\eta^1}z_{\eta^2}\cdots z_{\eta^k},
\end{equation}
we are done.
\end{proof}


\subsection{Asymptotics of the generating function} \label{asymptotics} \label{whatgoeswrong}
In \cite{pillow}, A. Eskin and A. Okounkov treat the generating function \eqref{generatingfunction} as the expectation of the function
\[\g_\nu(\lambda)=\frac{\mathbf f_{(\nu,2,2,\dots)}(\lambda)}{\mathbf f_{(2,2,\dots)}(\lambda)}\]
with respect to the probability distribution on the space of young diagrams $\lambda$ induced by the pillowcase weights
\[\w(\lambda)=\left(\frac{\dim \lambda}{|\lambda|!}\right)^2 \mathbf f_{(2,2,\dots)}(\lambda)^4, \]
multiplied by the additional (complex) parameter $q^{|\lambda|}$, $|q|<1$. This expectation is exactly the sum \eqref{generatingfunction} divided by a normalization constant 
\[Z=\sum_\lambda q^{|\lambda|}\w(\lambda)\]
that makes the weights $q^{|\lambda|}\w(\lambda)$ a probability distribution. The sum defining $Z$ is finite for $|q|<1$ because $0\leq\w(\lambda)<1$ as was proven by S. Fomin and N. Lulov \cite{fominlulov}. In other words,
\[Z(\mu,\nu;q)=Z\cdot \left\langle\g_\nu
\right\rangle_{\w,q}\]
where $\langle\cdot\rangle_{\w,q}$ denotes the expectation with respect to the distribution described above.

They prove that the functions $\g_\nu$ belong to the algebra $\overline\Lambda$ generated by the functions
\[\p_k(\lambda)=\sum_i \left[\left(\lambda_i-i+\tfrac12\right)^k-
\left(-i+\tfrac12\right)^k\right]+\left(1-\frac{1}{2^k}\right)\zeta(-k),\]
 and their $\Z/2\Z$-twisted analogs
 \[\pbar_k(\lambda)=\sum_i \left[(-1)^{\lambda_I-i+1}\left(\lambda_i-i+\tfrac12\right)^k-
(-1)^{-i+1}\left(-i+\tfrac12\right)^k\right]+c_k,\]
where $c_k$ is defined by
\[\sum_k\frac{z^k}{k!}\pbar_k(\emptyset)=\frac1{e^{z/2}-e^{-z/2}}.\]

 Finally, they also prove that the expectation of any function in $\overline\Lambda$ with respect to the distribution induced by the weights $q^{|\lambda|}\w(\lambda)$ is a quasimodular form (see Section \ref{quasimodularforms}).

To do this, they find that the exponential generating function
\begin{multline*}
F\left(e^{x_1},\dots,e^{x_n},-e^{y_1},\dots,-e^{y_m}\right)= \\ \sum_{i_1,\dots,i_n,j_1,\dots,j_n=1}^\infty \frac{x_1^{i_1}}{i_1!}\cdots \frac{x_n^{i_n}}{i_n!} \frac{y_1^{j_1}}{j_1!}\cdots \frac{y_m^{j_m}}{j_m!}\left\langle \p_{i_1}\cdots \p_{i_n}\pbar_{j_1}\cdots\pbar_{j_m}\right\rangle_{\w,q}
\end{multline*}
can be expressed explicitly in terms of Jacobi theta functions
\[
\vartheta(x,q)=(q^{1/2}-q^{-1/2})\prod_{i=1}^\infty\frac{(1-q^ix)(1-q^i/x)}{(1-q^i)^2}
\]
(see also Section \ref{jacobithetas}), as follows. We let $N=n+m$,
\[\hat x_i=\left\{\begin{array}{ll}
e^{x_i}, & i=1,2,\dots, n, \\
-e^{y_{i-n}}, & i=n+1,n+2,\dots,N.
\end{array}\right.
\]
Then $F$ equals
\begin{equation}\label{quotientofthetas}
\frac{1}{\prod_i\vartheta(\hat x_i)}[u_1^0\cdots u_{N}^0]\prod_{i<j}\frac{\vartheta(u_i/u_j)\vartheta(\hat x_iu_i/\hat x_ju_j)}{\vartheta(\hat x_iu_i/u_j)\vartheta(u_i/\hat x_ju_j)}\prod_i\sqrt{\frac{\vartheta(-u_i)\vartheta(\hat x_iu_i)}{\vartheta(u_i)\vartheta(-\hat x_iu_i)}}
\end{equation}
(For clarity, we have omitted $q$ in each occurrence of $\vartheta$.) The brackets indicate that we take the coefficient of $u_1^0\cdots u_N^0$ of the expansion of what is found to the right of them. The series expansion of this generating function must be taken in the domain
\[|u_N/q|>|\hat x_1 u_1|>|u_1|>\cdots>|\hat x_Nu_N|>|u_N|>1.\]
The rest of their argument is very similar to what we do in Section \ref{sec:quasimodularitytwoquotients}.

With expression \eqref{quotientofthetas} at hand, one can plausibly use the modular transformation of $\theta$ and integrals on loops surrounding zero to find, by stationary phase, the coefficients
\[\left\langle \p_{i_1}\cdots \p_{i_n}\pbar_{j_1}\cdots\pbar_{j_m}\right\rangle_{\w,q},\]
from where the expectation $\g_\nu$ can also be derived, since it is simply a linear combination of these numbers.

We note that for the shifted power functions $\p_k$ the first term of the expectation is multiplicative (see Section \ref{sec:limitshape}):
\[\langle\p_i\p_j\rangle_\w-\langle\p_i\rangle_\w\langle\p_j\rangle_\w =o(1-q)^{-i-j-2}.\]
(Of course, $\langle \p_k\rangle_\w=O(1-q)^{-k-1}$; see \cite[Section 3.3.5]{pillow}.)
However, the expectations of the functions $\pbar_k$ turn out to be quite delicate. They vanish to first order, the expectation of their products is not multiplicative, and a simplification reminiscent of Wick's theorem, that is, an identity of the kind
\[\langle abcd \rangle = \langle ab \rangle\langle cd\rangle+\langle ac \rangle\langle bd\rangle+\langle ad \rangle\langle bc\rangle,\]
(for functions $a$, $b$, $c$, and $d$ of zero mean) does not exist, as one learns as soon as one computes
\[\langle\pbar_1^4\rangle_{\w,q}\approx\frac{11\,\pi^4}{256\,h^4}\neq \frac{3\,\pi^4}{256\,h^4}\approx3\langle\pbar_1^2\rangle^2_{\w,q}.\]
Here, the symbol $\approx$ indicates the coefficient of the fastest-increasing term of each of these as $q\to 1$, and $q=e^{-h}$. (Of course one does have that $\langle\pbar_1\rangle_{\w,q}=0$.) These numbers seem to be a consequence of the lack of normal convergence to the limit shape for the distribution of the pillowcase weights, which impedes the existence of a Central Limit Theorem (see Remark \ref{rmk:noCLT}). For these reasons, the possibility of an analysis closely parallel to what was done for the volumes of the moduli spaces of Abelian differentials by A. Eskin and A. Okounkov \cite{branchedcoverings} seems rather unlikely.

\chapter
[A formula for near-involutions]{A formula for the characters of near-involutions}
\label{mychapter}
Recall from Section \ref{howtocomputethevolume} that we are interested in the asymptotics of the expectation of
\[\g_\nu(\lambda)=\frac{\f_{(\nu,2,2,\dots,2)}(\lambda)}{\f_{(2,2,\dots,2)}(\lambda)}.\]
with respect to the distribution induced by the pillowcase weights.
In this chapter, we intend to prove the following formula
\begin{equation}\label{eq:formulaforg}
\g_\nu(\lambda)=\frac{2^{|\nu|/2}}{\mathfrak z(\nu)}\sum_\mu\sigma_\mu\,\chi^\mu(\nu)\,\s_a(\alpha)\,\s_b(\beta),
\end{equation}
where the sum is taken over all balanced partitions $\mu$ of size $|\mu|=|\nu|$ whose Young diagram is completely contained inside the Young diagram of $\lambda$, $(\alpha,\beta)$ and $(a,b)$ are the 2-quotients of $\lambda$ and $\mu$ (to be defined below), and $\sigma_\mu$ equals $-1$ to the power of half the number of odd parts in $\mu$.

This will follow immediately from formula \eqref{eq:nearinvolutionformula}, whose proof we develop throughout the chapter, and the fact that
\[\g_\nu(\lambda)=\frac{2^{|\nu|/2}(|\lambda|/2)!}{\mathfrak z(\nu)\left(\frac{|\lambda|-|\nu|}2\right)!}\frac{\chi^\lambda(\nu,2,2,\dots,2)}{\chi^\lambda(2,2,\dots,2)}.
\]

In the final section of the chapter, we prove that the expectations
\[\langle\s_a(\alpha)\s_b(\beta)\rangle_{\w,q}\]
are quasimodular forms, which greatly eases their practical computation.
\section{2-quotients}\label{twoquotientssection}
We need to define the $2$-quotients of a partition and some of their properties. We mainly follow \cite[Excercise 7.59]{stanley}, \cite[Example I.1.8]{macdonald}.

\paragraph{Modified Frobenius coordinates and their geometric interpretation.}

Let $\lambda$ be a partition and let $\xi_i=\lambda_i-i+\frac12 \in \Z+\frac12$ be its \emph{modified Frobenius coordinates}\footnote{This use of the term `modified Frobenius coordinates' is non-standard.}. If we take $\lambda_i$ to equal 0 for all $i$ greater than the number of non-zero parts of $\lambda$, we can visualize these coordinates as follows. We rotate the picture of the Young diagram of $\lambda$ by $135^\circ$ counter-clockwise, and we rescale it by $\sqrt2$. We then place black pebbles at each point $\xi_i$ on the reversed $x$-axis, and white ones in the remaining half-integers. For example, if $\lambda=(5,4,4,2)$ we start with the usual Young diagram,
\[\yng(5,4,4,2)\]
and then we rotate to get the following diagram.
\begin{center}
\includegraphics{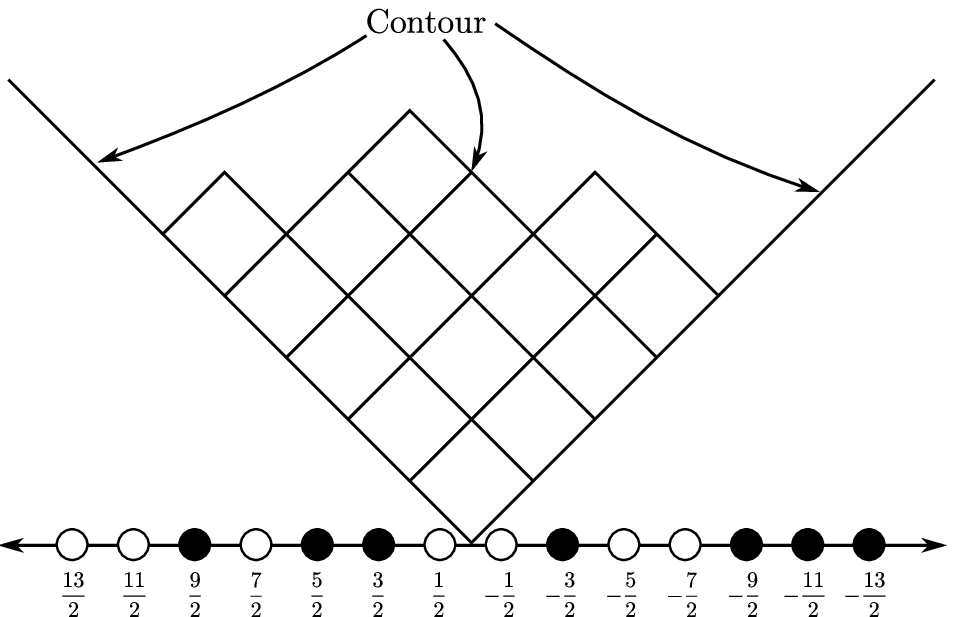}
\end{center}
In the diagram we have also drawn the rotated axes, which extend to infinity. The \emph{contour} $L_\lambda$ of $\lambda$ is the graph of the continuous function that describes the piecewise-linear curve starting (in the picture) with the diagonal axis on the left, running along the top border (or rim) of the diagram of the partition and ending with the diagonal axis on the right. It is clear that the contour of $\lambda$ has slope $-1$ wherever there is a white pebble, and has slope $+1$ in the intervals where  there is a black pebble. The association of sequences of pebbles to a Young diagram is known as a \emph{Maya diagram}. Observe that in such a diagram there are always as many black pebbles to the left of zero as there are white pebbles to the right of zero.

If we assign the number 0 to the white pebbles and the number 1 to the black ones, we get a sequence $\{c_i\}_{i\in \Z}$, such that $c_i\in\{0,1\}$, $c_i=0$ for all $i\ll 0$, and $c_i=1$ for all $i\gg 0$. In the case of the example above the sequence equals:
\[\dots,0,0,0,1,0,1,1,0,0,1,0,0,1,1,1,\dots\]
 The sequence completely determines the partition $\lambda$. It coincides with the sequence one gets if one assigns the number 1 to each vertical segment  and the number 0 to each horizontal segment in the contour, as follows:
\begin{center}
\includegraphics{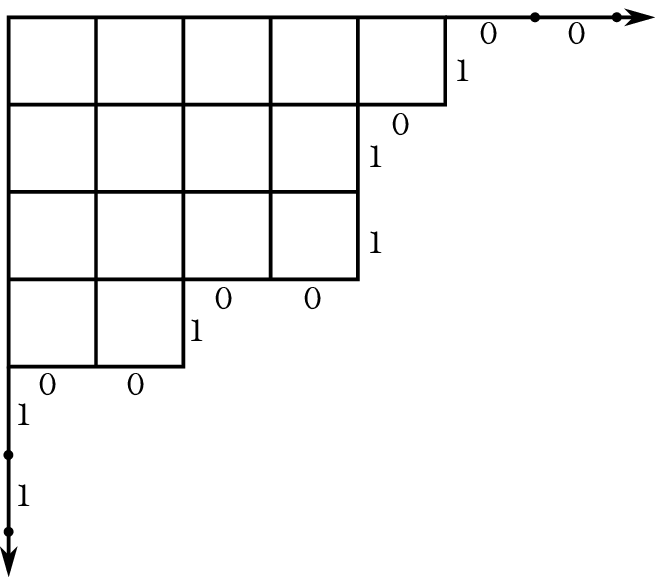}
\end{center}
Note that two sequences $\{a_i\}_{i\in\Z}$ and $\{b_i\}_{i\in\Z}$ of zeros and ones with the above properties determine the same partition if they are translates $a_i=b_{i-n}$ of each other, for some $n\in \Z$. Both points of view ---Maya diagrams and binary sequences--- are equivalent. In what follows, we will prefer the language of pebbles placed under the rotated diagram.

\paragraph{Pebble operations and their relation to strips and hooks.}

Consider how the sequence $\{c_i\}$ changes when one adds a strip of length $p$ to the rim of the partition. (Strips are defined in Section \ref{murnaghannakayamarule}.) Adding a strip is equivalent to moving a pebble to the left, and interchanging it with a white pebble. Here is an example: we add a 7-strip to the Young diagram of the partition $(5,4,4,2)$ to obtain $(6,6,5,5)$; compare with the diagram above.
\begin{center}
\includegraphics{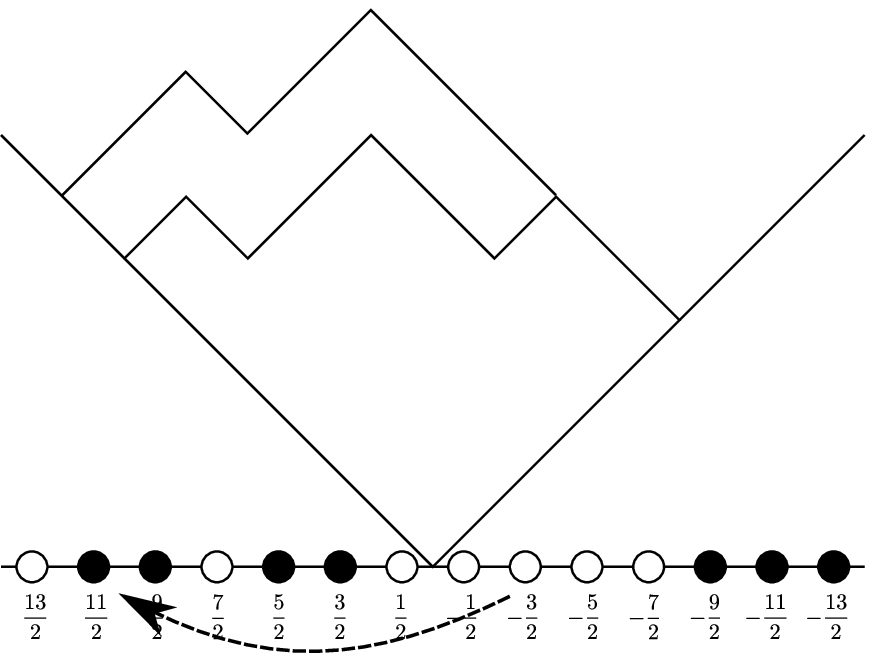}
\end{center}
In the picture, the black pebble that was originally located at the site $-\frac32$ has been moved to the site $-\frac32+7=\frac{11}2$.
The left-most cell of the strip is always posed on a site that originally had a downward slope (i.e., a white pebble) and it becomes an upward slope. The right-most cell of the strip is always posed on a spot with upward slope, and it turns it into a site with downward slope. All the sites in-between preserve their original slopes.

It is also easy to identify, given the sequence of pebbles of a partition, the sites from where it is possible to remove a $p$-strip: one only needs to look for a site with a black pebble such that in the site $p$ units to the right there is a white pebble. In the example above involving partition $(6,6,5,5)$, since we have a black pebble at the site $\frac{11}2$, we can remove strips of sizes 2, 5, 6, 7, 8, and 9, corresponding respectively to the white pebbles at the sites $\frac 72$, $\frac12$, $-\frac12$, $-\frac32$, $-\frac 52$, and $-\frac72$.

Note that there is an immediate correspondence between the strips we can remove and the hooks of the partition. (The hooks and their lengths are defined in Section \ref{murnaghannakayamarule}.) Indeed, if instead of joining the left-most and right-most cells of the strip along the rim, we look at the hook whose two ends correspond to these two cells, it is easy to see that the lengths of the hook and the strip are exactly the same. In the following diagram, we see the same strip as before, and the corresponding hook shaded in grey.
\begin{center}
\includegraphics{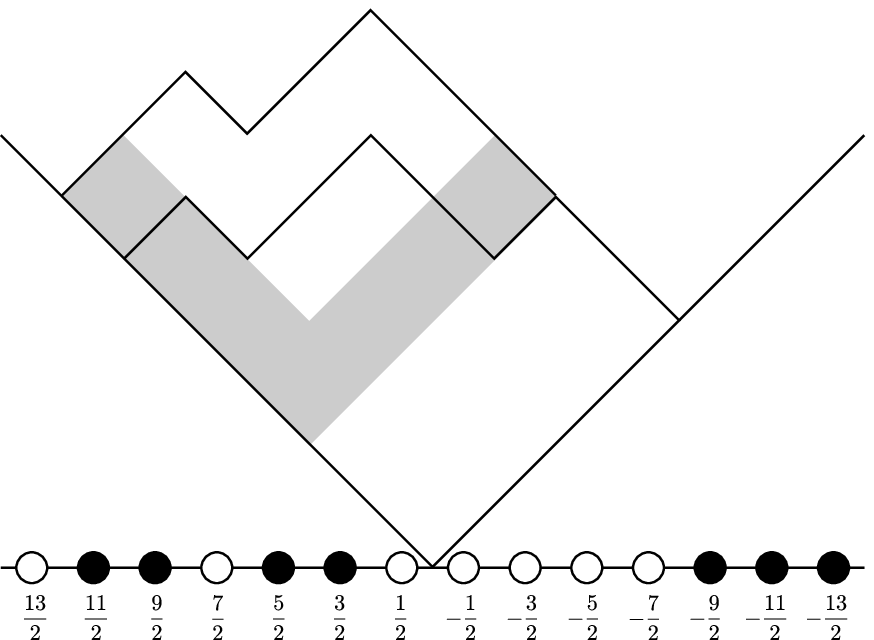}
\end{center}
Both the strip and the hook have length 7.

\paragraph{2-quotients.}

Given the sequence of black and white pebbles corresponding to any partition $\lambda$, we can split it into two sequences by picking the sites alternatingly. Namely, we assign to one of them all the pebbles on the sites $2n+\frac 12$, $n\in \Z$, and to the other, the pebbles on the sites $2n-\frac12$, $n\in \Z$. The resulting sequences determine two partitions $\alpha$ and $\beta$ which are together known as the \emph{2-quotient} $(\alpha,\beta)$ of $\lambda$, and they carry information on how $\lambda$ can be built by adjoining \emph{2-dominoes} $\yng(2)$, $\yng(1,1)$ and a 2-core, to be defined below. For example, in the case of the partition $(5,4,4,2)$, we split the pebbles as follows:
\begin{center}
\includegraphics{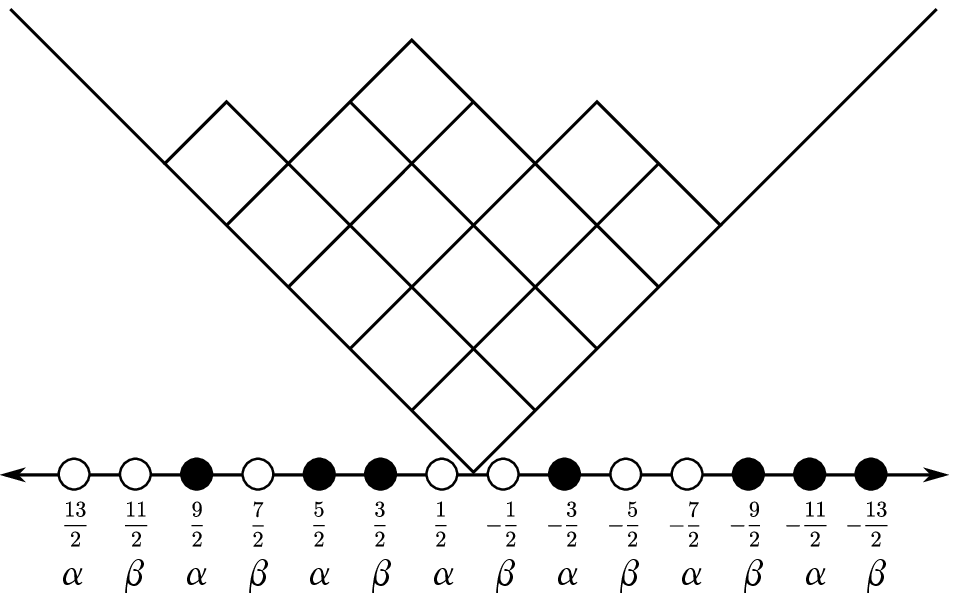}
\end{center}
and we obtain the following two sequences of black and white pebbles:
\begin{center}
\includegraphics{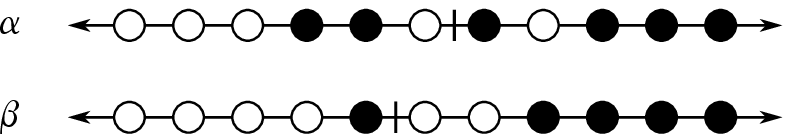}
\end{center}
(Here, the position of zero has been recorded with a vertical line. It is off-center in the sense that the number of black pebbles to the left of it does not equal the number of white pebbles to the right of it. We will discuss this below.)
These in turn correspond to the following partitions:
\[\alpha=\yng(2,2,1)\quad\textrm{and}\quad\beta=\yng(2).\]

Every time we remove a 2-domino (which is the same as a strip of length 2) from $\lambda$, we are removing a cell from one of the components $\alpha$ and $\beta$ of the 2-quotient. This is because we are exchanging two pebbles of opposite color which are exactly 2 units apart, and hence belong in the same component of the 2-quotient.

\paragraph{2-cores.}

In the process of removing 2-dominoes, we may get stuck before removing all the cells in the partition $\lambda$. The resulting partition is known as the \emph{2-core} and it is always shaped as a staircase, that is, it is of the form $(n,n-1,n-2,\dots,2,1)$, $n=1,2,\dots$. The following are the first few 2-cores:
\[\yng(1)\qquad\yng(2,1)\qquad\yng(3,2,1)\qquad\yng(4,3,2,1)\]
It is easy to see, for example, that the 2-core of the partition $(5,4,4,2)$ is $(1)$. For instance, we may remove its 2-dominoes in the order suggested by the following diagram:
\[\young(\hfil6633,7522,7511,44)\]
These moves in turn correspond to the following removals in the 2-quotients:
\[\alpha=\young(63,52,4)\quad \beta=\young(71)\]
In terms of pebble exchanges, the picture is as follows:
\begin{center}
\includegraphics{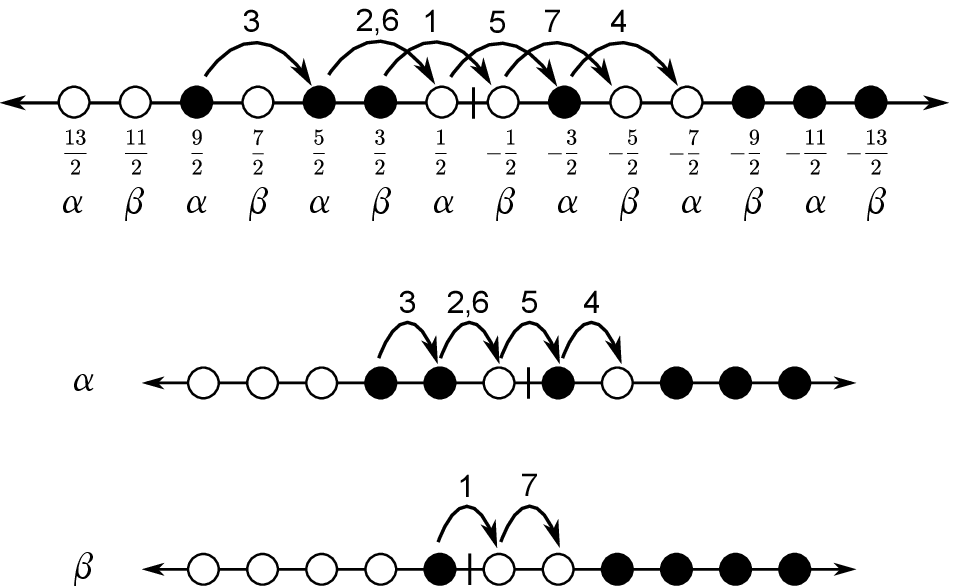}
\end{center}
As we noted before, the vertical bar denoting the original position of zero with respect to the sequences of pebbles for $\alpha$ and $\beta$ is off-centered; this is just a consequence of the presence of the non-trivial 2-core $\yng(1)$. A larger 2-core produces a larger translation of the components of the 2-quotient.

\paragraph{Balanced partitions.}

Recall that a partition is \emph{balanced} when it can be completely constructed by adjoining 2-dominoes. Equivalently, a partition is balanced exactly when its 2-core is empty. This is also equivalent to the components of the 2-quotient being centered (i.e., having the same number of black pebbles to the left of zero as white pebbles to the right of zero). Examples of balanced partitions are $(2)$, $(4,4,2)$, and $(3,3,2,1,1)$. As we saw earlier, $(5,4,4,2)$ and $(3,2,1)$ are not balanced.

It follows from the Murnaghan-Nakayama rule (see Section \ref{murnaghannakayamarule}) that the character $\chi^\lambda(2,2\dots,2)$ vanishes automatically when $\lambda$ is not balanced.

We record the following important facts about balanced partitions:
\begin{lem}\label{LemTwoQuotients}
Let $\lambda$ be a balanced partition and let $\alpha$ and $\beta$ be the components of its 2-quotient. Then
\begin{enumerate}[i.]
 \item \label{LemAlphaPlusBetaIsLambdaOverTwo} $|\alpha|+|\beta|=|\lambda|/2$
 \item \label{LemQuotientHooksAreHalvesOfEvenHooks} The set $\{h_\square/2:\textrm{$h_\square$ is the length of the hook of the cell $\square\in \lambda$ and $h_\square$ is even}\}$ equals the set $\{\textrm{hook lengths of $\alpha$}\}\cup\{\textrm{hook lengths of $\beta$}\}$.
 \item \label{LemHalfOfHooksAreEven} Exactly half of the hook lengths of $\lambda$ are even.
 \item \label{LemOddAndEvenHooksPositions}The two pebbles encoding a hook of even length (i.e., the pebbles laying at its extremes) are both in the same component $\alpha,\beta$ of the 2-quotient. In the case of a hook of odd length, they are in opposite components of the 2-quotient.
\end{enumerate}
\end{lem}
All of these points are clear from the description given above.

\section{Domino tableaux}
This section contains a few results about shapes made up of 2-dominoes $\yng(2)$ that we will need later on. 
\subsection{The classical case}\label{sec:classicaldominotableaux}
In this section, we prove a classical formula for the characters of irreducible representations of the symmetric group evaluated at involutions, that is, elements of cycle type $(2,2,\dots,2)$. Our first observation is that $\chi^\lambda(2,2,\dots,2)$ vanishes unless $\lambda$ is balanced. This is because the sum appearing in the Murnaghan-Nakayama rule (see Section \ref{murnaghannakayamarule}) is empty if $\lambda$ cannot be built by adjoining 2-dominoes.

Let $\lambda$ be a balanced partition, $o$ be the number of \emph{odd} parts of $\lambda$, and $\sigma_\lambda=(-1)^{o/2}$. Note that $o$ must be even because $\lambda$ is balanced so item \eqref{LemHalfOfHooksAreEven} of Lemma \ref{LemTwoQuotients} holds. Let $(\alpha,\beta)$ be the 2-quotient of $\lambda$. Then we have
\begin{equation}\label{eq:dominoformula}
\chi^\lambda(2,2,\dots,2)=\sigma_\lambda \binom{|\lambda|/2}{|\alpha|}\dim \alpha \dim \beta.
\end{equation}
We now explain how this is an easy consequence of the Murnaghan-Nakayama rule (see Section \ref{murnaghannakayamarule}). This rule says that
\[\chi^\lambda(2,2,\dots,2)=\sum_\eta (-1)^{\height \eta}\]
where the sum is over all ordered 2-domino tilings of $\lambda$ starting at the origin (i.e., $(2,2,\dots,2)$-strip decompositions), and $\height \eta$ denotes the height of $\eta$, as defined in Section \ref{murnaghannakayamarule}.

First note that all tilings $\eta$ of $\lambda$ by 2-dominoes have the same height modulo 2, and it is congruent with $o/2$. To see this, observe the following. The 2-dominoes placed horizontally have height 0, so they are irrelevant. On the other hand, the length of each row of the Young diagram of $\lambda$ is congruent modulo 2 to the number of vertically-placed dominoes it intersects. Since each of these vertical dominoes spans two rows, we see that the number of vertical dominoes is congruent to $o/2$ modulo 2, so the height is as well. This determines the sign $\sigma_\lambda$ of the value of the character $\chi^\lambda(2,2,\dots,2)$.

To count the tilings $\eta$, observe that their construction is equivalent to constructing the components $\alpha$ and $\beta$ of the 2-quotient, one cell at a time. An obvious consequence of the Murghan-Nakayama rule for $\chi^\alpha(1,1,\dots,1)=\dim \alpha$ is that there are $\dim \alpha$ ways to construct $\alpha$ one-cell at a time, and similarly there are $\dim \beta$ ways to construct $\beta$. Among the total $|\alpha|+|\beta|$ cells that should be added, we have to decide in which order we will introduce the ones for $|\alpha|$ and the ones for $|\beta|$. The count of their `internal ordering' will be taken care of by $\dim \alpha$ and $\dim \beta$, but we get a binomial factor corresponding to the number of choices we have for assigning, among the total $|\alpha|+|\beta|$ turns, those that correspond to cells of $\alpha$ (and the rest will be for cells of $\beta$):
\[\binom{|\alpha|+|\beta|}{|\alpha|}=\binom{|\lambda|/2}{|\alpha|}.\]
For the last equality, we have used item \eqref{LemAlphaPlusBetaIsLambdaOverTwo} of Lemma \ref{LemTwoQuotients}. Thus the total number of 2-dominio tilings of $\lambda$ is
\[\binom{|\lambda|/2}{|\alpha|}\dim\alpha\dim\beta.\]
This finishes the proof of formula \eqref{eq:dominoformula}.

Note that, using the fact that $\dim \alpha$ and $\dim \beta$ are given by the hook formula \eqref{hookformula}, together with item \eqref{LemQuotientHooksAreHalvesOfEvenHooks} of Lemma \ref{LemTwoQuotients}, we get
\begin{equation}\label{dimensionsofquotients}
\dim\alpha\dim\beta=\frac{|\alpha|!|\beta|!}{\prod\textrm{halves of the even hook lengths of $\lambda$}}.
\end{equation}

\subsection{The case of skew tableaux}
We will need a slight generalization of formula \eqref{eq:dominoformula}. We let $\lambda$ and $\mu$ be two partitions, such that the Young diagram of $\mu$ is contained inside the diagram of $\lambda$, and we denote by $\lambda/\mu$ the associated skew diagram; see Section \ref{murnaghannakayamarule} for an example.

We define the character $\chi^{\lambda/\mu}$ to be the obvious generalization of the Murnaghan-Nakayama rule to this case. Namely, we define an $\eta$-strip decomposition of $\lambda/\mu$ to be a sequence of partitions $\nu^0=\mu,\nu^1,\nu^2,\dots,\nu^{k-1},\nu^k=\lambda$ such that $\nu^i/\nu^{i-1}$ is a strip of length $\eta_i$, $i=1,2,\dots,\ell(\eta)$. The height of an $\eta$-decomposition of $\lambda/\mu$ is the sum of the heights of the strips $\nu^i/\nu^{i-1}$, $i=1,2,\dots,\ell(\eta)$. Finally,
\[\chi^{\lambda/\mu}(\eta)=\sum_\nu (-1)^{\height \nu},\]
where the sum is taken over all $\eta$-decompositions $\nu$ of $\lambda/\mu$.

Let $\lambda$ and $\mu$ be balanced and let $(\alpha,\beta)$ and $(a,b)$ be their 2-quotients. We define the \emph{2-quotient} of $\lambda/\mu$ to be the pair of skew diagrams $(\alpha/a,\beta/b)$.

We can define the \emph{dimension} $\dim (\lambda/\mu)$ to be the value of $\chi^{\lambda/\mu}(1,1,\dots,1)$.
Then practically the same proof as for \eqref{eq:dominoformula} shows that
\begin{equation}\label{eq:skewdominoformula}
\chi^{\lambda/\mu}(2,2,\dots,2)=\sigma_\lambda\sigma_\mu \binom{|\lambda/\mu|/2}{|\alpha/a|}\dim (\alpha/a)\dim(\beta/b).
\end{equation}
This formula is still valid for non-balanced partitions $\lambda$ and $\mu$, with the caveat that the character $\chi^{\lambda/\mu}(2,2,\dots,2)$ vanishes when the 2-core of $\lambda$ is different from the 2-core of $\mu$.

\section{Formula for the dimension of skew diagrams}
In this section we will state and prove a formula due to A. Okounkov and G. Olshanski \cite{okounkovolshanski} for the dimension $\dim (\lambda/\mu)$ of a skew diagram. We follow their exposition.

Let
\[(x\downharpoonright k)=x(x-1)(x-2)\cdots(x-k+1).\]

We first define the \emph{shifted Schur polynomials} in $n$ variables, indexed by the partition $\mu$, as the following ratio of two $n\times n$ determinants:
\[\s_\mu(x_1,\dots,x_n)=\frac{\det [(x_i+n-i\downharpoonright\mu_j+n-j)]} {\det[(x_i+n-i\downharpoonright n-j)]},\quad 1\leq i,j\leq n.\]
These polynomials satisfy \cite{okounkovolshanski} a stability condition
\[\s_\mu(x_1,\dots,x_n,0)=\s_\mu(x_1,\dots,x_n),\]
which allows us to take inverse limits, just as in the definition of symmetric functions (see Section \ref{symmetricfunctions}). The resulting objects are known as \emph{shifted Schur functions} and we will denote them by $\s_\mu(x_1,x_2,\dots)$. These are also sometimes called \emph{Frobenius Schur functions} (see for example \cite{frobeniusschur}). The algebra $\Lambda^*$ they generate coincides with the algebra generated by the shifted-symmetric power functions $\p_k$.

We want to prove the following relation \cite{okounkovolshanski}:
\begin{equation}\label{eq:dimensionofskewdiagram}
\frac{\dim(\lambda/\mu)}{\dim\lambda}=\frac{\s_\mu(\lambda)}{ ( |\lambda|\downharpoonright|\mu|)}.
\end{equation}

We choose to reproduce the second proof of \eqref{eq:dimensionofskewdiagram} given in \cite{okounkovolshanski} because it is more self-contained and elementary than the two others. Let $l=|\lambda|$ and $k=|\mu|$. We have, by the branching rule, orthonormality of the characters, and Frobenius reciprocity (see Section \ref{branchingrule}),
\[
\dim(\lambda/\mu)=\langle\res\chi^\lambda,\chi^\mu\rangle_{S(k)} =\langle\chi^\lambda,\ind_{S(l)}\chi^\mu\rangle_{S(l)}.
\]
The \emph{characteristic map} $\ch$, defined in \cite[Section I.7]{macdonald}, assigns to each character of the symmetric group (of any order) a symmetric function. It is defined, for a character $\chi$ of $S(n)$, by
\[\ch \chi = \sum_{|\rho|=n}\frac {\chi(\rho)}{\mathfrak z(\rho)} p_\rho,\]
where $p_\rho=p_{\rho_1}\cdots p_{\rho_k}$ is the symmetric power function, the product of several instances of $p_j=\sum_{i\ge 1} x_i^j$.
In this correspondence, $\ch \chi^\lambda=s_\lambda$ is the (traditional) symmetric Schur function. On the other hand, the induction of $\chi^\mu$ to $S(k+1)$ equals the induction of $\chi^\mu\times\chi^{(1)}$ (i.e., the character of the product of the representation indexed by $\mu$ and the character of the 1-dimensional identity representation) from $S(k)\times S(1)$ to $S(k+1)$. It is a property of $\ch$ that
\[\ch \ind_{S(k+1)} \left(\chi^\mu\times\chi^{(1)}\right)=\ch \chi^\mu\cdot\ch \chi^{(1)}=s_\mu p_1.\]
After $l-k$ similar steps, we get
\[\ch \ind_{S(l)} \chi^\mu=s_\lambda p_1^{l-k}.\]
Taking the \emph{canonical inner product} in the algebra of symmetric functions, defined by
\[\left( s_\lambda,s_\mu \right)=\delta_{\lambda,\mu},\]
the map $\ch$ turns out to be an isometry, so that
\[\left\langle\chi^{\lambda},\ind_{S(l)}\chi^\mu\right\rangle_{S(l)}=
\left(\ch \chi^{\lambda},\ch\ind_{S(l)}\chi^\mu\right) =
\left(s_\lambda,s_\mu p_1^{l-k}\right). \]
If we restrict the symmetric functions to $n\geq l$ variables, we see by looking at the definition of $s_\lambda$ (see Section \ref{symmetricfunctions}) that we are simply looking for the coefficient of
\[x^{\lambda_1+n-1}_1 x^{\lambda_2+n-2}_2\cdots x_n^{\lambda_n}\]
in the expansion of
\begin{equation}\label{eq:polynomialinquestion}
(x_1+\cdots+x_n)^{l-k}\det \left(x_i^{\mu_j+n-j}\right)_{1\leq i,j\leq n}.
\end{equation}
By expanding \eqref{eq:polynomialinquestion}, we can rewrite it as
\[\sum_{s\in S(n)}\sgn(s)\sum_{r_1+\cdots+r_n=l-k}\frac{(l-k)!}{r_1!\cdots r_n!} \prod_{i=1}^n x_i^{\mu_{s(i)}+n-s(i)+r_i}\]
The coefficient we need is then
\begin{align*}
\sum_{s\in S(n)} & \sgn(s)\frac{(l-k)!}{\prod_{i=1}^n(\lambda_i-\mu_{s(i)}-i+s(i))!} \\
&=(l-k)!\det\left[\frac1{(\lambda_i-\mu_j-i+j)!}\right] \\
&=(l-k)!\det\left[\frac1{((\lambda_i+n-i)-(\mu_j+n-j))!} \right]\\
&=\frac{(l-k)!}{\prod_i(\lambda_i+n-i)!} \det\left[\frac{(\lambda_i+n-i)!}{((\lambda_i+n-i)-(\mu_j+n-j))!}\right] \\
& =\frac{(l-k)!}{l!} \cdot \frac{l!\prod_{p<q}(\lambda_p-\lambda_q+q-p)}{\prod_i(\lambda_i+n-i)!} \cdot \frac{\det[(\lambda_i+n-i\downharpoonright \mu_j+n-j)]} {\prod_{p<q}(\lambda_p-\lambda_q+q-p)} \\
&=\frac{\dim \lambda}{(l\downharpoonright k)}\,\s_\mu(\lambda).
\end{align*}
Here we used the definition of $\s_\mu$ and formula \eqref{eq:hookalternative} for $\dim \lambda$. This completes the proof of formula \eqref{eq:dimensionofskewdiagram}.

\section{Formula for characters of near-involutions}\label{sec:nearinvolutions}
Let $\lambda$ and $\nu$ be balanced partitions such that the Young diagram of $\nu$ is completely contained inside $\lambda$. We want to prove the following relation:
\begin{equation}\label{eq:nearinvolutionformula}
\frac{\chi^\lambda(\nu,2,2,\dots,2)}{\chi^\lambda(2,2,\dots,2)} = \frac{(|\lambda/\mu|/2)!}{(|\lambda|/2)!} \sum_{\mu}\sigma_\mu\chi^\mu(\nu) \,\s_a(\alpha)\,\s_b(\beta),
\end{equation}
where the sum is taken over all \emph{balanced} partitions $\mu$ of size $|\nu|$ whose Young diagram is completely contained in the Young diagram of $\lambda$, $(\alpha, \beta)$ and $(a,b)$ are the 2-quotients of $\lambda$ and $\mu$, respectively, and $\sigma_\mu$ is defined as in Section \ref{sec:classicaldominotableaux}.

To prove equation \eqref{eq:nearinvolutionformula}, we first observe that the Murnaghan-Nakayama rule (see Section \ref{murnaghannakayamarule}) implies that
\[\chi^\lambda(\nu,2,2,\dots,2) = \sum_{|\mu|=|\nu|}\chi^{\mu}(\nu)\chi^{\lambda/\mu}(2,2,\dots,2)\]
The sum is over all partitions $\mu$ of size $|\nu|$ whose diagram is completely contained inside the diagram of $\lambda$.
Clearly, $\chi^{\lambda/\mu}(2,2,\dots,2)$ vanishes unless $\mu$ is balanced, so all sums from this point on will be over balanced partitions $\mu$ of size $|\nu|$.
To this expression we apply formula \eqref{eq:skewdominoformula}, and we apply formula \eqref{eq:dominoformula} to the denominator, to get
\begin{equation*}
\frac{\chi^\lambda(\nu,2,2,\dots,2)}{\chi^\lambda(2,2,\dots,2)} =\sum_\mu B_\mu\,\sigma_\mu\,\chi^\mu(\nu) \,\frac{\dim(\alpha/a)}{\dim\alpha}\cdot\frac{\dim(\beta/b)}{\dim\beta}.
\end{equation*}
Here, $(\alpha,\beta)$ and $(a,b)$ are the 2-quotients of $\lambda$ and $\mu$, respectively, and \[B_\mu=\binom{|\lambda/\mu|/2}{|\alpha/a|}\Big/\binom{|\lambda|/2}{|\alpha|}.\] Now we can apply formula \eqref{eq:dimensionofskewdiagram} to each of the quotients of dimensions, and we get exactly formula \eqref{eq:nearinvolutionformula} because, by item \eqref{LemAlphaPlusBetaIsLambdaOverTwo} of Lemma \ref{LemTwoQuotients},
\[\frac{|\lambda|}2-|\alpha|=|\beta|\quad\textrm{and} \quad\frac{|\lambda/\mu|}{2}-|\alpha/a|=|\beta/b|.\]
\section[Quasimodularity of the expectations]{Quasimodularity of the expectations of shifted symmetric functions
on the 2-quotients}
\label{sec:quasimodularitytwoquotients}
In order for formula \eqref{eq:formulaforg} to be useful, we must prove that there is a reasonable way to compute the expectations of the terms involved. In this section we prove, in close parallel to the work of A. Eskin and A. Okounkov \cite{pillow}, that the expectations
\[\langle \s_a(\alpha)\s_b(\beta)\rangle_{\w,q}\]
are quasimodular forms. Here, the expectation is defined by
\[\langle f(\alpha,\beta)\rangle_{\w,q}=\sum_\lambda \frac{q^{|\lambda|}\w(\lambda)}{Z(q)}f(\alpha,\beta), \]
where sum is taken over all balanced partitions, $(\alpha,\beta)$ is the 2-quotient of $\lambda$, and
\[Z(q)=\sum_\lambda q^{|\lambda|}\w(\lambda).\]

It was proven by A. Okounkov and G. Olshanski \cite{okounkovolshanski} that
the algebra generated by the shifted Schur functions $\s_\mu$ coincides with the algebra generated by the shifted power functions $\p_\rho$, whose definition we recall: for a positive integer $k$,
\begin{equation}\label{eq:defofpone}
\p_k(\lambda)=\sum_{j=1}^\infty \left[\left(\lambda_j-j+\tfrac12\right)^k-\left(-j+\tfrac12\right)^k\right]+\left(1-\frac1{2^k}\right)\zeta(-k),
\end{equation}
where $\zeta$ denotes the Riemann zeta function. For general partitions $\rho$, the definition of $\p_\rho$ is
\begin{equation}\label{eq:defofptwo}
\p_\rho=\prod_i \p_{\rho_i}.
\end{equation}
So it is enough to prove that the expectations of the form
\[\langle \p_\mu(\alpha)\p_\eta(\beta)\rangle_{\w,q}\]
are quasimodular forms.

To do this, we will first look at the exponential generating function of these expectations, and then we will analyze this generating function.

We consider first the $(n,m)$-point function,
\begin{align*}
F(x_1,\dots,x_n;z_1,\dots,z_m)=
\sum_{k_1,\dots,k_n,\ell_1,\dots,\ell_m\in\Z+\frac12}
x_1^{k_1}\cdots x_n^{k_n} z_1^{\ell_1}\cdots z_m^{\ell_m}
\sum_\lambda \frac{ \w(\lambda)q^{|\lambda|}}{Z(q)},
\end{align*}
where the last sum is taken over all balanced partitions $\lambda$ with 2-quotient $(\alpha,\beta)$ such that
\[k_1,\dots,k_n\in\{\alpha_i-i+\frac12:i=1,2,\dots\}\]
and
\[\ell_1,\dots,\ell_m\in\{\beta_i-i+\frac12:i=1,2,\dots\}.\]
 Letting $u_1,\dots,u_n$ and $v_1,\dots,v_m$ be defined by $x_i=e^{u_i}$ and $z_i=e^{v_i}$, it is easy to verify that
 \begin{multline}\label{eq:generatingfunctionexpectationsps}
 F(e^{u_1},\dots,e^{u_n};e^{v_1},\dots,e^{v_m})=\\
 \sum_{i_1,\dots,i_n,j_1,\dots,j_m\in\mathbb Z_{\geq0}} \frac{u_1^{i_1}}{i_1!}\cdots \frac{u_n^{i_n}}{i_n!}\frac{v_1^{j_1}}{j_1!}\cdots \frac{v_m^{j_m}}{j_m!}
 \langle \p_{i_1}(\alpha)\cdots\p_{i_n}(\alpha)\p_{j_1}(\beta)\cdots\p_{j_m}(\beta)\rangle_{\w,q}.
\end{multline}
 This follows from the fact that the exponential generating function for the shifted power functions $\p_k(\mu)$ is (cf. \cite[eq. (0.18)]{blochokounkov}, \cite[eq. (2.8)]{branchedcoverings})
\[\sum_k\frac{x^k}{k!}\p_k(\mu)=\sum_{i} e^{\left(\mu_i-i+\frac12\right)x},\]
and we apply this identity separately for $\alpha$ and for $\beta$.
This identity is true for $x$ with positive real part; this must be taken into account when computing series expansions.

We define the \emph{Jacobi theta function} by
\[\vartheta(x)=\vartheta(x,q)=(q^{1/2}-q^{-1/2})\prod_{i=1}^\infty\frac{(1-q^ix)(1-q^i/x)}{(1-q^i)^2}.\]
(See Section \ref{jacobithetas} for a discussion.)

\begin{lem}\label{lem:formulafornmpointfunction}
The $(n,m)$-point function is given by
\begin{multline*}
F(x_1^2,\dots,x_n^2;z_1^2,\dots,z_m^2)=
\frac1{2^{n+m}}\sqrt{\frac{x_1\cdots x_n}{z_1\cdots z_m}}
\sum_{s,s',\sigma,\sigma'}\frac{\left(-\sqrt{-1}\right)^{\ell(s,s',\sigma,\sigma')}}
{\prod_i\vartheta\!\left(\frac{s_ix_i}{\sigma_i}\right)
\prod_j\vartheta\!\left(\frac{s_j'z_j}{\sigma_j'}\right)}\times\\
[y_1^0\cdots y_n^0w_1^0\cdots w_m^0]
\sqrt{\prod_{i}
\frac{\vartheta(-\sigma_iy_i)\vartheta(s_ix_iy_i)}
{\vartheta(\sigma_iy_i)\vartheta(-s_ix_iy_i)}
\prod_j
\frac{\vartheta(-\sigma'_j w_j)\vartheta(s'_jz_jw_j)}
{\vartheta(\sigma'_jw_j)\vartheta(-s'_jz_jw_j)}
}\,\times\\
\prod_{i<j}
\frac{
\vartheta\!\left(\frac{\sigma_iy_i}{\sigma_jy_j}\right)
\vartheta\!\left(\frac{s_ix_iy_i}{s_jx_jy_j}\right)
}
{
\vartheta\!\left(\frac{s_ix_iy_i}{\sigma_jy_j}\right)
\vartheta\!\left(\frac{\sigma_iy_i}{s_jx_jy_j}\right)
}
\prod_{i<j}
\frac{
\vartheta\!\left(\frac{\sigma'_iw_i}{\sigma'_jw_j}\right)
\vartheta\!\left(\frac{s'_iz_iw_i}{s'_jz_jw_j}\right)
}
{
\vartheta\!\left(\frac{s'_iz_iw_i}{\sigma'_jw_j}\right)
\vartheta\!\left(\frac{\sigma'_iw_i}{s'_jz_jw_j}\right)
}
\prod_{i,j}
\frac{
\vartheta\!\left(\frac{s_ix_iy_i}{s_j'z_jw_j}\right)
\vartheta\!\left(\frac{\sigma_iy_i}{\sigma_j'w_j}\right)
}
{
\vartheta\!\left(\frac{\sigma_iy_i}{s'_jz_jw_j}\right)
\vartheta\!\left(\frac{s_ix_iy_i}{\sigma_j'w_j}\right)
},
\end{multline*}
where the sum is taken over all functions
\[
\begin{array}{c}
s,\sigma:\{1,2,\dots,n\}\to\{+1,-1\},\\
s',\sigma':\{1,2,\dots,m\}\to\{+1,-1\},
\end{array}
\]
and \[\ell(s,s',\sigma,\sigma')=|s^{-1}(-1)|+|(s')^{-1}(-1)|+|\sigma^{-1}(-1)|+|(\sigma')^{-1}(-1)|\]
is the number of points at which these functions equal $-1$.
The expansion is valid in the domain
\begin{multline*}
|y_n/q|>|z_1w_1|>|w_1|>\cdots>|z_mw_m|>|w_m|>\\
|x_1y_1|>|y_1|>\cdots>|x_ny_n|>|y_n|>1.
\end{multline*}
\end{lem}

\begin{proof}
In the context of the Fock space representation $\Lambda^{\frac{\infty}2}_0$ (see Section \ref{fock} for definitions and notations), we consider the vertex operator
\[\pcop =\exp\left(\sum_{n>0}\frac{\alpha_{-2n-1}}{2n+1}\right) \exp\left(-\sum_{n>0}\frac{\alpha_{2n+1}}{2n+1}\right).\]
In the paper \cite{pillow}, $\pcop$ is called the \emph{pillowcase operator} because we have
\begin{thm}[\protect{\cite[Theorem 4]{pillow}}]
The diagonal matrix elements of $\pcop$ are as follows:
\[\langle\pcop \,v_\lambda,v_\lambda\rangle=
\w(\lambda).\]
\end{thm}
We omit the proof of this theorem. The form of $\pcop$ was probably guessed using the tools developed by A. Okounkov and N. Reshetikhin \cite[Section 2.2.6]{okounkovreshetikhin}.

We define the operators
\[\psi_0(z)=\sum_{k\in2\Z+\frac12}z^k\psi_k,\quad \psi^*_0(z)=\sum_{k\in2\Z+\frac12}z^{-k}\psi^*_k,\]
and
\[\psi_1(z)=\sum_{k\in2\Z-\frac12}z^k\psi_k,\quad \psi^*_1(z)=\sum_{k\in2\Z-\frac12}z^{-k}\psi^*_k.\]
Note that, if $(\alpha,\beta)$ is the 2-quotient of the partition $\lambda$,
\[[x^ky^0]\psi_0(xy)\psi^*_0(y)v_\lambda=\left\{\begin{array}{ll}
v_\lambda, &\textrm{if $k \in \{2(\alpha_i-i+\frac12)+\frac12:i=1,2,\dots\}$},\\
0, & \textrm{otherwise.}
\end{array}\right.\]
where $[x^ky^l]$ stands for the operation of expanding the series with respect to $x$ and $y$ and finding the coefficient of $x^ky^l$. Similarly,
\[[x^ky^0]\psi_1(xy)\psi^*_1(y)v_\lambda=\left\{\begin{array}{ll}
v_\lambda, &\textrm{if $ k \in \{2(\beta_i-i+\frac12)-\frac12:i=1,2,\dots\}$,}\\
0, & \textrm{otherwise.}
\end{array}\right.\]
Define $Hv_\lambda=|\lambda|v_\lambda$.
Then $F(x_1^2,\dots,x_n^2;z_1^2,\dots,z_m^2)$ can be expressed as the following trace in $\Lambda^{\frac\infty2}_0$:
\begin{multline*}
F(x_1^2,\dots,x_n^2;z_1^2,\dots,z_m^2)=
\frac1{Z(q)}\sqrt{\frac{x_1\cdots x_n}{z_1\cdots z_m}}\,\times\\
[y_1^0\cdots y_n^0 w_1^0 \cdots w_m^0]\trace \!\big(q^H\pcop\,
\psi_0(x_1y_1)\psi^*_0(y_1)\cdots\psi_0(x_ny_n)\psi^*_0(y_n)
\\\times
\psi_1(x_1w_1)\psi^*_1(w_1)\cdots\psi_1(x_nw_m)\psi^*_1(w_m)\big).
\end{multline*}

In terms of the usual operators
\[\psi(z)=\sum_{k\in\Z+\frac12}z^k\psi_k,\quad
\psi^*(z)=\sum_{k\in\Z+\frac12}z^{-k}\psi^*_k,\]
we have
\begin{align*}
&\psi_0(z)=\frac12\left(\psi(z)-\sqrt{-1}\,\psi(-z)\right) , &&\psi^*_0(z)=\frac12\left(\psi^*(z)-\sqrt{-1}\,\psi^*(-z)\right),\\
&\psi_1(z)=\frac12\left(\psi(z)+\sqrt{-1}\,\psi(-z)\right) , &&\psi^*_1(z)=\frac12\left(\psi^*(z)+\sqrt{-1}\,\psi^*(-z)\right).
\end{align*}

Thus $F(x_1^2,\dots,x_n^2;z_1^2,\dots,z_m^2)$ is a sum of terms of the form
\begin{equation*}
\frac {\left(-\sqrt{-1}\right)^\ell}{2^{n+m}Z(q)}\sqrt{\frac{x_1\cdots x_n}{z_1\cdots z_m}}\,
[y_1^0\cdots y_n^0 w_1^0 \cdots w_m^0]\trace \!\big(q^H\pcop\,
A_1A_1^*\cdots A_nA_n^*B_1B_1^*\cdots B_mB_m^*\big)
\end{equation*}
where
\begin{align*}
&A_i=\psi(s_ix_iy_i),  &&A_i^*=\psi^*(\sigma_iy_i),\\
&B_i=\psi(s'_iz_iw_i), &&B_i^*=\psi^*(\sigma'_iw_i),
\end{align*}
$s_i,s_i',\sigma_i,\sigma_i'\in\{+1,-1\}$, and $\ell$ is the number of them that equals $-1$.
From the boson-fermion correspondence, we have the following formula \cite[Theorem 14.10]{kac}:
\begin{multline*}
\psi(xy)\psi^*(y)=\frac{1}{x^{1/2}-x^{-1/2}}\times\\
\exp\left(\sum_{n>0}\frac{(xy)^n-y^n}n\,\alpha_{-n}\right)
\exp\left(\sum_{n>0}\frac{y^{-n}-(xy)^{-n}}n\,\alpha_n\right).
\end{multline*}
Note that the factor $(x^{1/2}-x^{-1/2})^{-1}$ at the front can be checked comparing the coefficients of the vacuum vector $v_{\emptyset}$ on both sides of the identity when the operator is applied to $v_\emptyset$. Adjusting the signs of $x$ and $y$, we have all products $A_iA_i^*$ and $B_iB_i^*$ covered with this formula.

Now, these operators factor through the decomposition
\[\Lambda^{\frac\infty2}_0V=\bigotimes_n\bigoplus_k\alpha_{-n}^kv_\emptyset,\]
and in each of these factors we have, for all $R,S\in\C$,
\begin{equation}\label{eq:simplificationidentity}
\trace \left.e^{R\alpha_{-n}}e^{S\alpha_n} \right|_{\bigoplus_{k=1}^\infty\alpha^k_{-n}v_\emptyset}
=\frac1{1-q^n}\exp\frac{nRSq^n}{1-q^n}.
\end{equation}
This identity can be easily checked directly by expanding both sides. In our case,
\begin{align*}
R&=\frac1n\left(\delta_{\textrm{$n$ odd}} +\sum_{i=1}^n \left[(s_ix_iy_i)^n-(\sigma_iy_i)^n\right]+\sum_{i=1}^m \left[(s'_iz_iw_i)^n-(\sigma'_iw_i)^n\right]\right),\\
S&=\frac1n\left(-\delta_{\textrm{$n$ odd}}+\sum_{i=1}^n
\left[(\sigma_iy_i)^{-n}-(s_ix_iy_i)^{-n}\right]+\sum_{i=1}^m \left[(\sigma'_iw_i)^{-n}-(s'_iz_iw_i)^{-n}\right]\right).
\end{align*}

We take the product of \eqref{eq:simplificationidentity} for all $n$, and expand the series of the factor $(1-q^n)^{-1}$ in the exponent of the right hand side of equation \eqref{eq:simplificationidentity}. Then we exchange the order of summation, and use the series
\[\log(1-x)=-\sum_{k=1}^\infty\frac{x^k}{k}, \qquad \log\sqrt{\frac{1-x}{1+x}}=-\sum_{k=1}^\infty\frac{x^{2k-1}}{2k-1},\]
together with the definition of the theta function $\vartheta$ (see Section \ref{jacobithetas}) and the formula \cite[Section 3.2.4]{pillow}
\[Z(q)=\trace q^H\pcop=\prod_{i=1}^\infty\left(1-q^{2i}\right)^{-1/2}\]
to get the result of the statement of the lemma.

The domain of the expansion is justified in the same way as in \cite[Section 3.2.3]{pillow}: we have the estimate
\[(\pcop v_\lambda,v_\mu)\sim \textrm{const}\,\max(|\lambda|,|\mu|),\]
so the trace converges in the given domain.
\end{proof}

We proceed to complete the proof of the quasimodularity of the coefficients of the generating function $F(e^{u_1},\dots,e^{u_n};e^{v_1},\dots,e^{v_m})$.

In order to obtain the coefficient of $y_1^0\cdots y_n^0w_1^0\cdots w_m^0$ in the formula for $F$ given in Lemma \ref{lem:formulafornmpointfunction}, we want to take the integral
\begin{multline*}
\frac{1}{(2\pi i)^{m+n}}\oint_{|y_1|=c}\frac{dy_1}{y_1}\cdots\oint_{|y_n|=c}\frac{dy_n}{y_n}\times\\
\oint_{|w_1|=c}\frac{dw_1}{w_1}\cdots\oint_{|w_n|=c}\frac{dw_m}{w_m} \, F(e^{u_1},\dots,e^{u_n};e^{v_1},\dots,e^{v_m}).
\end{multline*}
To do this we introduce variables $a_i=\sqrt{-1}\log y_i$ and $b_i=\sqrt{-1}\log w_i$, which changes the domain of integration to the interval $[0,2\pi]$ on each of the variables. Then we plug in the identities
\begin{align*}
\log\frac z{\vartheta(e^z)}&=2\sum_{k\geq1}\frac{z^{2k}}{(2k)!}E_{2k}(q), \\
\log\frac{\vartheta(-e^{z})}{\vartheta(-1)}&=2\sum_{k\geq1}\frac{z^{2k}}{(2k)!}
\left[E_{2k}(q)-2^{2k}E_{2k}(q^2)\right],\\
\vartheta(-1)&=2i\left(\prod_{n\geq1}\frac{1+q^n}{1-q^n}\right)^2
=\frac{\eta(q^2)^2}{\eta(q)^4},
\end{align*}
where $\eta$ is the Dedekind eta function
\[\eta(q)=q^{1/24}\prod_{i\geq1}(1-q^i)\]
and $E_{2k}$ are the Eisentstein series
\[E_{2k}(q)=\frac{\zeta(1-2k)}2+\sum_{n=1}^\infty\left(\sum_{d|n}d^{2k-1}\right)q^n,\quad k=1,2,\dots,\]
into the formula of Lemma \ref{lem:formulafornmpointfunction}. These identities are easy to check directly.
Finally, we expand the series with respect to the variables $a_i,b_i, u_i,v_i$ and we integrate term by term.

What we get in the end is a series in the variables $u_i$ and $v_i$ whose coefficients are polynomials in the Eisenstein series $E_{2k}$. This means that the coefficients are quasimodular forms (see Section \ref{quasimodularforms}). Since this series is precisely the generating function \eqref{eq:generatingfunctionexpectationsps}, we are done.

\chapter{Analysis of the pillowcase distribution}
\label{chap:pillowcase}
In this chapter, we first find the limit shape induced by the pillowcase distribution and we analyze whether convergence is normal. This is the content of Section \ref{sec:limitshape}.

Then, we prove that the pillowcase distribution concentrates most of its measure near the partitions that have very similar 2-quotients. This requires us first to find an alternative formula for $\w(\lambda)$ in terms of hooks; we do this in Section \ref{sec:wformulahooks}. Then we do a variational argument in Section \ref{sec:variational} to obtain asymptotics for $\w(\lambda)$. We show in Section \ref{sec:meinardus} that the $\w$-probability of the set of partitions whose 2-quotients significantly differ vanishes asymptotically.

Finally, as an application of these results, in Sections \ref{sec:expectationasympt} and \ref{sec:nextterm} we analyze the first terms of the asymptotics of the expectation of $\g_\nu$ arising from formula \eqref{eq:formulaforg}.

\section{The limit shape}\label{sec:limitshape}
In this section we argue that the limit shape induced by the pillowcase weights coincides with the limit shape induced by the uniform distribution (see Section \ref{limitshapeuniform}). We will also make some comments on the multiplicativity of the mean, and on the possibility of a Central Limit Theorem, at the end of this section.

\paragraph{Moments of the limit shape.}

Let $\Omega:\R_+\to\R_+$ be a non-increasing continuous function such that
\[\int_{\R_+}\Omega(x)\,dx=1.\]
Let $\lambda^n$, $n=1,2,\dots$, be a sequence of partitions such that $|\lambda^n|\to\infty $ as $n\to\infty$.
We say that $\lambda^n$ \emph{converges} to $\Omega$ if
\[\limsup_{n\to\infty}\max_j\left|\frac1{\sqrt n}\,\lambda^n_j-\Omega\!\left(\frac j{\sqrt n}\right)\right|=0.\]

Let  $L_n$ be the contour of $\lambda^n$; this was defined in Section \ref{twoquotientssection}. Let $L_\infty:\R\to\R$ be the function whose graph corresponds to the graph of $\Omega$ rotated $45^\circ$ in the counter-clockwise direction (or $135^\circ$ if we think of $\Omega$ drawn upside down, resembling a very large partition). Then if $\lambda^n$ converge to $\Omega$, we also have that $L_n$ converge to $L^\infty$ in the topology of the supremum norm.

Recall that the shifted power functions $\p_k$ were defined in equation \eqref{eq:defofpone}. We want to give an interpretation of them as moments of the contours. From the definition of $\p_k$ it is clear that as $n\to\infty$,
\[\left|n^{-\frac{k+1}2}\p_k(\lambda^n)
-\int_{\R_+}\left(\left(\Omega(x)-x\right)^k-x^k\right)dx\right|\to0.\]
Now,
\begin{align*}
\int_{\R_+}&\left(\left(\Omega(x)-x\right)^k-x^k\right)dx
=\int_0^\infty\int_0^{\Omega(x)}\frac{d}{dy}(y-x)^k\,dy\,dx\\
&=\int_0^\infty\int_0^{\Omega(x)}k(y-x)^{k-1}\,dy\,dx=
\int_{-\infty}^\infty\int_{|s|}^{L_\infty(s)}ks^{k-1}\,dt\,ds\\
&=k\int_{-\infty}^\infty s^{k-1}\left(L_\infty(s)-|s|\right)\,ds
\end{align*}
where $s=y-x$ and $t=y+x$. The change of variables is done as per the following diagram, in which the curve $\Gamma$ depicts both the graph of $\Omega(x)$ and the graph of $L_\infty(s)$, and the integration domain is shaded in gray.

\begin{center}
\includegraphics{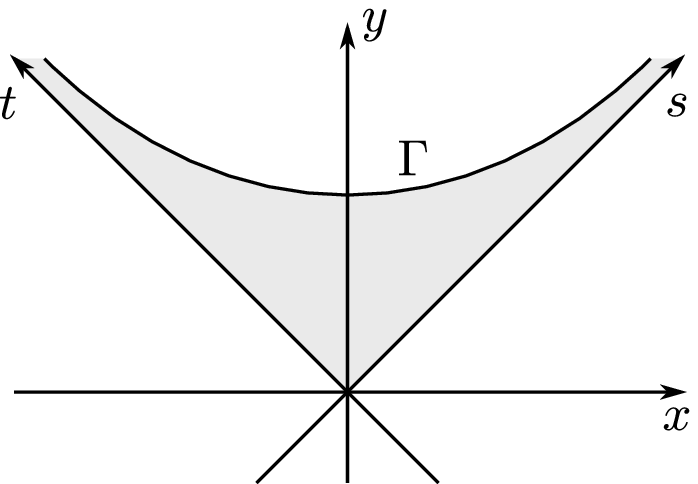}
\end{center}

In other words, if we consider $L(x)-|x|$ as a density with moments $\mu_n$, then
\[\p_k(\lambda^n)\to k\mu_{k-1}\quad\textrm{as}\quad n\to\infty.\]

\paragraph{Candidate for the limit shape.}

In order to determine what the candidate for the limit shape is, we will show that, in the $q\to1$ limit, the expectations of the shifted power functions
\begin{equation*}
\langle\p_k\rangle_{\w,q}
\end{equation*}
coincide with those of the uniform distribution. The functions $\p_k$ play the role of the moments of the limit shape, and because these moments $\mu_k$ do not grow very fast, they in fact determine it uniquely. A sufficient condition for this to be the case (see for example \cite[Chapter 30]{billingsley}) is that
\[\limsup_{k\to\infty}\frac{1}{2k}(\mu_{2k})^{\frac1{2k}}<\infty,\]
which is true in this case.

In the case of the uniform distribution, in which the weight of the partition $\lambda$ is $q^{|\lambda|}$, the rescaled limit of the exponential generating function of the numbers $\langle \p_k\rangle_{\mathrm{unif},q}$  is
\begin{equation}\label{eq:asymptoticsuniform}
F_{\mathrm{unif}}(e^{hu})\approx\frac{\pi}{h\sin\pi u},
\end{equation}
where $q=e^{-h}$, $h\to+0$, and $\approx$ means ``up to exponentially small terms''; compare with formula \eqref{eq:generatingfunctionexpectationsps}.
To prove that this is true, note that (in the notations of Sections \ref{fock} and \ref{sec:quasimodularitytwoquotients}, and using similar techniques)
\[F_{\mathrm{unif}}(x)=\frac{1}{Z}[y^0]\trace q^H\psi(xy)\psi^*(y)=\frac{1}{\vartheta(x)}.\]
Here, $\vartheta$ stands for the Jacobi theta function, whose defintion we recall:
 \[\vartheta(x)=\vartheta(x,q)=\left(q^{1/2}-q^{-1/2}\right)\prod_{i=1}^\infty \frac{(1-q^ix)(1-q^i/x)}{(1-q^i)^2}\]
 (see also Section \ref{jacobithetas}).
Thus the approximation \eqref{eq:asymptoticsuniform} follows from

\begin{lem}[\protect{\cite[Proposition 4.1]{branchedcoverings}}]
\label{lem:asymptoticsoftheta}
We have
\[\frac{\vartheta(e^{hu},e^{-h})}{\vartheta'(0,e^{-h})}=h\frac{\sin(\pi u)}{\pi}\exp\left(\frac{hu^2}{2}\right)\left(1+O\left(e^{-\frac{4\pi^2}{h}}\right)\right)\]
as $h\to+0$ uniformly in $z$. This asymptotic relation can be differentiated any number of times.
\end{lem}
\begin{proof}
We have
\[\vartheta(e^{hu},e^{-h})=i\sqrt{\frac{2\pi}{h}}\exp\left(\frac{hu^2}{2}\right)
\vartheta\left(e^{-2\pi i u},e^{-\frac{4\pi^2}{h}}\right),\]
so, expanding the series of $\vartheta$,
\[\vartheta\left(e^{-2\pi i u},e^{-\frac{4\pi^2}{h}}\right)=\sum_{n\in\Z}(-1)^n\exp
\left(-\frac{2\pi^2\left(n+\frac12\right)^2}h\right)e^{-2\pi\left(n+\frac12\right)u}.\]
From here it is clear that as $h\to+0$ the terms $n=0$ and $n=-1$ dominate all others.
\end{proof}

We will show that we have the same result \eqref{eq:asymptoticsuniform} holds for the rescaled limit of the exponential generating function of the numbers $\langle \p_k\rangle_{\w,q}$, that is, in the case of the pillowcase weights.

Recall the formula of A. Okounkov and A. Eskin \cite[Theorem 5]{pillow} for the 1-point function,
\[
F(x)=\frac{1}{\vartheta(x)}[y^0]
\sqrt{\frac{\vartheta(-y)\vartheta( xy)}{\vartheta(y)\vartheta(- xy)}}.
\]
Here $[y^0]$ means that we expand the series in the variable $y$ and then we take the constant coefficient.

Since (in the conventional notations which are recalled in Section \ref{jacobithetas}),
\[\vartheta(z,\tau)=-\eta^{-3}(\tau)\vartheta_{11}(z,\tau),\]
where $x=e^{2\pi i z}$, $q=e^{2\pi i \tau}$, and $\eta$ stands for the Dedekind eta function, the following lemma about $\vartheta_{11}$ is relevant.

\begin{lem}\label{lem:thetaident}
We have the following identities and approximations for $\vartheta_{11}$ at $x=1$ and $x=-1$, respectively:
 \begin{align*}
 \vartheta_{11}(z,\tau)&=-(-i\tau)^{-1/2}\sum_{n\in\Z}
 e^{\pi i\left(n+\frac12\right)}\exp-\frac{\pi i}\tau\left(n+\tfrac12-z\right)^2\\
 &\approx-(-i\tau)^{-1/2} e^{\pi i \left(\round{z-\frac12}+\tfrac12\right)}
  \exp-\frac{\pi i}\tau\fractional{z-\tfrac12}^2,\\
 \vartheta_{11}\!\left(z+\tfrac12,\tau\right)&=-(-i\tau)^{-1/2}\sum_{n\in\Z}
 \exp-\frac{\pi i}\tau\left(n+\tfrac12-z\right)^2\\
 &\approx-(-i\tau)^{-1/2}\exp-\frac{\pi i}{\tau}\fractional{z-\tfrac12}^2,
 \end{align*}
 where $\round{x}$ stands for the integer closest to $x$, $\fractional{x}=x-\round x$, and $\approx$ means ``up to exponentially small terms.''
\end{lem}
\begin{proof}
This is a straightforward application of the modular transformation; see Section \ref{jacobithetas}. We also use the identity
\[\vartheta_{11}\!\left(z+\tfrac12,\tau\right)=-\vartheta_{10}(z,\tau).\qedhere\]
\end{proof}
In order to obtain the coefficient of $y^0$, we take the integral
\[
\oint_{|y|=c}\frac{dy}y\sqrt{\frac{\vartheta(-y)\vartheta( xy)}{\vartheta(y)\vartheta(- xy)}}\]
Letting $x=e^{2\pi i z}$, $y=e^{2\pi i w}$, and $q=e^{2\pi i \tau}$, and using the approximations in Lemma \ref{lem:thetaident}, this becomes
\begin{multline*}
\int_0^1e^{\pi i\left(\round{w+z-\frac12}-\round{w-\frac12}+1\right)}\times\\
\exp-\frac{\pi i}\tau
\left(\fractional{w}^2+\fractional{w+z-\tfrac12}^2-\fractional{w-\tfrac12}^2
-\fractional{w+z}^2\right)dw.
\end{multline*}
This integral can be performed by appropriately partitioning the domain to reflect the discontinuities inherited from the functions $\fractional{\cdot}$ and $\round{\cdot}$. The result is
\begin{equation}\label{eq:integralresult}
1-(1+i)z.
\end{equation}

Since we are interested in the expansion with respect to $u$ of $F(e^{hu})$, for $x=e^u=e^{2\pi i z}$ and $q=e^{-h}=e^{2\pi i \tau}$, the result of equation \eqref{eq:integralresult} becomes $1+(2\pi)^{-1}(1-i)hu=1+O(h)$
as $h\to+0$. Thus
\[F(e^{hu})=\frac{1}{\vartheta(x)}(1+O(h))\approx \frac{\pi}{h\sin(\pi u)}.\]
Since this is precisely equivalent to \eqref{eq:asymptoticsuniform}, this establishes the question of a candidate for the limit shape, which is settled to equal the one corresponding to the uniform distribution. Also, from here we know the top-degree terms of the expectations (up to a factor due to the choice of scaling of the limit shape) are
\[\langle\p_k\rangle_{\w,q}\approx \left(\frac 2{h}\right)^{k+1} \left(1-\frac{1}{2^k}\right)\zeta(-k)+\textrm{lower order terms}.\]

\paragraph{Multiplicativity of the expectation.}
To establish this property, we need a preliminary result.
Recall the definition of the $n$-point function,
\[F(x_1,\dots,x_n)=\frac1Z\sum_{j_1,\dots,j_n\in\Z+\frac12}  x_1^{j_1}\cdots x_n^{j_n} \sum_{\lambda\in\mathcal P(j_1,\dots,j_n)}q^{|\lambda|}\w(\lambda)\]
where $\mathcal P(j_1,\dots,j_n)$ denotes the set of partitions $\lambda$ such that the numbers $j_1,\dots,j_n\in \Z+\frac12$ are all contained in the set of modified Frobenius coordinates $\{\lambda_i-i+\frac12:i=1,2,\dots\}$ of $\lambda$, and $Z=\sum_\lambda q^{|\lambda|}\w(\lambda)$ is the normalization constant. For the same reasons as explained in Section \ref{sec:quasimodularitytwoquotients}, the $n$-point function encodes the expectations of products of shifted symmetric power functions. Whence our interest in its asymptotics:
\begin{prop}
The highest degree term of the exponential generating function $F\left(e^{hu_1},e^{hu_2},\dots,e^{hu_n}\right)$ of the expectations $\langle \prod_{i=1}^m\p_{k_i}\rangle_{\w,q}$ of products of shifted-sym\-met\-ric power functions is
\[\left(\prod_{i=1}^n\frac\pi{h\sin(\pi u_i h)}\right)(1+O(h))\]
as $h\to+0$.
\end{prop}
\begin{proof}
We consider A. Okounkov and A. Eskin's formula for the $n$-point function,
\[F(x_1,x_2,\dots,x_n)=\frac{1}{\prod_i\vartheta(x_i)}\left[y_1^0y_2^0\cdots y_n^0\right]
\prod_{i<j}\frac{\vartheta\!\left(\frac{y_i}{y_j}\right)
\vartheta\!\left(\frac{x_iy_i}{x_jy_j}\right)}
{\vartheta\!\left(\frac{x_iy_i}{y_j}\right)
\vartheta\!\left(\frac{y_i}{x_jy_j}\right)}
\sqrt{\prod_i\frac{\vartheta(-y_i)\vartheta(x_iy_i)}{\vartheta(y_i)\vartheta(-x_iy_i)}}
.\]
We want to show that the factors of the form
\begin{equation}\label{eq:thecrossterm}
\frac{\vartheta\!\left(\frac{y_i}{y_j}\right)
\vartheta\!\left(\frac{x_iy_i}{x_jy_j}\right)}
{\vartheta\!\left(\frac{x_iy_i}{y_j}\right)
\vartheta\!\left(\frac{y_i}{x_jy_j}\right)}
\end{equation}
are of order $1+O(h)$ as $h\to+0$. Then the same argument as we used before for the 1-point function will complete the proof.
We apply the approximation of Lemma \ref{lem:asymptoticsoftheta} to each factor of the crossterm \eqref{eq:thecrossterm}. Letting $x_i=e^{u_i}$ and $y_i=e^{v_i}$, we get
\begin{equation}\label{eq:expandquotientofsines}
e^{-\frac{u_1u_2}h}\frac{\sin\frac\pi h(u_1+v_1-u_2-v_2)\sin\frac \pi h(v_1-v_2)}
{\sin\frac\pi h(u_1+v_1-v_2)\sin\frac\pi h(v_1-u_2-v_2)}
+\textrm{exponentially small terms}.
\end{equation}
We will show that the quotient of sines is approximately equal to 1 for small, suitable $h$.

We assume that $h$ is very close to the imaginary axis and is in the first quadrant in $\C$ (i.e., $\imagp h\gg\realp h>0$), and that $u_1,u_2,v_1,v_2$ are such that the determinants $\det(h,u_1+v_1-u_2-v_2)$, $\det(h,v_1-v_2)$, $\det(h,u_1+v_1-v_2)$, $\det(h, v_1-u_2-v_2)$ are positive. Here,
\[\det(h,b)=\det\begin{pmatrix}
\realp h & \imagp h \\
\realp b & \imagp b
\end{pmatrix}.\]
As $h\to0$,
\[\sin(b/h)\approx\frac1{2i}\exp\left(-\frac{ib}h\sgn \det(h,b)\right),\]
so the quotient of sines behaves as
\begin{multline*}
\exp\bigg(-\frac{i\pi}h\big((u_1+v_1-u_2-v_2)+(v_1-v_2)-\\
(u_1+v_1-v_2)-( v_1-u_2-v_2)\big)(+1)\bigg)=1.
\end{multline*}

So only the exponential term in \eqref{eq:expandquotientofsines} counts. Since we really are interested in the limiting behavior of $F(e^{hu_1},e^{hu_2},\dots,e^{hu_n})$, the correct form of the term is
\[e^{-hu_1u_2}=1+O(h)\]
as $h\to+0$, and we are done.
\end{proof}

It follows from the following proposition and from the estimate \cite[Section 3.3.5]{pillow}
\[\left\langle\p_\rho\right\rangle_{\w,q}=O\left(h^{-|\rho|-\ell(\rho)}\right)\]
that
\begin{equation}\label{eq:multiplicativity}
\lim_{h\to+0}h^{-\sum_i\left(k_i+1\right)}\left(\left\langle\prod_i\p_{k_i}
\right\rangle_{\w,q}
-\prod_i\langle\p_{k_i}\rangle_{\w,q}\right)=0.
\end{equation}
In other words, the top-degree term of these expectations is multiplicative.
\begin{rmk}
This is only possible if the limiting measure is completely concentrated at the limit shape candidate proposed above. Indeed, this means that for any function $f$ in the algebra generated by the functions $\p_k$,
\[h^{-2\deg f}|\var f|=h^{-2\deg f}\left(\langle f^2\rangle_{\w,q}-\langle f\rangle_{\w,q}^2\right)\to0.\]
(Here, $\deg f$ refers to the grading that assigns $\deg \p_k=k+1$.) The variance vanishes, so the limiting measure must be a Dirac delta supported at the limit shape.
\end{rmk}

\begin{rmk}
This is optimal: the multiplicativity does not go beyond the highest degree term. Already in the case of products of $\p_1$ we have\footnote{These numbers were obtained by finding the decomposition of the corresponding series as a polynomial of Eisenstein series, through the direct computation of the first few terms.}
\begin{align*}
\langle\p_1\rangle_{\w,q}&=\frac{\pi^2}{24\,h^2}+\frac1{4\,h}+\mathrm{e.s.t.} \\
\langle\p_1^2\rangle_{\w,q}&
 =\frac{\pi^4}{576\,h^4}+\frac{7\,\pi^2}{48\,h^3}+\frac{17}{16\,h^2}+\mathrm{e.s.t.} \\
\langle\p_1^3\rangle_{\w,q}&=
 \frac{\pi^6}{13\,824\,h^6}+\frac{13\,\pi^4}{768\,h^5}+\frac{93\,\pi^2}{128\,h^4}
 +\frac{305}{64 \,h^3}+\mathrm{e.s.t.} \\
\langle\p_1^4\rangle_{\w,q}&=\frac{\pi^8}{331\,776\,h^8}+\frac{19\,\pi^6}{13\,824\,h^7}
+\frac{433\,\pi^4}{1\,536\,h^6}+\frac{2\,339\,\pi^2}{384\,h^5} +\frac{8\,033}{256\,h^4}
+\mathrm{e.s.t.}
\end{align*}
Here, ``e.s.t.'' stands for ``exponentially small terms.''
\end{rmk}
\begin{rmk}\label{rmk:noCLT}
As another consequence of the asymptotic expansions above, we have no hope of getting a full Central Limit Theorem for this regime, because this would imply Wick's theorem. Wick's theorem characterizes Gaussian probability distributions with mean zero; see for example \cite[Section 1.2]{zinnjustin}. In our case, it would contain the statement that
\[\langle abcd\rangle_{\w,q}=\langle ab\rangle_{\w,q}\langle cd\rangle_{\w,q} +\langle ac\rangle_{\w,q}\langle bd\rangle_{\w,q} +\langle ad\rangle_{\w,q}\langle bc\rangle_{\w,q}\]
for all $a$, $b$, $c$, and $d$ in the algebra generated by the functions $\p_k$ with vanishing $(\w,q)$-mean.
In particular, setting $a=b=c=d=\p_1$, this means that we would need to have
\[\left\langle \left(\p_1-\langle\p_1\rangle_{\w,q}\right)^4\right\rangle_{\w,q}
=3\left\langle \left(\p_1-\langle\p_1\rangle_{\w,q}\right)^2\right\rangle_{\w,q}^2,\]
but instead we get
 \[\left\langle \left(\p_1-\langle\p_1\rangle_{\w,q}\right)^4\right\rangle_{\w,q}=
 \frac{11\pi^4}{64h^6}+\cdots\neq \frac{3\pi^4}{64h^6}+\cdots=3\left\langle \left(\p_1-\langle\p_1\rangle_{\w,q}\right)^2\right\rangle_{\w,q}^2.\]
  So there is no such theorem, and hence the convergence to the limit shape is not normal.
\end{rmk}

Figuring out what the next terms are in the expansions of the expectations of products of shifted power functions $\p_k$ is left for future work.

\section{Alternative formula of the weights in terms of hooks}
\label{sec:wformulahooks}
In this section, we will prove that
\begin{equation}\label{pillowcasehooksformula}
\w(\lambda)=\left(\frac{\prod\textrm{odd hook lengths of $\lambda$}}{\prod\textrm{even hook lengths of $\lambda$}}\right)^2
\end{equation}
for $\lambda$ balanced, and $\w(\lambda)=0$ for all other $\lambda$. 


Let $\lambda $ be a balanced partition with 2-quotient $(\alpha,\beta)$. Recall that
\begin{align}
\w(\lambda)&=\left(\frac{\dim\lambda}{|\lambda|!}\right)^2 \f_{(2,2,\dots,2)}(\lambda)^4 \notag \\
&=\left(\frac{\dim\lambda}{|\lambda|!}\right)^2  \left(\left|C_{(2,2\dots,2)}\right| \frac{\chi^\lambda(2,2,\dots,2)}{\dim\lambda} \right)^4 \label{plugeverythinghere}
\end{align}
Now,
\begin{equation}\label{conjugacyclasspiece}
\left|C_{(2,2\dots,2)}\right| = \frac{|\lambda|!}{\mathfrak z(2,2,\dots,2)}=\frac{|\lambda|!}{2^{|\lambda|/2}(|\lambda|/2)!}
\end{equation}
and, by formulas \eqref{eq:dominoformula} and \eqref{dimensionsofquotients},
\begin{align*}
\left|\chi^\lambda(2,2,\dots,2)\right|&=\binom{|\lambda|/2}{|\alpha|}\dim\alpha \dim\beta\\
&=\binom{|\lambda|/2}{|\alpha|}\frac{|\alpha|!|\beta|!}{\prod\textrm{halves of the even hook lengths of $\lambda$}}.
\end{align*}
By items \eqref{LemAlphaPlusBetaIsLambdaOverTwo} and \eqref{LemHalfOfHooksAreEven} of Lemma \ref{LemTwoQuotients}, this equals
\begin{equation}\label{characterpiece}
\left|\chi^\lambda(2,2,\dots,2)\right|=\frac{2^{|\lambda|/2}(|\lambda|/2)!}{\prod\textrm{even hook lengths of $\lambda$}}.
\end{equation}
On the other hand, by the hook formula \eqref{hookformula}, we have that
\begin{equation}\label{dimensionpiece}
\dim\lambda=\frac{|\lambda|!}{\prod\textrm{hook lengths of $\lambda$}}
\end{equation}
Since
\[\frac{\prod\textrm{hook lengths of $\lambda$}}{\left(\prod\textrm{even hook lengths of $\lambda$}\right)^2}=
\frac{\prod\textrm{odd hook lengths of $\lambda$}}{\prod\textrm{even hook lengths of $\lambda$}}, \]
plugging \eqref{conjugacyclasspiece}, \eqref{characterpiece}, and \eqref{dimensionpiece} into \eqref{plugeverythinghere}, we get formula \eqref{pillowcasehooksformula}.

On the other hand, as was remarked in Section \ref{twoquotientssection}, if $\lambda$ is not balanced, it is impossible to decompose $\lambda$ into a 2-domino tiling. By the Murnaghan-Nakayama rule (see Section \ref{murnaghannakayamarule}) this implies that $\chi^\lambda(2,2,\dots,2)$ vanishes when $\lambda$ is not balanced, whence also $\w(\lambda)=0$ in this case. 
\section{A variational argument}\label{sec:variational}
In order to deduce the asymptotic behavior of the pillowcase weights $\w(\lambda)$, we will use a variational argument resembling the idea that S. Kerov and A. Vershik \cite{vershikkerov1985} and B. Logan and L. Shepp \cite{loganshepp} used in the analysis of the Plancherel distribution.

The idea is that we can approximate the logarithm of products of hooks by an integral, whence we can exploit the form of formula \eqref{pillowcasehooksformula}.

For a non-increasing function $F:\R_+\to\R_+$, define
\[F^{-1}(y)=\inf\{x:F(x)\leq y\},\]
and
\[h_F(x,y)=F(x)+F^{-1}(y)-x-y,\]
which gives an approximation of the hook length at $(x,y)$.
Let $\lambda$ be a \emph{balanced} partition, and $n=|\lambda|$. We contract its diagram until it has area 1, rescaling by $1/\sqrt {n}$, and we associate to it a function $F(x)$ that describes its rim,
\[F(x)=\frac 1{\sqrt n}\#\{\textrm{parts of $\lambda$ of size $\leq\lceil \sqrt n\,x \rceil$}\}.\]
In the following diagram, we have the example of $\lambda=(4,3,3,2)$.
\begin{center}
\includegraphics{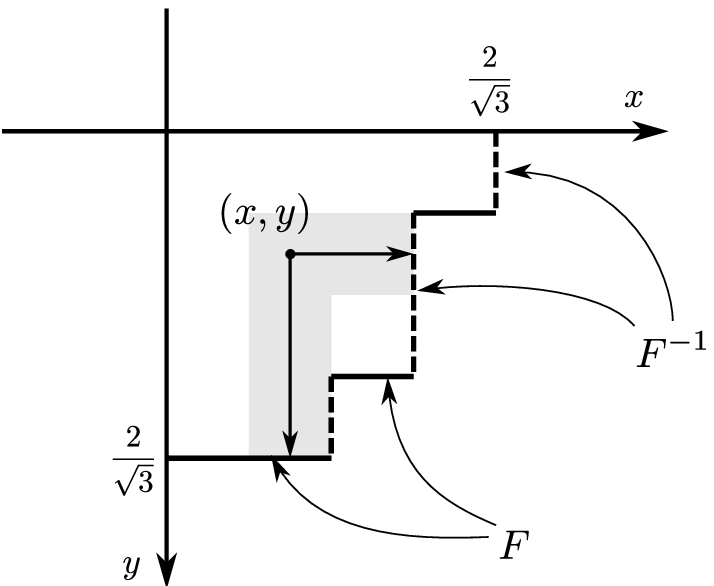}
\end{center}
In the picture, we have inverted the orientation of the $y$ axis to reflect the usual orientation of our Young diagrams, and we have drawn the parts of the rim of $\lambda$ that are described by $F$ with a solid line, while the parts that correspond do $F^{-1}$ are drawn with dashed lines. The numbers on the axes reflect the rescaling by $1/\sqrt{n}=1/\sqrt{12}$. We have also highlighted the hook corresponding to the cell $(2,2)$. We will approximate the logarithm of its length by
\[n\int_R \log \left(\sqrt{n} \,h_F(x,y)\right)\,dx\,dy,\]
where $R$ is the rectangle between the points $\frac1{\sqrt{12}}(1,1)$ and $\frac1{\sqrt{12}}(2,2)$. To motivate this approximation, we have drawn, at the center of the cell $(2,2)$, the point $(x,y)$ and an inverted L ending extending to the rescaled rim of $\lambda$; the length of this L is precisely $h_F(x,y)$: its vertical part measures $F(x)-y$, while its horizontal part measures $F^{-1}(y)-x$. Then $\sqrt n\,h_F(x,y)=4$ coincides with the length of the hook of cell $(2,2)$, and the integral above is very close to this number as well. In general, we have

\begin{lem}\label{lem:hookapproximation}
Let $\lambda$ be a partition and let $\square\in\lambda$ be a cell in its Young diagram, and denote by $F$ the non-increasing function that describes the rim of the Young diagram of $\lambda$, as defined above. Let $R_\square\subseteq\{(x,y):0\leq y\leq F(x)\}$ be the rectangular domain corresponding to $\square$. Let $h_\square$ denote the hook length of $\square$. Then
\[\log h_\square=n\int_{R_\square}\log\left(\sqrt n h_F(x,y)\right)dx\,dy + c\!\left(h_\square\right),\]
where
\[c(x)=\frac12\sum_{k=1}^\infty\frac{1}{k(k+1)(2k+1)x^{2k}}.\]
\end{lem}
\begin{rmk}\label{rmk:approximation}
The function $c$ is strictly decreasing in the interval $[1,+\infty)$. It remains between $c(1)=\frac12(3-4\log 2)\approx 0.113706$ and
\[\lim_{x\to+\infty} c(x)=0.\]
This implies that the approximation of the logarithm of a product of any $m$ hook lengths using the integral of $\log\left(\sqrt nh_F(x,y)\right)$ will have an error of order $\leq c(1)\cdot m$. This estimate can be improved in the large-scale case; see Lemma \ref{lem:asymptoticsofremainder}.
\end{rmk}
\begin{proof}
Let $(x_0,y_0)$ be the point at the center of $R_\square$. Recall that the area of $R_\square$ is $1/n$. Then
\begin{align*}
n\int_{R_\square} &\log \left(\sqrt n h_F(x,y)\right)\,dx\,dy =
n\int_{R_\square}\log \left(h_\square-\sqrt n(x-x_0)-\sqrt n(y-y_0)\right)dx\,dy \\
&=\log h_\square+n\int_{R_\square} \log\left(1-\frac{(x-x_0)+(y-y_0)}{h_\square/\sqrt n}\right)dx\,dy \\
&=\log h_\square +  \int_{\sqrt n y_0-\frac12}^{\sqrt n y_0+\frac12}\int_{\sqrt n x_0-\frac12}^{\sqrt n x_0+\frac12}\log\left(1-\frac{(u-\sqrt n x_0)+(v-\sqrt n y_0)}{h_\square}\right)du\,dv,
\end{align*}
where $u=\sqrt n x$ and $v=\sqrt n y$.
Now expand the Taylor series of the logarithm and integrate term by term.
\end{proof}
\begin{lem}\label{lem:asymptoticsofremainder}
\[\sum_{\square\in\lambda} (-1)^{h_\square}c(h_\square)=O\left(\sqrt{|\lambda|}\right)\quad \textrm{as $|\lambda|\to\infty.$}\]
\end{lem}
\begin{proof}
We have
\begin{equation}\label{eq:boundforsumofhooks}
\left|\sum_{\square\in\lambda} (-1)^{h_\square}c(h_\square)\right|\leq \sum_{\{\square \in \lambda:\,\textrm{$h_\square$ odd}\}} c(h_\square).
\end{equation}
Since $c$ is decreasing and $c(h)\to0$ very quickly as $h\to\infty$, the worst case scenario is the case in which we maximize the number of \emph{small} odd hook lengths $h_\square$. This happens in the case of the \emph{staircase partition} $(\ell,\ell-1,\dots,2,1)$. This partition is not balanced, but it is the worst possible case. In the case of the staircase, it is obvious that the right hand side of \eqref{eq:boundforsumofhooks} is of order $O\left(\sqrt{|\lambda|}\right)$.
\end{proof}

Since we want to approximate the logarithm of the quotient \eqref{pillowcasehooksformula}, we need to distinguish the domains $O$ and $E$ corresponding to the cells whose hooks are of odd and even length, respectively.
For example, in the following diagram we have shaded in grey the domain $E$ in the case of $\lambda=(4,3,3,2)$; $O$ would be the remaining white area inside the rescaled Young diagram.
\begin{center}
\includegraphics{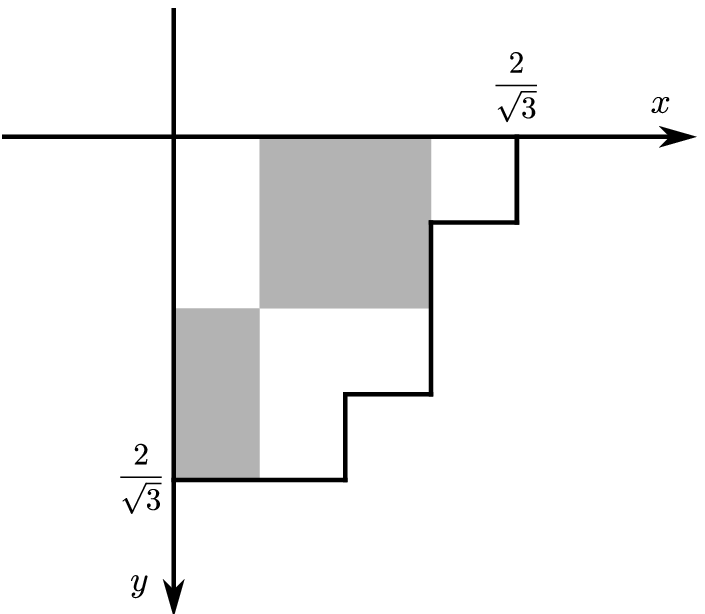}
\end{center}
Then by formula \eqref{pillowcasehooksformula} and Lemma \ref{lem:hookapproximation},
\[\log\w(\lambda)= 2 n\left(\int_O-\int_E\right)\log\left(\sqrt n\,h_F(x,y)\right)\,dx\,dy +\sum_{\square\in\lambda}(-1)^{h_\square}c(\lambda).\]
By item \eqref{LemHalfOfHooksAreEven} of Lemma \ref{LemTwoQuotients}, the area of $O$ is the same as the area of $E$, so the contributions of $\log \sqrt n$ cancel out and we are left with
\[\log\w(\lambda)= 2n\left(\int_O-\int_E\right)\log h_F(x,y) \,dx\,dy +\sum_{\square\in\lambda}(-1)^{h_\square}c(\lambda).\]

The domains $E$ and $O$ can be very complicated, but in the case of balanced partitions we can exploit item \eqref{LemOddAndEvenHooksPositions} of Lemma \ref{LemTwoQuotients} as follows.

First, rotate the domain $135^\circ$ and let $L:\R\to\R_+$ be the continuous function that gives the rim of the rotated domain, as in the following diagram:
\begin{center}
\includegraphics{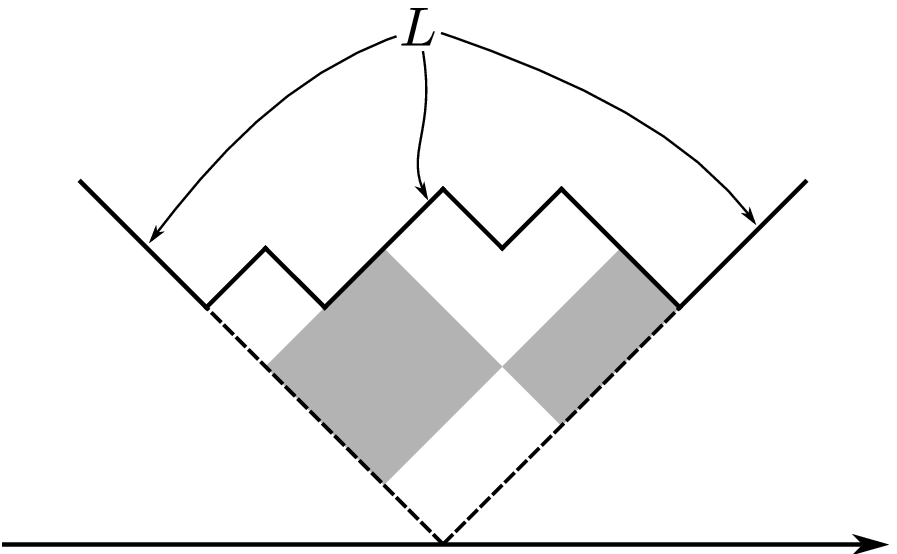}
\end{center}
Note that $L'=\pm1$ where it is defined, and $L(x)=|x|$ for $|x|$ large enough. Denote by $O'$ and $E'$ the rotated versions of $O$ and $E$.

It is now convenient to change variables to $s$ and $t$ such that $x=\frac1{\sqrt 2}(L(s)-s)$ and $y=\frac{1}{\sqrt 2}(L(t)+t)$. To explain how this change of variables works we have the following diagram.
\begin{center}
\includegraphics{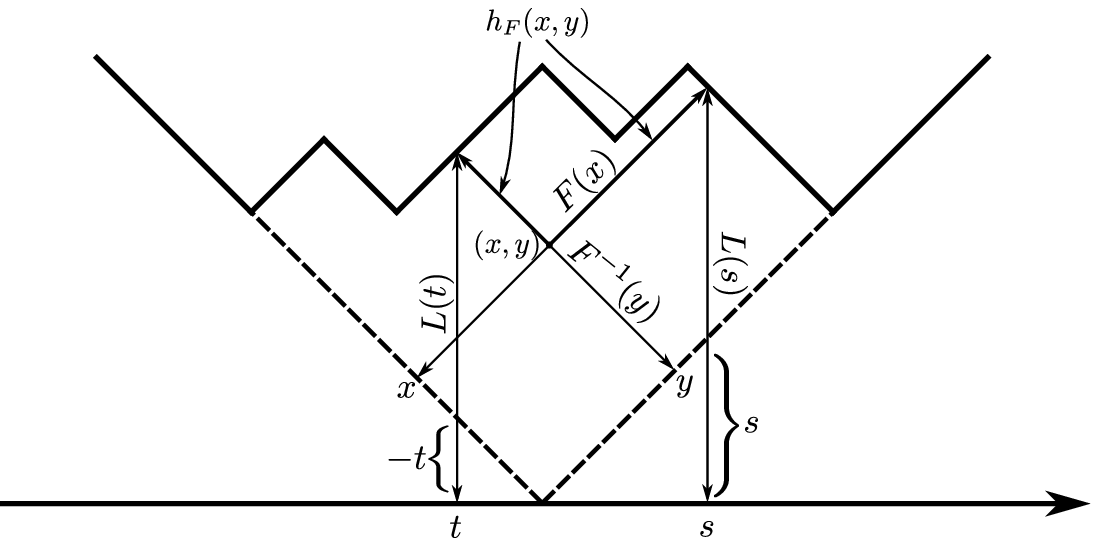}
\end{center}
In the diagram, the point $(x,y)$ is at the center, and through it pass two diagonals of lengths $F(x)$ and $F^{-1}(y)$ intersecting the diagonal axes at coordinates $x$ and $y$, respectively. From their intersections with the rim of the rotated Young diagram, we have vertical segments of lengths $L(s)$ and and $L(t)$.
From the diagram it is also clear that $h_F(x,y)=\sqrt 2(s-t)$, and that our integration domains $O$ and $E$ will be subsets of $\{t<s\}$.
The transform $(s,t)\mapsto (x,y)$ is not bijective, but it is surjective, and it is injective in the region in which its Jacobian is nonzero. The determinant of its Jacobian is
 \[-\frac12\left(1-L'(s)\right)\left(1+L'(t)\right),\]
which is supported in a compact set because $L(x)=|x|$ for large $|x|$.
If we denote by $\widetilde E$ and $\widetilde O$ the images of $E$ and $O$ under this change of variables, our integral becomes:
\[n\left(\int_{\widetilde O}-\int_{\widetilde E}\right)\left(\log \sqrt 2(s-t)\right)(1-L'(s))(1+L'(t)) \,ds\,dt.\]

Let $(\alpha,\beta)$ be the 2-quotient of $\lambda$, and let $L_\alpha$ and $L_\beta$ be the functions describing the rims of the Young diagrams of $\alpha$ and $\beta$, rotated $135^\circ$ and rescaled by $1/\sqrt n$ (so that together they have area $1/2$, by item \eqref{LemAlphaPlusBetaIsLambdaOverTwo} of Lemma \ref{LemTwoQuotients}).
Recall from the description given in Section \ref{twoquotientssection} that the information  of the derivative $L'$ is split into $L'_\alpha$ and $L'_\beta$.
Indeed, the black and white pebbles of the Maya diagram determine $L'$, and half of them correspond to $\alpha$ and determine $L'_\alpha$, while the other half correspond to $\beta$ and determine $L'_\beta$. In other words, we have
\begin{equation}\label{eq:partsofdomain}
L'(s)=\left\{
\begin{array}{ll}
L'_\alpha\left(s-\frac {k}{2\sqrt{2n}}\right), &\textrm{for $s\in \frac{1}{\sqrt{2n}}(k-1,k)$, $k\in 2\Z$}, \\
L'_\beta\left(s-\frac {k}{2\sqrt{2n}}\right), & \textrm{for $s\in \frac{1}{\sqrt{2n}}(k,k+1)$, $k\in 2\Z$}.
\end{array}
\right.
\end{equation}
Item \eqref{LemOddAndEvenHooksPositions} of Lemma \ref{LemTwoQuotients} then implies that the integral over $\widetilde E$ essentially involves the interaction of each component of the 2-quotient with itself, so it becomes
\begin{equation}\label{eq:evenintegral}
-n\int_{t<s} \left(\log\sqrt2(s-t)\right)\left(\left(1-L'_\alpha(s)\right)\left(1+L'_\alpha(t)\right) +\left(1-L'_\beta(s)\right)\left(1+L'_\beta(t)\right)\right)\,ds \,dt.
\end{equation}
Similarly, the integral over $\widetilde O$ involves the interactions of the components of the 2-quotient with each other, and becomes
\begin{equation}\label{eq:oddintegral}
n\int_{t<s} \left(\log\sqrt2(s-t)\right)\left(\left(1-L'_\alpha(s)\right)\left(1+L'_\beta(t)\right) +\left(1-L'_\beta(s)\right)\left(1+L'_\alpha(t)\right)\right)\,ds \,dt.
\end{equation}
Note that there is an implicit change of coordinates here, that would handle the translations by $k/\left(2\sqrt{2n}\right)$ in formula \eqref{eq:partsofdomain} and the split of the domains. Instead of translating and separating the parts of the domain, we work directly with the components of the 2-quotient.

Adding \eqref{eq:evenintegral} and \eqref{eq:oddintegral}, and simplifying, we obtain
\[n\int_{t<s}\left(\log\sqrt 2(s-t)\right)\Delta'(s)\Delta'(t)\,ds\,dt,\]
where $\Delta(s)=L_\alpha(s)-L_\beta(s)$. We can also write this as
\[\Theta(L_\alpha,L_\beta)=\frac n2\int_{\R^2}\left(\log\sqrt 2|s-t|\right)\Delta'(s)\Delta'(t)\,ds\,dt.\]
Integrating by parts twice, we can rewrite this as:
\[\Theta(L_\alpha,L_\beta)=-\frac n2\|\Delta\|^2,\]
where the \emph{Sobolev norm} $\|\cdot\|$ is given by
\[\|\Delta\|^2=\int_{\R^2}\left(\frac{\Delta(s) -\Delta(t)}{s-t}\right)^2 ds\,dt.\]

In sum we have, as a consequence of Lemmas \ref{lem:hookapproximation} and \ref{lem:asymptoticsofremainder},
\begin{prop}\label{lem:sobolev}
As $|\lambda|\to\infty$,
\[\w(\lambda)=\exp\left(-\frac {|\lambda|}2\|\Delta\|^2+O(\sqrt{|\lambda|})\right),\]
where $\Delta=L_\alpha- L_\beta$ is the difference of the functions describing the (rescaled) rims of the Young diagrams of the components $\alpha$ and $\beta$ of the 2-quotient of $\lambda$.
\end{prop}

\begin{rmk}
The partitions with identical components $\alpha=\beta$ (and hence $\Delta=0$) of the 2-quotient $(\alpha,\beta)$ are easily characterized: they are precisely the ones that can be constructed concatenating $2\times2$ blocks
\[\yng(2,2)\]
and they are hence in one-to-one correspondence with the partitions of size $n/4$.
\end{rmk}

\section{An application of the Theorem of Meinardus}\label{sec:meinardus}
\begin{lem}\label{lem:meinardusforZ}
As $n\to\infty$, we have the asymptotic behavior
\[Z=\sum_{\lambda}\w(\lambda)\sim\frac1{2^{1/8}3^{3/8}}\frac{e^{\pi\sqrt{\frac n6}}}{n^{7/8}}.\]
where the sum is over all $\lambda$ of size $|\lambda|=n$.
\end{lem}
\begin{proof}
It was shown by A. Eskin and A. Okounkov \cite[Section 3.2.4]{pillow} that
\[\sum_\lambda \w(\lambda)q^{|\lambda|}=\prod_{i=1}^\infty (1-q^{2i})^{-1/2}.\]
We thus have the following parameters for the Theorem of G. Meinardus, as presented by G. Andrews \cite[Chapter 6]{andrews}: $q$ in the book is $q^2$ here, $a_n=\frac12$, $D(s)=\zeta(s)/2$, $\alpha=1$, $A=\frac 12$, $\kappa=-\frac78$, $D(0)=-\frac14$, $D'(0)=\frac14\log2\pi$.
\end{proof}

Denote by $p(n)$ the number of partitions of $n$. Recall that the Theorem of G. Meinardus also implies the asymptotics
\begin{equation}\label{eq:asymptoticspartitions}
p(n)\sim \frac1{4n\sqrt3}\exp\left(\pi\sqrt{\tfrac23 n}\right).
\end{equation}
See \cite[Theorem 6.3]{andrews}.
We thus have
\begin{prop}\label{prop:measureconcentration}
Let $\varepsilon>0$, and let $S_{\varepsilon,n}$ be the set of balanced partitions $\lambda$ of $n$ whose 2-quotients $(\alpha,\beta)$ satisfy $\|\Delta\|=\|L_\alpha-L_\beta\|>\varepsilon$ (in the notations of Section \ref{sec:variational}). Then the $\w$-probability of $S_{\varepsilon,n}$ is asymptotically of order $O\left(e^{-K\sqrt n}\right)$ as $n\to\infty$, for some $K>0$ that depends on $\varepsilon$.
\end{prop}
\begin{rmk}
It is the case that $K\to0$ as $\varepsilon\to+0$.
\end{rmk}
\begin{proof}
Since we have a bound for the cardinality
\[|S_{\varepsilon,n}|\leq p(n)\]
and the $\w$-probability of each partition $\lambda$ is given by
\[\frac{\w(\lambda)}{\sum_{|\mu|=|\lambda|} \w(\mu)},\]
the result follows from Proposition \ref{lem:sobolev}, Lemma \ref{lem:meinardusforZ}, and the asymptotic relation \eqref{eq:asymptoticspartitions}.
\end{proof}
\begin{cor}\label{cor:expectation}
 Let $f(\lambda)$ and $g(\lambda)$ be two functions with $|f(\lambda)|$ and $|g(\lambda)|$ increasing at most polynomially as $|\lambda|\to\infty$, with finite $\w$-expectations, and continuous in the space of piecewise continuous functions with topology of the supremum norm (i.e., we need $f$ and $g$ to be continuous with respect to the functions describing the rim of the partitions).
 Assume that $g$ grows at most polynomially:
 \[g(\lambda)=O(|\lambda|^b) \textrm{ as $|\lambda|\to\infty$}\]
 for some $b>0$. Then
\[\lim_{n\to\infty}\frac{1}{n^{b}} \left(\langle f(\alpha)g(\beta)\rangle_\w-\langle f(\alpha)g(\alpha)\rangle_\w\right)=0,\]
where $(\alpha,\beta)$ is the 2-quotient of the partition $\lambda$ over which the sum of the expectation is taken.
\end{cor}
\begin{proof}
Let $\varepsilon>0$.
Let $\delta>0$ be such that if $\|L_\alpha-L_\beta\|<\delta$ then $|g(\alpha)-g(\beta)|<\varepsilon n^b$.
We want to show that the following tends to 0:
\begin{align*}
\frac{1}{Zn^b}\sum_{|\lambda|=n} \w(\lambda)f(\alpha)\left(g(\beta)-g(\alpha)\right).
\end{align*}
Here, the summation can be split into two parts
\[
\sum_{\|L_\alpha-L_\beta\|<\delta} +
\sum_{\|L_\alpha-L_\beta\|\geq\delta}
\]
The later can be bounded easily using Proposition \ref{prop:measureconcentration}: we get
\[\frac 1{Zn^b}\sum_{\|L_\alpha-L_\beta\|\geq\delta}O
\left(e^{-K_\delta\sqrt n}\right)f(\alpha)\left(g(\beta)-g(\alpha)\right).\]
Since $f$ and $g$ grow polynomially, this tends to 0 as $n\to\infty$.
On the other hand, the first part in the summation above goes like
\[\frac1{Zn^b}\sum_{\|L_\alpha-L_\beta\|<\delta}
O\left(e^{-K_{\delta}\sqrt n}\right)f(\alpha)\varepsilon n^b.\qedhere\]
\end{proof}
\begin{rmk}
The results in this section, together with those of Section \ref{sec:limitshape}, imply that the components $\alpha$ and $\beta$ of the 2-quotient also approach the limit shape corresponding to the uniform distribution, for large partitions $\lambda$.
\end{rmk}

\section{Vanishing of the first approximation of the volumes}
\label{sec:expectationasympt}
We consider the distribution $\w=\w_n$ induced by the pillowcase weights on the set of partitions of $n$, for $n$ even. In this section we discuss the term of highest degree in the $\w$-expectation of formula \eqref{eq:formulaforg} for $\g_\nu$, which turns out to vanish.

We will prove that the term of top degree of the expectations
\[\langle\g_\nu\rangle_\w\]
can be computed with the following formula:
\begin{equation}\label{eq:expectationofg}
\langle\g_\nu\rangle_\w\approx\frac{K^{|\nu|}}{\mathfrak z(\nu)}
\sum_{|\mu|=|\nu|}\sigma_\mu\chi^\mu(\nu)\frac{\dim a}{|a|!}\frac{\dim b}{|b|!}
\end{equation}
where $(a,b)$ is the 2-quotient of $\mu$, and the sum is over all balanced partitions $\mu$ of size $|\nu|$, and $K$ is a constant that depends on the scaling. The approximation \eqref{eq:expectationofg} vanishes: using formula \eqref{eq:dominoformula} and the orthogonality of the characters, we have
\[
\sum_{|\mu|=|\nu|}\sigma_\mu\chi^\mu(\nu)\frac{\dim a}{|a|!}\frac{\dim b}{|b|!}=\frac{1}{(|\nu|/2)!}\sum_\mu
\chi^\mu(\nu)\chi^\mu(2,\dots,2)=0.
\]
This is why we have to also analyze the terms in lower degree; this is mostly future work, but we discuss it briefly in Section \ref{sec:nextterm}.

Let us show why \eqref{eq:expectationofg} is true. Taking another look at formula \eqref{eq:formulaforg} for $\g_\nu$, we see that we really are interested in the asymptotics of
\[\frac{2^{|\nu|/2}}{\mathfrak z(\nu)} \sum_\mu\sigma_\mu\,\chi^\mu(\nu)\,\langle\s_a(\alpha)\,\s_b(\beta) \rangle_\w,\]
where $(\alpha,\beta)$ stand for the 2-quotients of the partitions $\lambda$ over which summation occurs in the $\w$-expectation, and $(a,b)$ is as above.

By Corollary \ref{cor:expectation}, we can substitute the expectation above by
\[\frac{2^{|\nu|/2}}{\mathfrak z(\nu)} \sum_\mu\sigma_\mu\,\chi^\mu(\nu)\,\langle\s_a(\alpha)\,\s_b(\alpha) \rangle_\w.\]
The highest degree terms of $\s_\lambda$ are precisely the regular Schur functions $s_\lambda$; see \cite{okounkovolshanski}\footnote{The work of G. Olshanski, A. Regev and A. Vershik \cite{frobeniusschur} gives the full expansion of $\s_\mu$ in terms of the regular Schur functions $s_\lambda$.}. Thus the relation that defines the Littlewood-Richardson numbers (see Section \ref{littlewoodrichardson}) should hold to the highest degree at least, and we can write the above as
\[\frac{2^{|\nu|/2}}{\mathfrak z(\nu)} \sum_\mu\sigma_\mu\,\chi^\mu(\nu)\sum_\eta c^\eta_{ab}\langle\s_\eta(\alpha) \rangle_\w,\]
where $c^\eta_{ab}$ denotes the usual Littlewood-Richardson numbers.

Moreover, A. Okounkov and G. Olshanski \cite{okounkovolshanski} proved that
\[\s_\eta=\sum_{\rho}\frac{\chi^\eta(\rho)}{\mathfrak z(\rho)}\p_\rho+\textrm{lower order terms},\]
where the sum is over all partitions $\rho$ of size $|\eta|$, so we really are
interested in the asymptotics of
\begin{equation}\label{eq:substepschurtopower}
\frac{2^{|\nu|/2}}{\mathfrak z(\nu)} \sum_\mu\sigma_\mu\,\chi^\mu(\nu)\sum_\eta c^\eta_{ab}\left\langle \sum_\rho\frac{\chi^\eta(\rho)}{\mathfrak z(\rho)}\p_\rho(\alpha) \right\rangle_\w,
\end{equation}
where the new sums are over all partitions $\rho$ of size $|a|$ and all partitions $|\eta|$ of size $|b|$.

A. Eskin and A. Okounkov \cite[Section 3.3.5]{pillow} proved the following about the asymptotics of the expectations of the power functions as the size $n$ of the partitions considered grows:
\begin{equation*}
\langle\p_\mu\rangle_\w=O\!\left(n^{\frac{|\mu|+\ell(\mu)}2}\right)\quad\textrm{as  $n\to\infty$.}
\end{equation*}
In our case, a similar proof (using the generating function of Lemma \ref{lem:formulafornmpointfunction}) shows that
\begin{equation}\label{eq:asymptoticdegree}
\langle\p_\mu(\alpha)\rangle_\w =O\!\left(n^{\frac{|\mu|+\ell(\mu)}2}\right)\quad\textrm{as  $n\to\infty$.}
\end{equation}
This means that the terms that grow fastest in \eqref{eq:substepschurtopower} are precisely those that maximize the number $|\rho|+ \ell(\rho)$. Since $|\rho|=|a|$ is fixed, only  $\ell(\rho)$ varies. The maximum is thus achieved when $\rho$ equals the partition $(1,1,\dots,1)$ of the corresponding length.
The highest degree term of \eqref{eq:substepschurtopower} is then
\begin{equation}\label{eq:fromthisformula}
\frac{2^{|\nu|/2}}{\mathfrak z(\nu)}\sum_\mu\sigma_\mu\chi^\mu(\nu)
\sum_\eta c^\eta_{ab}\frac{\chi^\eta(1^{|\nu|/2})}{\mathfrak z(1^{|\nu|/2})}
\left\langle \p_1(\alpha)^{|\eta|}\right\rangle_\w.
\end{equation}
Since $\p_1(\alpha)=|\alpha|$, and since by Proposition \ref{lem:sobolev} $|\alpha|\approx|\beta|\approx n/4$ with very high probability\footnote{This also follows directly from the result in Section \ref{sec:quasimodularitytwoquotients}: from the form of the $(n,m)$-point function found in Proposition \ref{lem:formulafornmpointfunction}, it is clear that the expected sizes of $\alpha$ and $\beta$ (i.e., the expectations $\langle\p_1(\alpha)\rangle_{\w,q}$ and $\langle\p_1(\beta)\rangle_{\w,q}$) are the same, and also that they coincide, up to rescaling, with $\langle\p_1\rangle_{\w,q}$.}, for all $k>0$,
\[\left\langle\p_1^k(\alpha)\right\rangle_\w=\langle\p_1(\alpha)\rangle_\w^k\approx\left(\frac n4\right)^k.\]
Also $\mathfrak z(1^m)=m!$ and $\chi^\eta(1,1,\dots,1)=\dim\eta$, so if we let
\[K=\sqrt{2}\langle\p_1(\alpha)\rangle_\w,\]
we obtain, from \eqref{eq:fromthisformula},
\begin{equation}
\frac{K^{|\nu|}}{\mathfrak z(\nu)}\sum_\mu\sigma_\mu\chi^\mu(\nu)
\sum_\eta c^\eta_{ab}\frac{\dim\eta}{|\eta|!}.
\end{equation}

Formula \eqref{eq:expectationofg} now follows from the following
\begin{lem}\label{lem:identidadcs}
For any two (possibly empty) partitions $a$ and $b $,
\[\sum_\eta c^\eta_{a b}\dim\eta=\binom{|a|+|b|}{|a|}\dim a\dim b\]
where the sum is taken over all partitions $\eta$ of size $|a|+|b|$, and $c^\eta_{ab}$ denotes the Littlewood-Richardson coefficients (see Section \ref{littlewoodrichardson}).
\end{lem}
\begin{proof}
Compare the coefficients of the symmetric power function $p_{(1,\dots,1)}$ in the expansions of both sides of
\[\sum_{\lambda,\mu}\frac{\chi^a(\lambda)\chi^b(\mu)}{\mathfrak z(\lambda)\mathfrak z(\mu)}p_{\lambda\cup\mu}=s_as_b=\sum_\eta c^\eta_{ab}s_\eta=\sum_{\eta,\gamma} c^\eta_{ab}\frac{\chi^\eta(\gamma)}{\mathfrak z(\gamma)}p_\gamma.\qedhere\]
\end{proof}
\begin{rmk}
In this proof, one can simply use the multiplicativity of the expectation (proved in Section \ref{sec:limitshape}) instead of introducing the Littlewood-Richardson numbers, but here we decided to give an proof that would be independent of the other result.
\end{rmk}

\section{Next term}\label{sec:nextterm}

We go back to expression \eqref{eq:substepschurtopower}, and according to the behavior given by \eqref{eq:asymptoticdegree}, we see that since the first term vanishes, the next one will be the one corresponding to $\rho=(2,1,1,\dots,1)$.

Appealing again to formula \eqref{eq:dimensionofskewdiagram} and to the Murnaghan-Nakayama rule (see Section \ref{murnaghannakayamarule}), we see that
\begin{align*}
 \chi^\eta(2,1,1,\dots,1)&=\chi^{(2)}(2)\dim \left(\eta/(2)\right)+\chi^{(1,1)}(2)\dim\left(\eta/(1,1)\right)\\
 &=\frac{\dim \eta}{|\eta|!/ 2}\left(\s_{(2)}(\eta)-\s_{(1,1)}(\eta)\right).
 \end{align*}
Also,
\[\mathfrak z(2,1^{|\eta|-2})=2(|\eta|-2)!,\]
and
\[\langle\p_{(2,1^{|\eta|-2})}(\alpha)\rangle_\w\approx\langle\p_2(\alpha\rangle_\w \left(\frac n4\right)^{|\eta|-2},\]
so the next term in the approximation is
\begin{equation*}
\frac{2K^{|\nu|-2}\langle\p_2\rangle_\w}{\mathfrak z(\nu)}\sum_\mu\sigma_\mu\chi^\mu(\nu)\sum_\eta c^{\eta}_{ab}\frac{\dim \eta}{(|\eta|-2)!}\frac{\s_{(2)}(\eta)-\s_{(1,1)}(\eta)}{|\eta|!}.
\end{equation*}
The sum is over all balanced partitions $\mu$ of size $|\nu|$ and over all partitions $\eta$ of size $|\nu|/2$.
We can simplify this further to
\begin{equation}\label{eq:finalformula}
\frac{2K^{|\nu|-2}\langle\p_2\rangle_\w}{\mathfrak z(\nu)}\sum_\mu\sigma_\mu\chi^\mu(\nu)v(a)v(b),
\end{equation}
where
\[
v(\lambda)=\frac{\dim \lambda}{(|\lambda|-2)!}\frac{\s_{(2)}(\lambda)-\s_{(1,1)}(\lambda)}{|\lambda|!}\]
The proof is analogous to that of Lemma \ref{lem:identidadcs}, but this time looking at the coefficients of $p_{(2,1,\dots,1)}$.
Alternatively, one can use the multiplicativity of the expectation proved in Section \ref{sec:limitshape}.

\begin{rmk}\label{rmk:degreeconsiderations}
We have by \eqref{eq:asymptoticdegree},
\[\langle\p_{(2,1^{|\nu|-2})}\rangle_w=O\left(n^{\frac{|\nu|+\ell(\nu)-1}2}\right).\] However, for $k=(\nu_1-2,\dots,\nu_\ell-2)$, $\ell=\ell(\nu)$, since $4g-4=\sum_ik_i$,
\[\dim \mathcal Q(k)=2g-2+\ell=\frac{|\nu|}2\]
if $\mathcal Q(k)$ is not self-resolvent (see Section \ref{moduli}) and
\[\dim \mathcal Q(k)=2g-1+\ell=\frac{|\nu|}2+1\]
if $\mathcal Q(k)$ is self-resolvent.
Hence we expect the term of highest degree of first non-vanishing term to be of this degree. In particular, we expect formula \eqref{eq:finalformula} to vanish in most cases.
\end{rmk}

\appendix
\chapter{Prerequisites}\label{prerequisitesappendix}
\section{Representations of the symmetric group}
\label{representationssection}
The \emph{symmetric group} $S(n)$ in $n$ elements is the set of bijections of the set of $n$ elements $\{1,2,\dots,n\}$. Its elements are known as \emph{permutations}. The group operation is the composition of permutations.

A \emph{representation} $(V,\rho)$ of the symmetric group $S(n)$ consists of a vector space $V$ and a homomorphism $\rho:S(n)\to GL(V)$ of $S(n)$ into the multiplicative group $GL(V)$ of invertible linear transformations on $V$. A representation is \emph{reducible} if there are subspaces $U,W\subset V$ such that $V= U\oplus W$ and $\rho$ factors through $GL(U)\times GL(W)$, that is, there exist representations $\rho_U:S(n)\to GL(U)$ and $\rho_W:S(n)\to GL(W)$ such that $\rho(\sigma)|_U=\rho_U(\sigma)$ and $\rho(\sigma)|_W=\rho_W(\sigma)$ for all $\sigma \in S(n)$. If two such subspaces fail to exist, we say that $(V,\rho)$ is an \emph{irreducible representation} of $S(n)$.

The \emph{dimension} of the representation $(V,\rho)$ is simply the dimension of the underlying vector space $V$.

Two representations $(U,\rho_U)$ and $(V,\rho_V)$ are \emph{isomorphic} if there exists a linear bijection $\phi:U\to V$ such that $\phi\circ\rho_V=\rho_U$. The existence of an isomorphism of representations implies that the underlying vector spaces $U$ and $V$ are of the same dimension.

The \emph{character} $\chi_V:S(n)\to \C$ of a representation $(V,\rho)$ is the function $\chi_V(\sigma)=\trace\rho(\sigma)$. The character is invariant under conjugation: for every $\sigma, \tau\in S(n)$,
\[\chi_V(\tau^{-1}\sigma\tau)=\trace\rho(\tau^{-1}\sigma\tau)=\trace\left(\rho(\tau)^{-1}\rho(\sigma)\rho(\tau)\right)=\trace \rho(\sigma)=\chi_V(\sigma).\]
It is thus constant over conjugacy classes.

The classification of the finite-dimensional representations of the symmetric group can be easily achieved using the standard theory of group characters; for details we refer the reader to \cite{fultonharris,sagan}, for example. In order to explain this classification, we must introduce some combinatorial objects.
A \emph{partition} of a positive integer $n$ is a way to write $n$ as a sum $n=a_1+a_2+\cdots+a_k$ of positive integers $a_1,a_2,\dots,a_k$, and is usually denoted $(a_1,a_2,\dots,a_k)$. Conventionally, we order the parts so that $a_1\geq a_2\geq\dots\geq a_k$. For example, the partitions of the number 5 are $(5)$, $(4,1)$, $(3,2)$, $(3,1,1)$, $(2,2,1)$, $(2,1,1,1)$, and $(1,1,1,1,1)$. It is also standard to abbreviate repetitions using exponents. For example, $(3,3,1,1,1,1,1)$ can be written $(3^2,1^5)$. We will use the notation $|\lambda|=\sum_i\lambda_i=n$, and $\ell(\lambda)$ will denote the number of parts of $\lambda$.

The irreducible representations of the symmetric group $S(n)$ can be indexed by partitions of $n$, that is, for each partition $\lambda$ of $n$ there exists an irreducible representation $(V_\lambda,\rho_\lambda)$ of $S(n)$. We will denote by $\chi^\lambda$ the character of the representation corresponding to the partition $\lambda$.

In the case of the symmetric group $S(n)$, the conjugacy classes are also indexed by partitions: the lengths of the cycles of an element of $S(n)$ are invariant under conjugation, they completely determine the conjugacy class, and they correspond to a partition of $n$. For example, the permutation in $S(5)$ that takes $1\mapsto 2$, $2\mapsto 5$, $5\mapsto 1$, $3\mapsto 4$, $4\mapsto 3$ is denoted in cycle notation by $(1,2,5)(3,4)$; it is composed of two cycles, $(1,2,5)$ and $(3,4)$, of lengths 3 and 2 respectively. Its conjugacy type corresponds to the partition $(3,2)$, and contains all other permutations with one cycle of length 3 and one of length 2, such as $(1,2,3)(4,5)$ and $(3,4,5)(1,2)$. The partition $\lambda$ that determines the conjugacy class $C_\lambda$ of a permutation is also known as its \emph{cycle type}. The number of elements $|C_\lambda|$ in a conjugacy class is equal to $|S(n)|/\mathfrak z(\lambda)$, where
\[\mathfrak z(\lambda)=\prod_{i=1}^\infty i^{m_i(\lambda)} m_i(\lambda)!\]
is the size of the centralizer of $C_\lambda$, and $m_i(\lambda)$ is the number of parts in $\lambda$ that are equal to $i$.

The only permutation of cycle type $(1,1,\dots,1)$ is the identity. A permutation $\sigma$ of cycle type $(2,2,\dots,2,1,1,\dots,1)$ is an involution, that is, $\sigma\circ\sigma$ is the identity. A permutation of cycle type $(2,2,\dots,2)$ is an involution without fixed points.

Since the characters are constant over conjugacy classes, we usually denote the character of a permutation of cycle type $\mu$ in the irreducible representation corresponding to the partition $\lambda$ by $\chi^\lambda(\mu)$. In the following sections, we study some related objects that provide ways to compute the values of these functions more or less explicitly, and that relate them to other combinatorial constructions.

We recall two general facts about group representations. The first one is the decomposition of the group algebra $\C G$ of any finite group $G$. It turns out to contain each of the irreducible representations of $G$ exactly as many times as their dimensions, that is,
\[\C G=\bigoplus_i V_i^{\oplus\dim V_i},\]
where the modules $V_i$ range over the different irreducible representations of $G$. See, for example, \cite[Section 1.10]{sagan}.

The second one is \emph{Schur's lemma}: Let $V$ and $W$ be two irreducible $G$-modules. If $\phi:V\to W$ is a $G$-homomorphism (i.e., $\phi$ commutes with the action of $G$), then either $\phi$ is a $G$-isomorphism, or $\phi$ is the zero map. See, for example, \cite[Section 1.6]{sagan}.

\section{Symmetric and shifted-symmetric functions}\label{symmetricfunctions}
 For variables $x_1,x_2,\dots,x_n$, a \emph{symmetric polynomial} is a polynomial $p(x_1,x_2,\dots\,x_n)$ with coefficients in the rational numbers $\Q$ such that $p$ is invariant by any reordering of the variables:
\[p(x_1,x_2,\dots,x_n)=p(x_{\sigma(1)},x_{\sigma(2)},\dots,x_{\sigma(n)})\]
for all $\sigma\in S(n)$. Denote by $P_n$ the set of all symmetric polynomials in $n$ variables.

Let $\psi:P_{n+1}\to P_n$ be the specialization to $n$ variables, given by
\[\psi(p)(x_1,x_2,\dots,x_n)=p(x_1,x_2,\dots,x_n,0).\]
Using the identification provided by $\psi$, we can form the inverse limit
\[\Lambda=\underleftarrow \lim \,P_n\]
whose elements are known as \emph{symmetric functions}, and consist of formal sums of monomials in countably many variables $x_1,x_2,\dots$ of the form
\[a x_{i_1}x_{i_2}\cdots x_{i_k}\]
for $a\in \Q$, $i_1,\dots,i_k\in \mathbb N$, where the $i_j$'s may be repeated. By construction, symmetric functions are invariant by any finite permutation. They form a ring, denoted $\Lambda$, that is obviously also a $\Q$-module.

A basis of the symmetric functions $\Lambda$ is given by the \emph{symmetric power functions} $p_\mu$. To define them, we let
\[p_k(x_1,x_2,\dots)=x_1^k+x_2^k+\cdots,\quad k=1,2,\dots,\]
and then extend this definition to partitions $\mu=(\mu_1,\mu_2,\dots,\mu_k)$ by letting
\[p_\mu=p_{\mu_1}p_{\mu_2}\cdots p_{\mu_k}.\]
In other words, $\Lambda$ is algebraically generated by the functions $p_k$ \cite{macdonald}.

While the symmetric power functions are by far the simplest linear basis of the symmetric functions $\Lambda$, there are a few other basis that are of interest. Among these is the one composed of the so-called \emph{Schur functions} $s_1,s_2,\dots$, which can be defined in many equivalent ways. One way is as the sum \cite{macdonald}
\begin{equation}\label{powertoschur}
s_\lambda=\sum_\rho \frac{\chi^\lambda(\rho)}{\mathfrak z(\rho)} p_\rho,
\end{equation}
where the sum is over all partitions $\rho=(\rho_1,\rho_2,\dots,\rho_k)$, $\mathfrak z(\rho)=\prod_i i^{m_i(\rho)}m_i(\rho)!$, $m_i(\rho)$ is the number of parts $\rho_j$ of $\rho$ that are equal to $i$, and $\chi^\lambda(\rho)$ stands for the character of the representation of the symmetric group corresponding to the partition $\lambda$ evaluated at an element of cycle type $\rho$.

A second way to define the Schur functions is by their characterization as being the symmetric functions that specialize (by setting the variables $x_{n+1},x_{n+2},\dots$ to zero) to the \emph{Schur polynomials}, defined by quotients of determinants:
\begin{equation}
s_\lambda(x_1,x_2,\dots,x_n,0,0,\dots)=\frac{\det(x_i^{\lambda_j+n-j})_{i,j=1,\dots,n}}{\det(x_i^{n-j})_{i,j=1,\dots,n}}.
\end{equation}
This form of the definition, together with equation \eqref{powertoschur} gives a way to compute the characters $\chi^\lambda(\rho)$. Other ways are described in the following sections.

\section{The hook formula}\label{sec:hookformula}
The formula obtained in 1954 by J. Frame, G. Robinson, and R. Thrall \cite{framerobinsonthrall} to calculate the characters $\chi^\lambda(1,1,\dots,1)$, has become known as the \emph{hook formula} due to the combinatorial objects it involves.

In order to introduce these concepts, we first remark that to a partition $\lambda=(\lambda_1,\lambda_2,\dots,\lambda_k)$, where $\lambda_1\geq\lambda_2\geq\cdots\geq\lambda_k$, we can assign a \emph{Young diagram} consisting of rows of $\lambda_i$ square boxes. For example, if our partition is $(5,4,4,2)$, the diagram looks like this:
\[\yng(5,4,4,2)\]
In this diagram, each square is called a \emph{cell}. To each cell, we can assign several quantities. For our purposes, the most important one is the \emph{hook length}, which we now define. For a given cell, the \emph{hook} consists of the cell itself, together with all the cells to the right of it and all  the cells below it. So, for example, in this diagram we have marked with bullets the hook of cell $(2,2)$:
\[\young(\hfil\hfil\hfil\hfil\hfil,\hfil\bullet\bullet\bullet,\hfil\bullet\hfil\hfil,\hfil\bullet)\]
The \emph{hook length} is the number of cells in the hook. In the example of the last diagram, the hook length is 5. In the following diagram we have filled in each cell with the corresponding hook length:
\[\young(87541,6532,5421,21)\]

Now let $\lambda$ be a partition of $n$, and let $h(\lambda)$ be the product of the hook lengths of $\lambda$. For the example above, $\lambda=(5,4,4,2)$, $n=15$, and $h(\lambda)=16\,128\,000$. Then the hook formula, discovered by Frame, Robinson, and Thrall, says that
\begin{equation}\label{hookformula}
\dim\lambda=\chi^\lambda(1,1,\dots,1)=\frac{n!}{h(\lambda)}.
\end{equation}
Here, $(1,1,\dots,1)$ stands for the partition with $n$ parts equal to 1.
There is an alternative way to express this as a function of the coordinates of the partition $\lambda$:
\begin{equation}\label{eq:hookalternative}
\dim\lambda=\frac{n!\prod_{i<j}(\lambda_i-\lambda_j-i+j)} {\prod_i(\lambda_i+\ell(\lambda)-i)!}.
\end{equation}
In the example above, formula \eqref{hookformula} gives
\[\dim(5,4,4,2)=\chi^{(5,4,4,2)}(1,1,\dots,1)=\frac{15!}{h((5,4,4,2))} =81\,081.\]
A proof of the hook formula can be found in the original paper \cite{framerobinsonthrall}, and also in \cite[Examples I.1.1 and I.7.6]{macdonald}, \cite[Section 3.10]{sagan}, and \cite[Corollary 7.21.6]{stanley}. 
\section{The Murnaghan-Nakayama rule}\label{murnaghannakayamarule}
The Murnaghan-Nakayama rule gives a combinatorial algorithm to compute the value of the character $\chi^\lambda(\mu)$ of an irreducible representation of the symmetric group $S(n)$ corresponding to the partition $\lambda$ and evaluated at an element of cycle type $\mu$.

In order to state the rule, we need to define some combinatorial objects. Let $\lambda$ and $\mu$ be two partitions, and assume that $\mu_i\leq\lambda_i$, $i=1,2,\dots$ A \emph{skew Young diagram} $\lambda/\mu$, consists of the cells of the Young diagram of $\lambda$ that would not belong to the Young diagram of $\mu$ if we were to overlap them. For example, let $\mu=(3,1)$ and $\lambda=(5,4,4,2)$. Then
\[\lambda/\mu=\young(:::\hfil\hfil,:\hfil\hfil\hfil,\hfil\hfil\hfil\hfil,\hfil\hfil).\]
A special kind of Young diagram that we will be interested in is the \emph{strip}, which consists of a connected skew Young diagram that does not contain a $2\times2$ block:
\[\yng(2,2)\] 
A skew Young diagram is connected if we can get from any cell to any other cell jumping on \emph{adjacent} cells. For example, the following are strips:
\begin{equation}\label{strips}
\young(::\hfil,::\hfil,::\hfil,:\hfil\hfil,:\hfil,\hfil\hfil) \qquad\young(\hfil\hfil\hfil\hfil,\hfil,\hfil,\hfil)\qquad \young(\hfil,\hfil,\hfil,\hfil,\hfil)
\end{equation}
The following are \emph{not} strips:
\[\young(::\hfil,::\hfil,::\hfil,:\hfil,:\hfil,\hfil\hfil)\qquad \young(\hfil,\hfil\hfil,:\hfil,\hfil\hfil)\qquad\young(:::::\hfil,::\hfil\hfil\hfil\hfil,:\hfil\hfil,\hfil\hfil\hfil,\hfil)\]
Here, the first example is disconnected (it is not enough to have a common corner for cells to be adjacent), the second example cannot be obtained as a difference of Young diagrams, and the third contains a $2\times2$ block.

Now suppose that the partitions $\mu$ and $\lambda$ are both partitions of the same number, that is, $\sum \lambda_i=\sum \mu_i$. A \emph{$\mu$-strip decomposition} of the partition $\lambda$ is a way to recover $\lambda$ by successively adjoining strips of size $\mu_1,\mu_2,\dots$. More precisely, it is a sequence of partitions $\eta^1,\eta^2,\dots,\eta^k$, where $k$ is the number of parts in $\mu$, such that $\eta^1$ is a strip of size $\mu_1$, $\eta^i/\eta^{i-1}$ is a strip of length $\mu_i$, $i=2,\dots, k$, and $\eta^k=\lambda$. For example, if $\lambda=(5,4,4,2)$ and $\mu=(5,3,2,2,2,1)$, a $\mu$-strip decomposition of $\lambda$ can be represented as follows: we will use number 1 to represent the strip of size 5, then number 2 to represent the strip of size 3, and so on, always assigning the number $i$ to the strip of length $\mu_i$:
\begin{equation}\label{decompositions}
\young(11136,1223,1255,44)\qquad\young(11233,1225,1445,16)\qquad\young(11244,1225,1335,16)
\end{equation}
Note that the second and third examples are different only in the order of the third and fourth strips.

The \emph{height} of a strip $\lambda/\mu$ is equal to the \emph{vertical} distance between the center of a cell in the lowest row of the strip to the center of a cell in the highest row of the strip, assuming that each cell has side length 1. For example, the strips in \eqref{strips} have heights 5, 3, and 4, respectively. The \emph{height of a $\mu$-strip decomposition} equals the sum of the height of the involved strips. For example, the height of the first decompositions in \eqref{decompositions} is $2+1+1+0+0+0=4$, and the height of the other two is $3+1+0+0+1+0=5$.

The \emph{Murnaghan Nakayama rule} is the simple assertion that
\begin{equation}\label{murnakruleeq}
\chi^\lambda(\mu)=\sum_{\{\eta_i\}} (-1)^{\height \{\eta^i\}},
\end{equation}
where the sum is taken over all the $\mu$-strip decompositions $\{\eta^i\}$ of $\lambda$. For example, if $\lambda=(4,3,2)$ and $\mu=(6,2,1)$, then there are exactly two $\mu$-strip decompositions of $\lambda$, namely (in the same notations as above),
\[\young(1111,123,12)\qquad\textrm{and}\qquad\young(1111,122,13)\]
The heights of these decompositions are $2+1+0=3$ and $2+0+0=2$, respectively, so the Murnaghan-Nakayama rule implies that
\[\chi^{(4,3,2)}((6,2,1))=(-1)^3+(-1)^2=0.\]

Proofs of the Murnaghan-Nakayama rule can be found in \cite[Exercises 3.11 and 7.5]{macdonald}, \cite[Section 4.10]{sagan}, or \cite[Theorem 7.17.3]{stanley}.

\section{The Littlewood-Richardson rule}\label{littlewoodrichardson}
Recall that the Schur functions $s_\lambda$ constitute a linear basis of the space of symmetric functions. The \emph{Littlewood-Richardson numbers} $c^\lambda_{\mu\nu}$ give the linear decomposition of products of Schur functions in terms of that basis:
\[s_\mu s_\nu=\sum_\lambda c^\lambda_{\mu\nu}s_\lambda.\]
The \emph{Littlewood-Richardson rule} gives a way to compute these numbers in terms of certain combinatorial objects.

Good references for this topic are \cite[Section I.9]{macdonald}, \cite[Section A1.3]{stanley}.

\section{The limit shape in the uniform distribution}\label{limitshapeuniform}
Let $\Young_n$ be the set of Young diagrams corresponding to all partitions of $n$, rescaled by $1/\sqrt n$. This means that we construct the diagrams with squares of side $1/\sqrt n$.

It is a very interesting phenomenon that, for large $n$, most of the diagrams in $\Young_n$ are very similar. In fact, once rescaled, most of them are very close to the curve
\[e^{-\frac\pi{\sqrt6}x}+e^{-\frac\pi{\sqrt6}y}=1.\]
Here is a simulation:

\begin{center}
\includegraphics{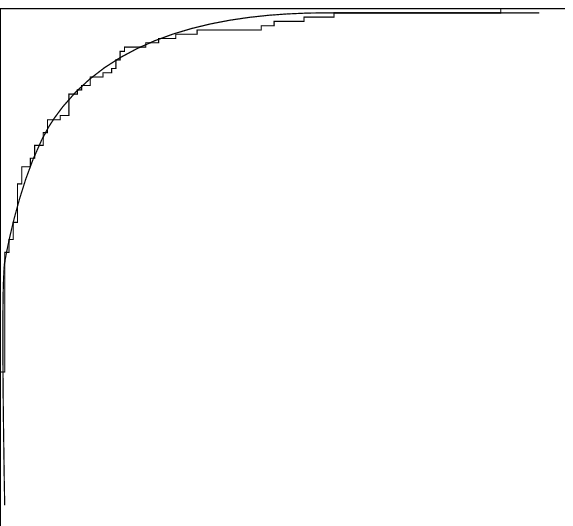}
\end{center}

In the picture we see the graph of the curve corresponding to the limit shape, together with the contour of a random partition of very large size $n$ rescaled by $1/\sqrt n$. The $x$-axis is horizontal and increases to the right; the $y$-axis is vertical and increases downwards.

This phenomenon was first remarked by A. Vershik \cite{vershik96}, who mentioned it informally to the Szalay and Turan after they had published a paper \cite{szalayturan} from where it can be deduced immediately. A more direct argument can be found in \cite{okounkovlimitshapes}. This result also follows from the Large Deviation Principle proved by A. Dembo, A. Vershik and O. Zeitouni \cite{dembovershikzeitouni}. 
\section{The branching rule and Frobenius reciprocity}\label{branchingrule}
The restriction $\res_{S(m)} V$ of a representation $\rho:S(n)\to GL(V)$ of $S(n)$ to a smaller symmetric group $S(m)$, $m<n$, induces a restriction $\res_{S(m)}\chi$ of the corresponding character $\chi$. If $\rho$ is irreducible and corresponds, say, to the partition $\lambda$ of $n$, $\res_{S(m)} \rho$ is, in general, reducible. The \emph{branching rule} gives its decomposition into irreducible representations. It states that, if $n=m+1$ and $V^\lambda$ is the irreducible representation corresponding to the partition $\lambda$, then
\[\res_{S(m)}V^\lambda=\bigoplus_{\lambda = \mu\cup\square}V^{\mu}, \]
where the direct sum is taken over all partitions $\mu$ of $m$ from which $\lambda$ can be obtained by adjoining one cell $\square$.
For example,
\[
\res_{S(11)}V^{(4,3,3,2)}=V^{(4,3,3,1)}\oplus V^{(4,3,2,2)} \oplus V^{(3,3,3,2)}.
\]
Note that, iterating, this gives a complete description of the decomposition corresponding to any pair $m<n$.
For a proof, see for example \cite[Section 2.8]{sagan} or \cite[Section 7.3]{fulton}.

A representation $\rho:S(m):\to GL(V)$ induces a representation $\ind_{S(n)} V$ of $S(n)$.
There is a version of the branching rule in this case:
\[\ind_{S(n)}V^\lambda=\bigoplus_{\lambda \cup \square= \mu}V^{\mu}, \]
where the sum is taken over all partitions $\mu$ that can be formed by adjoining a cell $\square$ to $\lambda$. Let $\chi$ be the character of $\rho$ and denote by $\ind_{S(n)} \chi$ the character of the induced representation. There is a way to compute $\ind_{S(n)} \chi$ using Frobenius reciprocity, described next.

The space of characters on $S(n)$ has an inner product
\[\langle \chi_1,\chi_2\rangle_{S(n)}=\frac 1{n!} \sum_{\sigma\in S(n)} \chi_1(\sigma)\chi_2(\sigma).\]
The characters of irreducible representations turn out to be orthonormal with respect to this inner product.
The relation
\[\langle \ind_{S(n)}\chi_1,\chi_2\rangle_{S(n)}=\langle \chi_1,\res_{S(m)}\chi_2\rangle_{S(m)},\]
which is true for all characters $\chi^1$ and $\chi^2$ of representations of $S(m)$ and $S(n)$, $m<n$, respectively, is known as \emph{Frobenius reciprocity}. For a proof see \cite[Section 3.3]{fultonharris}, \cite[Theorem 1.12.6]{sagan}.

\section{The half-infinite wedge Fermionic Fock space}\label{fock}
In this section we describe an algebraic framework to work with partitions and their connections to the symmetric group. Other good sources for this material are \cite[Chapter 14]{kac}, \cite{miwajimbodate}, and \cite[Section 2]{GWH1}.

Let $\lambda=(\lambda_1,\lambda_2,\dots,\lambda_k)$ be a partition. For convenience, we identify this vector representation with an infinite sequence by extending it with zeros: \[\lambda=(\lambda_1,\lambda_2,\dots,\lambda_k,0,0,\dots),\]
so that $\lambda_i=0$ for $i>k$. Now let $\xi_i=\lambda_i-i+\frac12$ be the \emph{modified Frobenius coordinates} of $\lambda$. Obviously, for $i>k$, $\xi_i=-i+\frac12$.

Let $\dots,\underline {\frac52},\underline{\frac32},\underline{\frac12},\underline {-\frac12},\underline{-\frac32},\underline{-\frac52},\dots$ be a bidirectionally-infinite sequence of vectors that form a basis of a countably-infinite dimensional vector space $V$. The \emph{0-charge sector of Fermionic Fock space} $\Lambda_0^{\frac\infty2}V$ (also known as the \emph{half-infinite wedge} $\Lambda_0^{\frac{\infty}2}$) is a vector space generated by the basis of vectors $v_\lambda$ indexed by partitions $\lambda=(\lambda_1,\lambda_2,\dots)$. These are defined by:
\[v_\lambda=\underline{\xi_1}\wedge\underline{\xi_2}\wedge \underline{\xi_3}\wedge\cdots
=\underline{\lambda_1-\tfrac12}\wedge\underline{\lambda_2-\tfrac32} \wedge\underline{\lambda_3-\tfrac52}\wedge\cdots\wedge \underline{\lambda_i-i+\tfrac12}\wedge\cdots.\]
Here, $\lambda_1\geq\lambda_2\geq\lambda_3\geq\cdots\geq 0$ and the $\lambda_i$ eventually become zero. The symbol $\wedge$ indicates that the terms are anti-commutative For example, for $\lambda=(5,4,4,2)$, \[v_\lambda=\underline{\tfrac92}\wedge\underline{\tfrac52} \wedge\underline{\tfrac32}\wedge\underline{-\tfrac32}\wedge \underline{-\tfrac92}\wedge
\underline{-\tfrac{11}2}\wedge\underline{-\tfrac{13}2}\wedge\cdots.\]

The relation between $v_\lambda$ and $\lambda$ can be best visualized as follows. We rotate the partition $135^\circ$ and rescale by $\sqrt 2$, so that its contour is determined by segments going up or down one unit.  We place a white pebble on the horizontal axis every time the upper contour goes down, and we place a black pebble when it goes up. The pebbles are placed on the half integer numbers $\Z+\frac12$. For $\lambda = (5,4,4,2)$, we get the following picture.
\begin{center}
\includegraphics{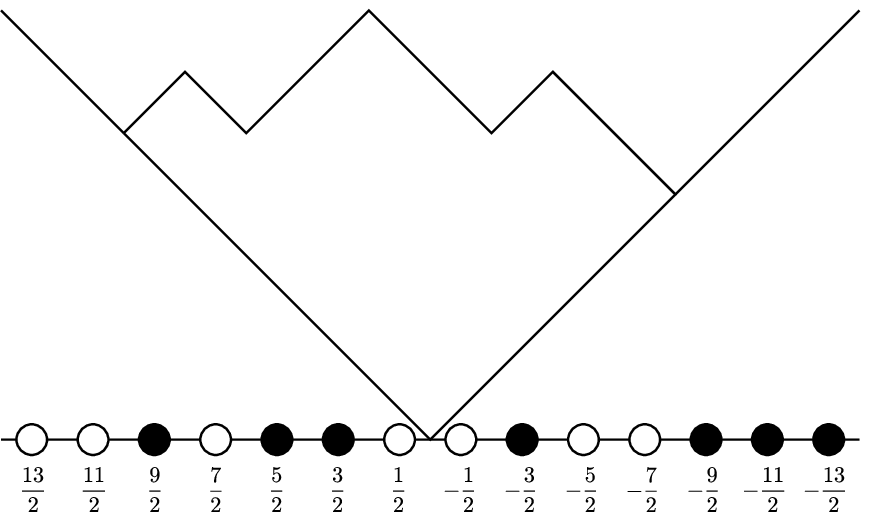}
\end{center}
If we think of the positive $x$-axis as being on the left, the positions of the black pebbles coincide exactly with the $\xi_i$'s. Hence, a chain of black and white pebbles like this is in one-to-one correspondence both with partitions $\lambda$ and with basis elements $v_\lambda$ of $\Lambda_0^{\frac\infty 2}V$.

The space $\Lambda_0^{\frac{\infty}2}V$ is contained inside \emph{fermionic Fock space} $\Lambda^{\frac{\infty}2}V$, the space generated by \emph{all} wedge products of the form
\begin{equation}\label{eq:generator}
v_{\{s_i\}}=\underline{ s_1}\wedge \underline{ s_2}\wedge \underline{ s_3}\wedge\cdots,
\end{equation}
where $s_i$ is a half-integer contained in $\Z+\frac12$, $s_1\geq s_2\geq s_3\geq\cdots$, and $s_i+1=s_{i+1}$ for $i$ large enough. The 0-charge sector is simply the subspace whose generators can be identified with partitions. Its basis of vectors of the form \eqref{eq:generator} is characterized by containing only those generators for which there are as many positive numbers $s_i$ as there are negative numbers $s_i$ missing from the wedge.

Let $k$ be a number in $\Z+\frac12$ and let $\psi_k$ be the \emph{creation operator} in $\Lambda^{\frac\infty2}V$ defined by
\[\psi_k(v)=\underline k\wedge v.\]
Note that $\underline k$ is added at the very beginning of the sequence, so the anti-commutativity may produce a sign (when comparing with the elements of the basis). For example,
\begin{align*}
\psi_{\frac32}(v_{(3,1)})&=\underline{\tfrac32}\wedge \left(\underline{\tfrac52}\wedge\underline{-\tfrac12}\wedge \underline{-\tfrac52}\wedge\underline{-\tfrac72}\wedge\cdots \right)\\
 &=- \left(\underline{\tfrac52}\wedge\underline{\tfrac32}\wedge\underline{-\tfrac12}
 \wedge\underline{-\tfrac52}\wedge\underline{-\tfrac72}\wedge\cdots\right)
\end{align*}
The anti-commutativity also implies that if the factor $\underline k$ is there already, the result is zero. For instance, $\psi_{\frac52}(v_{(3,1)})=0$.
We define an inner product in $\Lambda^{\frac\infty2}V$,
\[\left\langle v_{\{s_i\}},v_{\{t_i\}}\right\rangle=\delta_{\{s_i\},\{t_i\}},\]
where $\delta_{\{s_i\},\{t_i\}}$ is 1 if both sequences $s_i$ and $t_i$ are equal, and 0 otherwise. Then $\psi_k$ has an adjoint operator $\psi_k^*$, known as the \emph{destruction operator}. The idea is that, if the spot $k$ is empty (i.e., occupied by a white pebble), $\psi_k$ adds an electron (represented by a black pebble) there, and $\psi^*_k$ removes it. These operators satisfy the anti-commutation relations
\[\psi_i\psi_j^*+\psi_j^*\psi_i=\delta_{ij},\]
\[\psi_i\psi_j+\psi_j\psi_i=0=\psi_i^*\psi_j^*+\psi_j^*\psi_i^*.\]

Define the \emph{normally ordered product} by
\[:\psi_i\psi_j^*:=\left\{\begin{array}{ll}\psi_i\psi^*_j,&j>0,\\-\psi^*_j\psi_i,&j<0.\end{array}\right.\]
If $i>j$, the net effect of applying $:\psi_i\psi_j^*:$ to a vector $v_\lambda$, where $\lambda$ is some partition, is to add a strip to $\lambda$ if possible, and to add a sign according to the parity of the height of the strip (this should be reminiscent of equation \eqref{murnakruleeq}). For example,
\[:\psi_{\frac{11}2}\psi^*_{-\frac32}:v_{(5,4,4,2)}=-v_{(6,6,5,5)}.\]
This is illustrated in the following picture, which should be compared to the diagram above:
\begin{center}
\includegraphics{addstrip}
\end{center}
The height of the strip $(6,6,5,5)/(5,4,4,2)$ added in this procedure is 3, so the sign is indeed $(-1)^3=-1$:
 \[:\psi_{\frac{11}2}\psi^*_{-\frac32}:v_{(5,4,4,2)}=-v_{(6,6,5,5)}.\]
 If the addition of the strip is not possible, the result is zero.
 If $i<j$, the effect is precisely the opposite: a strip is removed and again a sign is added according to the parity of the height of the strip; if the strip cannot be removed, the result is zero.

Let $n>0$. We want to define operators $\alpha_{-n}$ and $\alpha_n$ that will, loosely speaking, respectively increase and decrease the energy of an electron by $n$ in all possible ways or, equivalently, they will respectively add and remove all possible strips of length $n$. We let, for $n$ in $\Z$,
\[\alpha_{-n}=\sum_{k\in \Z+\frac12}:\psi_{k+n}\psi_k^*:\]
For example,
\[\alpha_{-3}\,v_{(3,1)}=v_{(6,1)}-v_{(3,2,2)}+v_{(3,1,1,1,1)},\]
which graphically corresponds to:
\[\alpha_{-3}\left(\yng(3,1)\right)=\young(\hfil\hfil\hfil\bullet\bullet\bullet,\hfil) -\young(\hfil\hfil\hfil,\hfil\bullet,\bullet\bullet) +\young(\hfil\hfil\hfil,\hfil,\bullet,\bullet,\bullet).\]
The bullets $\bullet$ mark the cells corresponding to the added strips. Note that the signs correspond to the parities of the heights of these strips. Similarly,
\[\alpha_3\,v_{(5,4,3)}=v_{(5,4)}-v_{(3,3,3)}-v_{(5,2,2)},\]
or
\[\alpha_3\left(\yng(5,4,3)\right)=\yng(5,4)-\yng(3,3,3)-\yng(5,2,2).\]
Strips of the forms $\yng(3)$ and $\yng(2,1)$ have been removed, and again the signs correspond to the parity of their heights.

 Let $\emptyset$ stand for the empty partition. The vector $v_\emptyset=\underline{-\frac12}\wedge\underline{-\frac32}\wedge\underline{-\frac52}\wedge\cdots$ is known as the \emph{ground state}.

 Let $\mu$ be a partition. The following is an immediate consequence of the Murnag\-han-Nakayama rule, equation \eqref{murnakruleeq}:
 \[\alpha_{-\mu_1}\alpha_{-\mu_2}\cdots\alpha_{-\mu_k}v_\emptyset=\sum_\lambda \chi^\lambda(\mu)\,v_\lambda.\]
 The sum is over all partitions $\lambda$, but of course the character $\chi^\lambda(\mu)$ equals zero for all but finitely many of them. Since the characters $\chi^\lambda$ are linearly independent from each other, the vectors $\prod_i \alpha_{-\mu_i}v_\emptyset$ form a basis of the entire space $\Lambda^{\frac\infty2}_0V$.

 Note that the adjoint $\alpha^*_n$ of $\alpha_n$ satisfies
\[\alpha_n^*=\alpha_{-n},\]
and we have the commutation relations
\[[\alpha_k,\alpha_r]=k\,\delta_{-k,r}.\]

\section{The Jacobi theta function}\label{jacobithetas}
The \emph{Jacobi theta function}
\[\vartheta(x,q)=(q^{1/2}-q^{-1/2})\prod_{i=1}^\infty\frac{(1-q^ix)(1-q^i/x)}{(1-q^i)^2},\]
defined for arbitrary $x$ in $\C$ and $q$ in the unit disc $\{|q|<1\}\subset\C$, is important in our work. In this section we recall some of its properties.

The first property is its modularity. Let $\tau$ and $z$ be defined (modulo $\Z\subset\C$) by $q=e^{2\pi i \tau}$ and $x=e^{2\pi i z}$. It is traditional to abuse notation and write $\vartheta(z,\tau)$ for $\vartheta(x,q)$, and we will embrace this tradition here. As a consequence of the Poisson summation formula, $\vartheta$ satisfies
\[\vartheta\!\left(\frac z\tau,-\frac1\tau\right)=-\frac{e^{\pi i z^2/\tau}}{i\tau}\,\vartheta(z,\tau).\]

The next property is its additive expression, which can be derived using the \emph{Jacobi triple product}, which is the identity
\[\prod_{m=1}^\infty \left(1-x^{2m}\right)\left(1+x^{2m-1}y^2\right)\left(1+x^{2m-1}y^{-2}\right)=\sum_{n\in \Z}x^{n^2}y^{2n},\]
valid for all complex numbers $x$ and $y$ with $|x|<1$ and $y\neq0$. It follows that
\[\vartheta(x,q)=\eta^{-3}(q)\sum_{n\in \Z}(-1)^nq^{\frac{\left(n+\frac12\right)^2}{2}}x^{n+\frac12},\]
where $\eta$ stands for the \emph{Dedekind eta function},
\[\eta(q)=q^{1/24}\prod_{n=1}^\infty(1-q^n).\]

A second property we want to record for later reference is the expression of our $\vartheta$ in terms of the more classical Jacobi thetas, namely, \cite[page 17]{mumford}
\begin{align*}
\vartheta_{00}(z,\tau)&=\sum_{n\in\Z}\exp(\pi i n^2\tau+2\pi i n z) \\
\vartheta_{01}(z,\tau)&=\sum_{n\in\Z}\exp\!\left(\pi i n^2\tau+2\pi i n\left(z+\tfrac12\right)\right)\\
 &=\vartheta_{00}\!\left(z+\tfrac12,\tau\right)\\
 \vartheta_{10}(z,\tau)&=\sum_{n\in\Z}\exp\left(\pi i \left(n+\tfrac12\right)^2\tau+2\pi i \left(n+\tfrac12\right)z\right)\\
 &=\exp\!\left(\frac{\pi i \tau}4+\pi iz\right)\vartheta_{00}\!\left(z+\tfrac12 \tau,\tau\right)\\
 \vartheta_{11}(z,\tau)&=\sum_{n\in\Z}\exp\!\left(\pi i\left(n+\tfrac12\right)^2\tau+2\pi i\left(n+\tfrac12\right)\left(z+\tfrac12\right)\right) \\
 &=\exp\!\left(\frac{\pi i \tau}4+\pi i\left(z+\tfrac12\right)\right)\vartheta_{00}\!\left(z+\tfrac12(1+\tau),\tau\right)
\end{align*}
It follows that
\[\vartheta(z,\tau)=-i\eta^{-3}(\tau)\,\vartheta_{11}(z,\tau).\]
We further record the modular behavior of the traditional Jacobi thetas \cite[page 36]{mumford}:
\begin{align*}
\vartheta_{00}\!\left(\frac z\tau,-\frac1\tau\right)&=(-i\tau)^{\frac12}\exp\!\left(\frac{\pi i z^2}\tau\right)\vartheta_{00}(z,\tau), \\
\vartheta_{01}\!\left(\frac z\tau,-\frac1\tau\right)&=(-i\tau)^{\frac12}\exp\!\left(\frac{\pi i z^2}\tau\right)\vartheta_{10}(z,\tau), \\
\vartheta_{10}\!\left(\frac z\tau,-\frac1\tau\right)&=(-i\tau)^{\frac12}\exp\!\left(\frac{\pi i z^2}\tau\right)\vartheta_{01}(z,\tau), \\
\vartheta_{11}\!\left(\frac z\tau,-\frac1\tau\right)&=-(-i\tau)^{\frac12}\exp\!\left(\frac{\pi i z^2}\tau\right)\vartheta_{11}(z,\tau).
\end{align*}
Note the additional minus sign in the rule for $\vartheta_{11}$.

\section{Quasimodular forms}\label{quasimodularforms}

We follow \cite[Section 3.3.7]{pillow}.

A \emph{quasimodular form} for a subgroup $\Gamma\subset SL_2(\Z)$ is the holomorphic part of an almost holomorphic modular form for $\Gamma$. A function of $|q|<1$ is called \emph{almost holomorphic} if it is a polynomial in $(\log|q|)^{-1}$ with coefficients that are holomorphic functions of $q$. They form a graded algebra $\mathcal{QM}(\Gamma)$. It is a theorem of M. Kaneko and D. Zagier \cite{kanekozagier} that the space of quasimodular forms is simply
\[\mathcal{QM}(\Gamma)=\Q[E_2]\otimes\mathcal M(\Gamma),\]
where $\mathcal M(\Gamma)$ denotes the space of modular forms with respect ot the subgroup $\Gamma$ and $E_2$ denotes the first of the Eisenstein series, which are in general defined by
\begin{equation*}
E_{2k}(q)=\frac{\zeta(1-2k)}2+\sum_{n=1}^\infty\left(\sum_{d|n}d^{2k-1}\right)q^n,\quad k=1,2,\dots
\end{equation*}
In particular,
\begin{equation}\label{eq:certaineisensteinseries}
E_2(q),E_2(q^2),E_4(q^2)\in\mathcal {QM}(\Gamma_0(4))
\end{equation}
where
\[\Gamma_0(4)=\left\{\bigl(\begin{smallmatrix}a&b\\c&d\end{smallmatrix}\bigr):c\equiv0\!\!\mod 4\right\}.\]
It is well known that these series \eqref{eq:certaineisensteinseries} algebraically generate all other Eisenstein series and their products, so these too are quasimodular forms. Hence, in order to show that a certain function is a quasimodular form, it  is sufficient to show that it is a polynomial in the $E_{2k}$.

\section{Ramified covers of a surface}\label{ramifiedcovers}

Let $X$ be a (possibly disconnected) compact Riemann surface. A non-singular holomorphic map $\pi:S\to X$ from a compact Riemann surface $S$ that does not collapse any connected component of $S$ to a single point in $X$ is called is a \emph{ramified cover} of $X$.

Two ramified covers $\pi:S\to  X$ and $\pi':S'\to X$ are said to be \emph{isomorphic} if there is a holomorphic bijection $\varphi:S\to S'$ such that $\pi=\pi'\circ\varphi$. The set of holomorphic bijections $S\to S$ is the \emph{group of automorphisms} $\aut(S)$ of $S$.

A ramified cover $\pi:S\to X$ is a local holomorphic diffeomorphism except at finitely many points $p_1,p_2,\dots, p_k \in X$. This means that every point $q\in X-\{p_1,\dots,p_k\}$ is contained in a disc $\Delta$ whose preimage $\pi^{-1}(\Delta)$ consists of a disjoint union of discs $D_1,D_2,\dots,D_d$ and the restrictions $\pi|_{D_\ell}:D_\ell\to \Delta$ are bijections. The number $d$ of such discs remains constant for all $q$ throughout $X-\{p_1,\dots,p_k\}$ and is known as the \emph{degree} of the cover $\pi$.

The points $p_1,p_2,\dots,p_k$ are known as \emph{ramification points} of the cover $\pi$. For each $i=1,2,\dots,k$, let $\{q_1^i,q_2^i,\dots,q_{\ell_i}^i\}$ be the set of preimages $\pi^{-1}(p_i)$. For every $p_i$, let $V_i$ be a small open set homeomorphic to a disc, containing $p_i$ and no other $p_j$. For every point $q^i_j$ there is an open set $U^i_j \ni q^i_j$ that consists of a single connected component of $\pi^{-1}(p_i)$. $U^i_j$ contains no other $q_k^\ell$, and is homeomorphic to an open disc. On $U^i_j$, it is always possible to find a complex local coordinate $z$ such that $\pi|_{U^i_j}(z)=z^{\eta_j^i}$ for some integer $\eta_j^i\geq1$. $\eta_j^i=1$ only for those points $q^i_j$ where $\pi$ is a local diffeomorphism. Then $\eta^i=(\eta^i_1,\eta^i_2,\dots)$ is a partition of $d$.

Encircling each of the $p_i$'s, we can pick a small non-self-intersecting loop $\Gamma_i\subset V_i$ whose preimages will be loops encircling each of the points $q^i_j$. Denote by $\tilde\Gamma^i_j\subset\pi^{-1}(\Gamma_i)\subset U^i_j$ the loop corresponding to the point $q^i_j$. Pick a point $a_i$ in each $\Gamma_i$. The set $\pi^{-1}(a_i)$ consists of exactly $d$ points, each of them located on one of the loops $\tilde\Gamma^i_j$.

Note that $\pi|_{\cup_j\tilde\Gamma^i_j}:\cup_j\tilde\Gamma^i_j\to\Gamma_i$ is a topological covering space, and so it has a group of deck transformations. These transformations permute the $d$ preimages of $a_i$. Their action can be intuitively understood by the following procedure: move a point around $\Gamma_i$ starting at $a_i$ until you are back in $a_i$, and look at the lift of this motion: each point in $\pi^{-1}(a_i)$ is taken around a segment of a loop until it gets to another point in $\pi^{-1}(a_i)$. Labeling the elements of $\pi^{-1}(a_i)$ with the numbers $1,2,\dots,d$, this determines a permutation $s_i$ in $S(d)$, known as the \emph{monodromy}. Note that the cycle structure of $s_i$ coincides with the partition $\eta^i$. The monodromy $\{s_1,s_2,\dots,s_k\}$ together with the additional data $(X,p_1,\dots,p_k)$, determines the cover $\pi$ up to isomorphism. The partition $\eta^i$ corresponds to the cycle type of $s_i$, and  $\eta=(\eta^1,\dots,\eta^k)$ is known as the \emph{ramification profile} of $\pi$.

Recall the \emph{Riemann-Hurwitz formula}: for an analytic mapping $\pi:S\to S'$ of degree $d$ between two curves $S$ and $S'$ ramified at $n$ points with ramification profile given by partitions $\eta^1,\dots,\eta^n$, the relation between their Euler characteristics $\chi(S)$ and $\chi(S')$ is
\begin{equation}\label{riemannhurwitz}
\chi\left(S\right)=d\cdot\chi(S')
-\sum_i\left(\left|\eta^i\right|-\ell\left(\eta^i\right)\right).
\end{equation}
The genus $g$ of a surface $S$ equals $(2-\chi(S))/2$.

\bibliography{bib}{}
\bibliographystyle{plain}
\end{document}